\documentclass[10pt]{article}
\usepackage[english]{babel}
\usepackage{fullpage}
\usepackage[T1]{fontenc}
\usepackage{amsmath}
\usepackage{amsfonts}
\usepackage{amssymb}
\usepackage{mathrsfs}
\usepackage{amsthm}
\usepackage{latexsym}
\usepackage{dsfont}
\usepackage{array}
\usepackage{mathrsfs}
\usepackage{amsxtra}
\usepackage{amscd}
\usepackage{mathtools}
\usepackage{verbatim}
\usepackage{stmaryrd}
\usepackage{enumerate}
\usepackage{frcursive}
\usepackage{hyperref}
\hypersetup{
    colorlinks=true,
    linkcolor=blue,
	linktocpage=true
}
\usepackage{xcolor}
\usepackage[title]{appendix}

\definecolor{red}{rgb}{1.0,0.0,0.0}

\definecolor{blu}{rgb}{0.0,0.0,1.0}

\definecolor{gre}{rgb}{0.03,0.50,0.03}

\definecolor{amethyst}{rgb}{0.6, 0.4, 0.8}

\definecolor{blue-violet}{rgb}{0.54, 0.17, 0.89}

\definecolor{darkviolet}{rgb}{0.58, 0.0, 0.83}

\numberwithin{equation}{section}

\theoremstyle{plain}
\newtheorem{Theorem}{Theorem}[section]
\newtheorem{Definition}[Theorem]{Definition}
\newtheorem{Proposition}[Theorem]{Proposition}
\newtheorem{Lemma}[Theorem]{Lemma}

\newtheorem{Corollary}[Theorem]{Corollary}
\newtheorem{Remark}[Theorem]{Remark}

\newtheorem{innerAssumption}{Assumption}
\newenvironment{Assumption}[1]
{\renewcommand\theinnerAssumption{#1}\innerAssumption}
{\endinnerAssumption}

\def \trans{^{\scriptscriptstyle{\intercal }}}

\def \E{\mathbb{E}}
\def \F{\mathbb{F}}
\def \G{\mathbb{G}}
\def \N{\mathbb{N}}
\def \P{\mathbb{P}}
\def \R{\mathbb{R}}

\def \Ac{{\cal A}}
\def \Bc{{\cal B}}

\def \Fc{{\cal F}}
\def \Gc{{\cal G}}

\def \Nc{{\cal N}}

\def \eps{\varepsilon}

\allowdisplaybreaks

\begin{document}

\title{Path-dependent Hamilton-Jacobi-Bellman equation:\\ Uniqueness of Crandall-Lions viscosity solutions}

\author{
Andrea COSSO\footnote{Universit\`a degli Studi di Milano, Dipartimento di Matematica, Via Saldini 50, 20133, Milano, Italy; andrea.cosso@unimi.it.} \quad
Fausto GOZZI\footnote{Luiss University, Department of Economics and Finance, Rome, Italy; fgozzi@luiss.it.} \quad
Mauro ROSESTOLATO\footnote{Universit\`a di Genova, Dipartimento di Economia, Via F. Vivaldi 5, 16126, Genova, Italy; mauro.rosestolato@unige.it.}
\quad Francesco RUSSO\footnote{ENSTA Paris, Institut Polytechnique de Paris, Unit\'e de Math\'ematiques Appliqu\'ees, 828, bd.\ des Mar\'echaux, F-91120 Palaiseau, France; francesco.russo@ensta-paris.fr.\newline\newline
\emph{Acknowledgments.} \!The research of Francesco Russo was partially supported by the  ANR-22-CE40-0015-01 project (SDAIM).}
}

\maketitle

\begin{abstract}
\noindent We formulate a path-dependent stochastic optimal control problem under general conditions, for which we prove rigorously the dynamic programming principle and that the value function is the unique Crandall-Lions viscosity solution of the corresponding Hamilton-Jacobi-Bellman equation. Compared to the literature, the proof of our core result, that is the comparison theorem, is based on the fact that the value function is bigger than any viscosity subsolution and smaller than any viscosity supersolution. It also relies on the approximation of the value function in terms of functions defined on finite-dimensional spaces as well as on regularity results for parabolic partial differential equations.
\end{abstract}

\bigskip
\noindent\textbf{Mathematics Subject Classification (2020):} 35D40, 35B51, 93E20.

\bigskip
\noindent\textbf{Keywords:} path-dependent Hamilton-Jacobi-Bellman equations, viscosity solutions, comparison theorem, functional It\^o calculus.

\newpage

\tableofcontents

\section{Introduction}

The optimal control of path-dependent stochastic differential equations (SDEs) arises frequently in applications (for instance in Economics and Finance) where the dynamics are non-Markovian.
Such non-Markovianity makes difficult to apply the dynamic programming approach to those problems. Indeed, the standard dynamic programming approach is designed when the state equation is Markovian hence it cannot be applied to such problems as it is.
\\
More precisely, consider the following SDE on a complete probability space $(\Omega,\Fc,\P)$ where a $m$-dimensional Brownian motion $B=(B_t)_{t\geq0}$ is defined. Let $T>0$, $t\in[0,T]$, $x\in C([0,T];\R^d)$, and consider a progressively measurable process $\alpha\colon[0,T]\times \Omega \rightarrow A$ (with $A$ being a Polish space), where $x$ is the initial path and $\alpha$ the control process. Let the state process $X:[0,T]\times \Omega \rightarrow \R^d$ satisfy the following controlled path-dependent SDE:
\begin{equation}\label{eq:stateintro}
\begin{cases}
dX_s \ = \ b(s,X,\alpha_s)\,ds + \sigma(s,X,\alpha_s)\,dB_s, \qquad &\quad s\in(t,T], \\
X_s \ = \ x(s), &\quad s\in[0,t].
\end{cases}
\end{equation}
Here $X$ denotes the whole path, which, under mild assumptions, belongs to $C([0,T];\R^d)$.
We assume $b\colon[0,T]\times C([0,T];\R^d)\times A\rightarrow\R^d$ (as well as $\sigma$) to be non-anticipative, namely, for all $s\in [0,T]$, $a\in A$, $b(s,x,a)$ and $\sigma(s,x,a)$ depend on the path $x\in C([0,T];\R^d)$ only up to time $s$.
\\
The stochastic optimal control problem consists in maximizing the reward functional
\[
J(t,x,\alpha) \ = \ \E\bigg[\int_t^T f\big(s,X^{t,x,\alpha},\alpha_s\big)\,ds + g\big(X^{t,x,\alpha}\big)\bigg],
\]
with $f$ being non-anticipative, as $b$ and $\sigma$ above.
In the above formula, $X^{t,x,\alpha}$ denotes the solution to \eqref{eq:stateintro}. The value function is then defined as
\[
v(t,x) \ = \ \sup_{\alpha} J(t,x,\alpha), \qquad \forall\,(t,x)\in[0,T]\times C([0,T];\R^d),
\]
where the supremum is taken over all progressively measurable control processes $\alpha$.\\
We see that  the value function is defined on the infinite-dimensional space of continuous paths $C([0,T];\R^d)$, hence it is related to some Hamilton-Jacobi-Bellman (HJB) equation in infinite dimension.

The ``standard'' approach to study such problems consists in
changing state space transforming the path-dependent SDE into a Markovian SDE, formulated on an infinite-dimensional space $\mathcal{H}$, typically $C([0,T];\R^d)$ or $\R^d\times L^2([0,T];\R^d)$.
In this case the associated
Hamilton-Jacobi-Bellman equation is a PDE in infinite dimension (see for instance \cite{DPZ3,fabbrigozziswiech}) which contains
``standard'' Fr\'echet derivatives in the space $\mathcal{H}$.
Some results on the viscosity solution approach are given for instance in \cite{Federico08,Federico11,RosestolatoSwiech17};
however, uniqueness results seems not available up to now,
see the discussion in \cite[Section 3.14, pages 363-364]{fabbrigozziswiech}).

More recently, another approach has been developed after the seminal work of Dupire \cite{dupire}. This is based on the introduction of a different notion of ``finite-dimensional''
derivatives (known as horizontal/vertical derivatives), which allows to write
the associated HJB equation without
using the derivatives in the space $\mathcal{H}$.
We call such an equation a
path-dependent Hamilton-Jacobi-Bellman equation (see equation \eqref{HJB} below), which belongs to the more general class of path-dependent partial differential equations, that is PDEs where the
unknown depends on the paths and the involved derivatives are the Dupire horizontal and vertical derivatives. The definitions of these derivatives are recalled in Appendix \ref{App:PathwiseDeriv}. There are also other approaches, similar to that introduced by Dupire, but based on slightly different notions of derivatives, see in particular \cite{AH02,Lu07,GLP21}.
\\
The theory of path-dependent PDEs is very recent, yet there are already many papers on this subject, see for instance \cite{BouchardLoeperTan21,DGR,DGRClassical,tangzhang13,PS15,CF16,PengWang,EKTZ,etzI,etzII,flandoli_zanco13,R16,cosso_russoStrict,cosso_russoStrong-Visc,rtz1,rtz3,RenRosest,paperPathDep,CFGRT,BK18,CR19}.
\\
One stream in the literature looks at such equations using modified definitions of viscosity solution. In particular, we mention the notion of viscosity solution introduced in \cite{EKTZ} and by now well developed (see \cite{etzI,etzII,R16,rtz1,rtz3,CFGRT}), where maxima and minima are taken in expectation. We also mention the recent paper \cite{BouchardLoeperTan21}, where a notion of ``approximate'' viscosity solution is introduced, for which existence, comparison, and stability results are established under fairly general conditions.
\\
Another stream in the literature, to which this article belongs, looks at path-dependent PDEs
using the ``standard'' definition of viscosity solution adapted to the new derivatives. We call such a definition the ``Crandall-Lions'' one, recalling for instance their papers \cite{CL81,CL83}. In such a context there are only two papers, namely \cite{CR19} and \cite{Zhou}. Paper \cite{CR19} only addresses the path-dependent heat equation, however it is the first work where the main tools used in \cite{Zhou} and in the present paper were introduced, namely the use of a smooth gauge-type function and a smooth variational principle on the space of continuous paths  in order to generate maxima/minima on $[0,T]\times C([0,T];\R^d)$, therefore relying on the completeness of the underlying space in place of the missing local compactness. Concerning \cite{Zhou}, a comparison theorem for path-dependent HJB equations is provided using the approach of doubling variables. The proof of such a result turns out to be technically more involved compared to our approach, even though we impose stronger assumptions on the diffusion coefficient.

In the present paper we prove existence and uniqueness of Crandall-Lions viscosity solutions of HJB equations associated to the optimal control of path-dependent SDEs. The proof of uniqueness (or, more precisely, of the comparison theorem, from which uniqueness is derived) is built on refinements of the original approach developed in \cite{L83b} and is based on the existence of the candidate solution $v$ (the value function), which is shown to be bigger than any subsolution and smaller than any supersolution.
The latter is traditionally based on regularity results which are missing in the present context. We overcome this non-trivial technical difficulty
relying on suitable approximation procedures as well as on regularity results for parabolic partial differential equations, see Lemmas \ref{L:CylindrApprox1}-\ref{L:CylindrApprox2}-\ref{L:CylindrApprox3} and Theorem \ref{T:CylindrApprox1} of Appendix \ref{App:CylindricalApprox}. Moreover, in order to generate maxima or minima for functions on $[0,T]\times C([0,T];\R^d)$, we use the idea introduced in \cite{CR19}, and also adopted in \cite{Zhou}, of relying on a smooth variational principle of Borwein-Preiss type and on a smooth gauge-type function. Here, we exploit the smooth gauge-type function introduced in \cite{Zhou} (see Lemma \ref{L:SmoothGauge}), which turns out to be simpler than that built in \cite{CR19}.

Once the comparison theorem is proved, we deduce from our existence result (Theorem \ref{T:Existence}) that the value function $v$ is the unique Crandall-Lions viscosity solution of the path-dependent HJB equation. The existence result is based, as usual, on the dynamic programming principle, which is proved rigorously in the present paper, see Theorem \ref{T:DPP}.

The rest of the paper is organized as follows. In Section \ref{S:Formulation} we formulate the stochastic optimal control problem of path-dependent SDEs and prove the dynamic programming principle. In Section \ref{S:Visc} we introduce the notion of Crandall-Lions viscosity solution and prove that the value function $v$ solves in the viscosity sense the path-dependent Hamilton-Jacobi-Bellman equation. In Section \ref{S:Uniq} we state the smooth variational principle on $[0,T]\times C([0,T];\R^d)$ and prove the comparison theorem, from which the uniqueness result follows. In Appendix \ref{App:PathwiseDeriv} we recall the definitions of horizontal and vertical derivatives together with the functional It\^o formula. Finally, in Appendix \ref{App:CylindricalApprox} we report all the results concerning the approximation of the value function needed in the proof of the comparison theorem.

\section{Path dependent stochastic optimal control problems}
\label{S:Formulation}

\subsection{Notations and basic setting}

Let $(\Omega,\Fc,\P)$ be a complete probability space on which a $m$-dimensional Brownian motion $B=(B_t)_{t\geq0}$ is defined. Let $\F=(\Fc_t)_{t\geq0}$ denote the $\P$-completion of the filtration generated by $B$. Notice that $\F$ is right-continuous, so that it satisfies the usual conditions. Furthermore, let $T>0$ and let $A$ be a Polish space, with $\Bc(A)$ being its Borel $\sigma$-algebra. We denote by $\Ac$ the family of all $\F$-progressively measurable processes $\alpha\colon[0,T]\times\Omega\rightarrow A$. Finally, for every $p\geq1$, we denote by $\mathbf S_p(\F)$ the set of $d$-dimensional continuous $\F$-progressively measurable processes $X\colon[0,T]\times\Omega\rightarrow\R^d$ such that
\begin{equation}\label{S_p}
\|X\|_{\mathbf S_p} \ := \ \E\Big[\sup_{0\leq t\leq T}|X_t|^p\Big]^{1/p} \ < \ \infty.
\end{equation}
The \emph{state space} of the stochastic optimal control problem is the set $C([0,T];\R^d)$ of continuous $d$-dimensional paths on $[0,T]$. For every $x\in C([0,T];\R^d)$ and $t\in[0,T]$, we denote by $x(t)$ or $x_t$ the value of $x$ at time $t$ and we set $x(\cdot\wedge t):=(x(s\wedge t))_{s\in[0,T]}$ or $x_{\cdot\wedge t}:=(x(s\wedge t))_{s\in[0,T]}$. Observe that $x(t)$ (or $x_t$) is an element of $\R^d$, while $x(\cdot\wedge t)$ (or $x_{\cdot\wedge t}$) belongs to $C([0,T];\R^d)$. We endow $C([0,T];\R^d)$ with the supremum norm $\|\cdot\|_T$ (also denoted by $\|\cdot\|_\infty$) defined as
\[
\|x\|_T \ = \ \sup_{s\in[0,T]} |x(s)|, \qquad x\in C([0,T];\R^d),
\]
where $|x(s)|$ denotes the Euclidean norm of $x(s)$ in $\R^d$. We remark that $(C([0,T];\R^d),\|\cdot\|_T)$ is a Banach space and we denote by $\mathscr B$ its Borel $\sigma$-algebra. We also define, for every $t\in[0,T]$, the seminorm $\|\cdot\|_t$ as
\[
\|x\|_t \ = \ \| x_{\cdot\wedge t}\|_T, \qquad x\in C([0,T];\R^d).
\]
Finally, on $[0,T]\times C([0,T];\R^d)$ we define the pseudometric $d_\infty\colon([0,T]\times C([0,T];\R^d))^2\rightarrow[0,\infty)$ as
\[
d_\infty\big((t,x),(t',x')\big) \ := \ |t - t'| + \big\|x(\cdot\wedge t) - x'(\cdot\wedge t')\big\|_T.
\]
We refer to \cite[Section 2.1]{CR19} for more details on such a pseudometric. On $[t_0,T]\times C([0,T];\R^d)$, with $t_0\in[0,T)$, we consider the restriction of $d_\infty$, which we still denote by the same symbol.

\begin{Definition}\label{D:Modulus}
We say that $w\colon[0,\infty)\rightarrow[0,\infty)$ is a modulus of continuity if $w$ is continuous, increasing, subadditive, and $w(0)=0$.
\end{Definition}

\noindent We refer to \cite[Appendix D]{fabbrigozziswiech} for more details on the notion of modulus of continuity.

\subsection{Assumptions and state equation}

We consider the coefficients
\[
b,\;\sigma,\;f\colon[0,T]\times C([0,T];\R^d)\times A \ \longrightarrow \ \R^d,\;\R^{d\times m},\;\R, \qquad\quad g\colon C([0,T];\R^d) \ \longrightarrow \ \R,
\]
on which we impose the following assumptions.

\begin{Assumption}{\bf(A)}\label{AssA}\quad
\begin{enumerate}[\upshape (i)]
\item The maps $b,\sigma,f,g$ are continuous.
\item There exist a constant $K\geq0$ such that
\begin{align*}
|b(t,x,a)-b(t,x',a)| + |\sigma(t,x,a)-\sigma(t,x',a)| + |f(t,x,a)-f(t,x',a)| &\leq K\|x - x'\|_t, \\
|g(x)-g(x')| &\leq K\|x - x'\|_T, \\
|b(t,x,a)| + |\sigma(t,x,a)| + |f(t,x,a)| + |g(x)| &\leq K,
\end{align*}
for all $a\in A$, $(t,x),(t',x')\in[0,T]\times C([0,T];\R^d)$, with $|\sigma(t,x,a)|:=(\sum_{i,j}|\sigma_{i,j}(t,x,a)|^2)^{1/2}=(\textup{tr}(\sigma\sigma\trans)(t,x,a))^{1/2}$ denoting the Frobenius norm of $\sigma(t,x,a)$.
\end{enumerate}
\end{Assumption}

\begin{Assumption}{\bf(B)}\label{AssB}
The maps $b,\sigma,f$ are uniformly continuous in $t$, uniformly with respect to the other variables. In particular, there exists a modulus of continuity $w$ such that
\[
|b(t,x,a) - b(s,x,a)| + |\sigma(t,x,a) - \sigma(s,x,a)| + |f(t,x,a) - f(s,x,a)| \ \leq \ w(|t - s|),
\]
for all $t,s\in[0,T],x\in C([0,T];\R^d),a\in A$.
\end{Assumption}

\begin{Assumption}{\bf(C)}\label{AssC}\quad
\begin{enumerate}[\upshape (i)]
\item There exist $\bar d\in\N$ and $\bar\sigma\colon[0,T]\times\R^{d\bar d}\times A\rightarrow\R^{d\times m}$ satisfying, for all $(t,x,a)\in[0,T]\times C([0,T];\R^d)\times A$,
\[
\sigma(t,x,a) \ = \ \bar\sigma\bigg(t,\int_{[0,t]}\varphi_1(s)\,d^-x(s),\ldots,\int_{[0,t]}\varphi_{\bar d}(s)\,d^-x(s),a\bigg),
\]
for some continuously differentiable maps $\varphi_1,\ldots,\varphi_{\bar d}\colon[0,T]\rightarrow\R$, where the above deterministic forward integrals are defined as in Definition \ref{D:DetInt} with $T$ replaced by $t$ (see also Remark \ref{R:IntbyParts_sigma}).
\item There exists a constant $K\geq0$ such that
\begin{align*}
|\bar\sigma(t,y,a) - \bar\sigma(t',y',a)| \ &\leq \ K\,|y - y'|, \\
|\bar\sigma(t,y,a)| \ &\leq \ K,
\end{align*}
for all $(t,a)\in[0,T]\times A$, $y,y'\in\R^{d\bar d}$, with $|y-y'|$ denoting the Euclidean norm of $y-y'$ in $\R^{d\bar d}$.
\item For every $a\in A$, the map $\bar\sigma(\cdot,\cdot,a)$ is $C^{1,2}([0,T]\times\R^{d\bar d})$. Moreover, there exist constants $K\geq0$ and $q\geq0$ such that
\[
\big|\partial_t\bar\sigma(t,y,a)\big| + \big|\partial_y\bar\sigma(t,y,a)\big| + \big|\partial_{yy}^2\bar\sigma(t,y,a)\big| \ \leq \ K\,\big(1 + |y|\big)^q,
\]
for all $(t,y,a)\in[0,T]\times\R^{d\bar d}\times A$.
\end{enumerate}
\end{Assumption}

\begin{Remark}
Notice that Assumptions \ref{AssC}\textup{-(i)-(ii)} imply the validity of \ref{AssA}\textup{-(ii)} for the function $\sigma$ (namely, Lipschitzianity in $x$ and boundedness). As a matter of fact, boundedness is obvious, while the Lipschitz property follows from the Lipschitz property of $\bar\sigma$ and the integration by parts formula \eqref{IntbyPartsC1}.
\end{Remark}

\begin{Remark}\label{R:IntbyParts_sigma}
Since the functions $\varphi_1,\ldots,\varphi_{\bar d}$ appearing in Assumption \ref{AssC}\textup{-(i)} are continuously differentiable, we can use the integration by parts formula \eqref{IntbyParts} to rewrite the forward integrals as follows:
\begin{equation}\label{IntbyPartsC1}
\int_{[0,t]}  \varphi_i(s)\,d^- x(s) \ = \ \varphi_i(t) \,x(t) - \varphi_i(0)\,x(0) - \int_0^t x(s)\,\frac{d\varphi_i}{ds}(s)\,ds,
\end{equation}
for every $i=1,\ldots,\bar d$, where we have used that the Lebesgue-Stieltjes integral $\int_{(0,t]}x(s)d\varphi(s)$ is equal to the Lebesgue integral $\int_0^t x(s)\frac{d\varphi_i}{ds}(s)ds$.
\end{Remark}

\begin{Remark}
By the Lipschitz continuity of $b,\sigma,f$, we deduce that they satisfy the following non-anticipativity condition:
\[
b(t,x,a) \ = \ b(t,x_{\cdot\wedge t},a), \qquad \sigma(t,x,a) \ = \ \sigma(t,x_{\cdot\wedge t},a), \qquad f(t,x,a) \ = \ f(t,x_{\cdot\wedge t},a),
\]
for every $(t,x,a)\in[0,T]\times C([0,T];\R^d)\times A$.
\end{Remark}

\noindent For every $t\in[0,T]$, $\xi\in\mathbf S_2(\F)$, $\alpha\in\Ac$, the state process satisfies the following system of controlled stochastic differential equations:
\begin{equation}\label{SDE}
\begin{cases}
dX_s \ = \ b(s,X,\alpha_s)\,ds + \sigma(s,X,\alpha_s)\,dB_s, \qquad &\quad s\in(t,T], \\
X_s \ = \ \xi_s, &\quad s\in[0,t].
\end{cases}
\end{equation}

\begin{Proposition}\label{P:SDE}
Suppose that Assumption \ref{AssA} holds. Then, for every $t\in[0,T]$, $\xi\in\mathbf S_2(\Fc_t)$, $\alpha\in\Ac$, there exists a unique solution $X^{t,\xi,\alpha}\in\mathbf S_2(\F)$ to equation \eqref{SDE}. Moreover, it holds that
\begin{equation}\label{ContinuityViscProof}
\lim_{r\rightarrow t^+}\sup_{\alpha\in\Ac}\E\Big[\sup_{0\leq s\leq T}\big|X_{s\wedge r}^{t,\xi,\alpha} - \xi_{s\wedge t}\big|^2\Big] \ = \ 0.
\end{equation}
\end{Proposition}
\begin{proof}
See \cite[Proposition 2.8]{CKGPR20} for the existence and uniqueness result. Concerning \eqref{ContinuityViscProof} we refer to \cite[Remark 2.9]{CKGPR20}.
\end{proof}

\subsection{Value function}

Given $t\in[0,T]$ and $x\in C([0,T];\R^d)$, the stochastic optimal control problem consists in finding $\alpha\in\Ac$ maximizing the following functional:
\[
J(t,x,\alpha) \ = \ \E\bigg[\int_t^T f\big(s,X^{t,x,\alpha},\alpha_s\big)\,ds + g\big(X^{t,x,\alpha}\big)\bigg].
\]
Finally, the value function is defined as
\begin{equation}\label{Value}
v(t,x) \ = \ \sup_{\alpha\in\Ac} J(t,x,\alpha), \qquad \forall\,(t,x)\in[0,T]\times C([0,T];\R^d).	
\end{equation}

\begin{Proposition}\label{P:Value}
Suppose that Assumption \textup{\ref{AssA}} holds. Then, the value function $v$ is bounded, continuous on $([0,T]\times C([0,T];\R^d),d_\infty)$, and there exists a constant $c\geq0$ $($depending only on $T$ and $K$$)$ such that
\begin{equation}\label{Value_Lipschitz}
|v(t,x) - v(t',x')| \ \leq \ c\big(|t-t'|^{1/2} + \|x(\cdot\wedge t) - x'(\cdot\wedge t')\|_T\big),
\end{equation}
for all $(t,x),(t',x')\in[0,T]\times C([0,T];\R^d)$.
\end{Proposition}
\begin{proof}[\textbf{Proof.}]
The boundedness of $v$ follows directly from the boundedness of $f$ and $g$. Moreover, estimate \eqref{Value_Lipschitz} follows from \cite[Theorem 3.7]{tangzhang13} (notice however that in \cite[Lemma 3.6]{tangzhang13} the right-hand side of estimate (29) should be replaced by $c(|t-t'|^{1/2} + \|x(\cdot\wedge t) - x(\cdot\wedge t')\|_T)$, as in the right-hand side of \eqref{Value_Lipschitz} but with $x=x'$; this is indeed a consequence of the proof of that lemma when estimating the term denoted ``Part1'').
\end{proof}

\subsection{Dynamic programming principle}
\label{SubS:DPP}

In Section \ref{S:Visc}, Theorem \ref{T:Existence}, we prove that the value function $v$ is a viscosity solution of a suitable path-dependent Hamilton-Jacobi-Bellman equation. The proof of this property is standard and it is based, as usual, on the dynamic programming principle which is stated below. We prove it relying on \cite[Theorem 3.4]{CKGPR20} and on the two next technical Lemmata \ref{L:A_t} and \ref{L:DPP}. For other rigorous proofs of the dynamic programming principle in the path-dependent case we refer to \cite{ElKT1,ElKT2}.

We begin introducing some notations. For every $t\in[0,T]$, let $\F^t=(\Fc_s^t)_{s\in[0,T]}$ be the $\P$-completion of the filtration generated by $(B_{s\vee t}-B_t)_{s\in[0,T]}$. Let also $Prog(\F^t)$ denote the $\sigma$-algebra of $[t,T]\times\Omega$ of all $(\Fc_s^t)_{s\in[t,T]}$-progressive sets. Finally, let $\Ac_t$ be the subset of $\Ac$ of all $\F^t$-progressively measurable processes.

\begin{Lemma}\label{L:A_t}
Suppose that Assumption \ref{AssA} holds. Then, the value function defined by \eqref{Value} satisfies
\begin{equation}\label{Value_A_t}
v(t,x) \ = \ \sup_{\alpha\in\Ac_t} J(t,x,\alpha), \qquad \forall\,(t,x)\in[0,T]\times C([0,T];\R^d).	
\end{equation}
\end{Lemma}
\begin{proof}[\textbf{Proof.}]
Fix $(t,x)\in[0,T]\times C([0,T];\R^d)$. Since $\Ac_t\subset\Ac$, we see that $v(t,x)\geq\sup_{\alpha\in\Ac_t}J(t,x,\alpha)$. It remains to prove the reverse inequality
\begin{equation}\label{Value_A_t_Proof}
v(t,x) \ \leq \ \sup_{\alpha\in\Ac_t} J(t,x,\alpha).
\end{equation}
We split the proof of \eqref{Value_A_t_Proof} into four steps.

\vspace{1mm}

\noindent\textsc{Step I.} \emph{Additional notations.} We firstly fix some notations. Let $(\R^m)^{[0,t]}$ be the set of functions from $[0,t]$ to $\R^d$, endowed with the product $\sigma$-algebra $\Bc(\R^m)^{[0,t]}$ generated by the finite-dimensional cylindrical sets of the form: $C_{t_1,\ldots,t_n}(H)=\{y\in(\R^m)^{[0,t]}\colon(y(t_1),\ldots,y(t_n))\in H\}$, for some $t_i\in[0,t]$, $H=H_{t_1}\times\cdots\times H_{t_n}$, $H_{t_i}\in\Bc(\R^m)$. Now, consider the map $\mathbf B^t\colon\Omega\rightarrow(\R^m)^{[0,t]}$ defined as follows:
\[
\mathbf B^t\colon\omega \ \longmapsto \ (B_s(\omega))_{0\leq s\leq t}.
\]
Such a map is measurable with respect to $\Fc_t$, as a matter of fact the counterimage through $\mathbf B^t$ of a finite-dimensional cylindrical set $C_{t_1,\ldots,t_n}(H)$ clearly belongs to $\Fc_t$. In addition, the $\sigma$-algebra generated by $\mathbf B^t$ coincides with $\Gc_t:=\sigma(B_s,0\leq s\leq t)$. Notice that
\[
\Fc_t \ = \ \Gc_t\vee\Nc,
\]
where $\Nc$ is the family of $\P$-null sets.\\
Finally, let $(E^t,\mathscr E^t)$ be the measurable space given by $E^t=[t,T]\times\Omega$ and $\mathscr E^t=Prog(\F^t)$. Then, we denote by $\mathcal I^t\colon E^t\rightarrow E^t$ the identity map.

\vspace{1mm}

\noindent\textsc{Step II.} \emph{Representation of $\alpha$.} Given $\alpha\in\Ac$, let us prove that there exists a map $\boldsymbol{\mathrm a}^t\colon[t,T]\times\Omega\times(\R^m)^{[0,t]}\rightarrow A$ such that:
\begin{enumerate}[1)]
\item $\boldsymbol{\mathrm a}^t$ is measurable with respect to the product $\sigma$-algebra $Prog(\F^t)\otimes\Bc(\R^m)^{[0,t]}$;
\item the processes $\alpha_{|[t,T]}$ (denoting the restriction of $\alpha$ to $[t,T]$) and $(\boldsymbol{\mathrm a}^t(s,\cdot,\mathbf B^t))_{s\in[t,T]}$ are indistinguishable.
\end{enumerate}
In order to prove the existence of such a map $\boldsymbol{\mathrm a}^t$, we begin noticing that the following holds:
\[
\Fc_s \ = \ \Fc_t\vee\Fc_s^t \ = \ \Gc_t\vee\Fc_s^t, \qquad \forall\,s\in[t,T],
\]
where the second equality follows from the fact that $\Nc$, the family of $\P$-null sets, is contained in both $\Fc_t$ and $\Fc_s^t$. Recalling that $\alpha$ is $\F$-progressively measurable, we have that $\alpha_{|[t,T]}$ is progressively measurable with respect to the filtration
\[
\sigma\big(\Gc_t\vee\Fc_s^t\big)_{s\in[t,T]}.
\]
In other words, the map $\alpha_{|[t,T]}\colon[t,T]\times\Omega\rightarrow A$ is $Prog(\F^t)\vee(\{\emptyset,[t,T]\}\otimes\Gc_t)$-measurable, with $\{\emptyset,[t,T]\}$ denoting the trivial $\sigma$-algebra on $[t,T]$.

Now, recall the definitions of $\mathcal I^t$ and $\mathbf B^t$ from \textsc{Step I}, and let denote still by the same symbol $\mathbf B^t$ the canonical extension of $\mathbf B^t$ to $[t,T]\times\Omega$ (or, equivalently, to $E^t$), defined as $\mathbf B^t\colon[t,T]\times\Omega\rightarrow(\R^m)^{[0,t]}$ with $(t,\omega)\mapsto\mathbf B^t(\omega)$. Then, the $\sigma$-algebra generated by the pair $(\mathcal I^t,\mathbf B^t)\colon[t,T]\times\Omega\rightarrow E^t\times (\R^m)^{[0,t]}$ coincides with $Prog(\F^t)\vee(\{\emptyset,[t,T]\}\otimes\Gc_t)$. Therefore, by Doob's measurability theorem (see for instance \cite[Lemma 1.13]{Kallenber}) it follows that the restriction of $\alpha$ to $[t,T]$ can be represented as follows: $\alpha_{|[t,T]}=\boldsymbol{\mathrm a}^t(\mathcal I^t,\mathbf B^t)$, for some map $\boldsymbol{\mathrm a}^t\colon[t,T]\times\Omega\times(\R^m)^{[0,t]}\rightarrow A$ satisfying items 1)-2) above.

\vspace{1mm}

\noindent\textsc{Step III.} \emph{The stochastic process $X^{t,x,\mathbf B^t}$.} Given $\alpha\in\Ac$, let $\boldsymbol{\mathrm a}^t$ be as in \textsc{Step II}. For every $y\in(\R^m)^{[0,t]}$, let $X^{t,x,y}$ be the unique solution in $\mathbf S_2(\F)$ to the following equation:
\begin{equation}\label{SDE_y}
\begin{cases}
dX_s \ = \ b(s,X,\boldsymbol{\mathrm a}^t(s,\cdot,y))\,ds + \sigma(s,X,\boldsymbol{\mathrm a}^t(s,\cdot,y))\,dB_s, \qquad &\quad s\in(t,T], \\
X_s \ = \ x(s), &\quad s\in[0,t].
\end{cases}
\end{equation}
By the standard Picard iteration argument for the existence of a solution to equation \eqref{SDE_y}, together with Proposition 1 in \cite{StrickerYor}, we can deduce that the random field $X\colon[0,T]\times\Omega\times(\R^m)^{[0,t]}\rightarrow\R^d$ is measurable with respect to the product $\sigma$-algebra $Prog(\F^t)\otimes\Bc(\R^m)^{[0,t]}$. As a consequence, we can consider the composition of $X^{t,x,y}$ and $\mathbf B^t$, denoted $X^{t,x,\mathbf B^t}$. Using the independence of $\Gc_t=\sigma(\mathbf B^t)$ and $\Fc_T^t$, we deduce that the process $X^{t,x,\mathbf B^t}$ satisfies the following equation:
\begin{equation}\label{SDE_B^t}
\begin{cases}
dX_s \ = \ b(s,X,\boldsymbol{\mathrm a}^t(s,\cdot,\mathbf B^t))\,ds + \sigma(s,X,\boldsymbol{\mathrm a}^t(s,\cdot,\mathbf B^t))\,dB_s, \qquad &\quad s\in(t,T], \\
X_s \ = \ x(s), &\quad s\in[0,t].
\end{cases}
\end{equation}
As a matter of fact, we have
\begin{align*}
&\E\bigg[\sup_{s\in[t,T]}\bigg|X_s^{t,x,\mathbf B^t}-x(t)-\int_t^s b(r,X^{t,x,\mathbf B^t},\boldsymbol{\mathrm a}^t(r,\cdot,\mathbf B^t))dr - \int_t^s \sigma(r,X^{t,x,\mathbf B^t},\boldsymbol{\mathrm a}^t(r,\cdot,\mathbf B^t))dB_r\bigg|\bigg] \\
&= \E\bigg[\E\bigg[\sup_{s\in[t,T]}\bigg|X_s^{t,x,\mathbf B^t}-x(t)-\int_t^s b(r,X^{t,x,\mathbf B^t},\boldsymbol{\mathrm a}^t(r,\cdot,\mathbf B^t))dr \\		
&\hspace{9cm} - \int_t^s \sigma(r,X^{t,x,\mathbf B^t},\boldsymbol{\mathrm a}^t(r,\cdot,\mathbf B^t))dB_r\bigg|\bigg|\Gc_t\bigg]\bigg] \\
&= \E\bigg[\E\bigg[\sup_{s\in[t,T]}\bigg|X_s^{t,x,y}-x(t)-\int_t^s b(r,X^{t,x,y},\boldsymbol{\mathrm a}^t(r,\cdot,y))dr \\
&\hspace{9cm} - \int_t^s \sigma(r,X^{t,x,y},\boldsymbol{\mathrm a}^t(r,\cdot,y))dB_r\bigg|\bigg]_{y=\mathbf B^t}\bigg],
\end{align*}
where the last equality follows from the so-called freezing lemma, see for instance \cite[Lemma 4.1]{Baldi}. Since $X^{t,x,y}$ solves equation \eqref{SDE_y}, we have
\[
\E\bigg[\sup_{s\in[t,T]}\bigg|X_s^{t,x,y}-x(t)-\int_t^s b(r,X^{t,x,y},\boldsymbol{\mathrm a}^t(r,\cdot,y))dr - \int_t^s \sigma(r,X^{t,x,y},\boldsymbol{\mathrm a}^t(r,\cdot,y))dB_r\bigg|\bigg] \ = \ 0.
\]
Hence
\[
\E\bigg[\sup_{s\in[t,T]}\bigg|X_s^{t,x,\mathbf B^t}-x(t)-\int_t^s b(r,X^{t,x,\mathbf B^t},\boldsymbol{\mathrm a}^t(r,\cdot,\mathbf B^t))dr - \int_t^s \sigma(r,X^{t,x,\mathbf B^t},\boldsymbol{\mathrm a}^t(r,\cdot,\mathbf B^t))dB_r\bigg|\bigg] \ = \ 0.
\]
This shows that $X^{t,x,\mathbf B^t}$ solves equation \eqref{SDE_B^t}.\\
Now, recalling from \textsc{Step II} that $\alpha_{|[t,T]}$ and $(\boldsymbol{\mathrm a}^t(s,\cdot,\mathbf B^t))_{s\in[t,T]}$ are indistinguishable, and noticing that the solution to equation \eqref{SDE_proof} below depends on $\alpha$ only through its values on $[t,T]$ (namely, it depends only on $\alpha_{|[t,T]}$), we conclude that $X^{t,x,\mathbf B^t}$ solves the same equation of $X^{t,x,\alpha}$, namely
\begin{equation}\label{SDE_proof}
\begin{cases}
dX_s \ = \ b(s,X,\alpha_s)\,ds + \sigma(s,X,\alpha_s)\,dB_s, \qquad &\quad s\in(t,T], \\
X_s \ = \ x(s), &\quad s\in[0,t].
\end{cases}
\end{equation}
From pathwise uniqueness for equation \eqref{SDE_proof}, we get that $X^{t,x,\mathbf B^t}$ and $X^{t,x,\alpha}$ are also indistinguishable.

\vspace{1mm}

\noindent\textsc{Step IV.} \emph{The stochastic process $X^{t,x,\mathbf B^t}$.} Given $\alpha\in\Ac$, let $\boldsymbol{\mathrm a}^t$ be as in \textsc{Step II} and $X^{t,x,\mathbf B^t}$ as in \textsc{Step III}. Then, we have
\begin{align*}
J(t,x,\alpha) \ &= \ \E\bigg[\int_t^T f\big(s,X^{t,x,\alpha},\alpha_s\big)\,ds + g\big(X^{t,x,\alpha}\big)\bigg]	\\
&= \ \E\bigg[\int_t^T f\big(s,X^{t,x,\mathbf B^t},\boldsymbol{\mathrm a}^t(s,\cdot,\mathbf B^t)\big)\,ds + g\big(X^{t,x,\mathbf B^t}\big)\bigg].
\end{align*}
Denoting by $\boldsymbol\mu^t$ the probability distribution of $\mathbf B^t$ on $((\R^m)^{[0,t]},\Bc(\R^m)^{[0,t]})$, and recalling the independence of $\Gc_t=\sigma(\mathbf B^t)$ and $\Fc_T^t$, by Fubini's theorem we obtain
\begin{align*}
&\E\bigg[\int_t^T f\big(s,X^{t,x,\mathbf B^t},\boldsymbol{\mathrm a}^t(s,\cdot,\mathbf B^t)\big)\,ds + g\big(X^{t,x,\mathbf B^t}\big)\bigg] \\
&= \ \int_{(\R^m)^{[0,t]}} \E\bigg[\int_t^T f\big(s,X^{t,x,y},\boldsymbol{\mathrm a}^t(s,\cdot,y)\big)\,ds + g\big(X^{t,x,y}\big)\bigg] \boldsymbol\mu^t(dy).
\end{align*}
Now, fix some $a_0\in A$ and, for every $y\in(\R^m)^{[0,t]}$, denote
\[
\beta_s^y \ := \ a_0\,1_{[0,t)}(s) + \boldsymbol{\mathrm a}^t(s,\cdot,y)\,1_{[t,T]}, \qquad \forall\,s\in[0,T].
\]
Notice that $\beta^y\in\Ac_t$. Moreover, recalling that $X^{t,x,y}$ solves equation \eqref{SDE_y}, we see that it solves the same equation of $X^{t,x,\beta^y}$. Then, by pathwise uniqueness, $X^{t,x,y}$ and $X^{t,x,\beta^y}$ are indistinguishable. In conclusion, we obtain
\begin{align*}
&\int_{(\R^m)^{[0,t]}} \E\bigg[\int_t^T f\big(s,X^{t,x,y},\boldsymbol{\mathrm a}^t(s,\cdot,y)\big)\,ds + g\big(X^{t,x,y}\big)\bigg] \boldsymbol\mu^t(dy) \\
&= \ \int_{(\R^m)^{[0,t]}} \E\bigg[\int_t^T f\big(s,X^{t,x,\beta^y},\beta_s^y\big)\,ds + g\big(X^{t,x,\beta^y}\big)\bigg] \boldsymbol\mu^t(dy) \\
&= \ \int_{(\R^m)^{[0,t]}} J(t,x,\beta^y)\,\boldsymbol\mu^t(dy) \ \leq \ \int_{(\R^m)^{[0,t]}} \sup_{\gamma\in\Ac_t}J(t,x,\gamma)\,\boldsymbol\mu^t(dy) \ = \ \sup_{\gamma\in\Ac_t}J(t,x,\gamma).
\end{align*}
This proves that $J(t,x,\alpha)\leq\sup_{\gamma\in\Ac_t}J(t,x,\gamma)$, for every $\alpha\in\Ac$. Then, inequality \eqref{Value_A_t_Proof} follows from the arbitrariness of $\alpha$.
\end{proof}

\noindent Next lemma expresses in terms of $v$ the value of the optimal control problem formulated at time $t$, with \emph{random} initial condition $\xi\in\mathbf S_2(\F)$. In order to state such a lemma, we introduce the function $V\colon[0,T]\times\mathbf S_2(\F)\rightarrow\R$ defined as follows:
\begin{equation}\label{Value_V}
V(t,\xi) \ = \ \sup_{\alpha\in\Ac}\E\bigg[\int_t^T f\big(s,X^{t,\xi,\alpha},\alpha_s\big)\,dr + g\big(X^{t,\xi,\alpha}\big)\bigg],
\end{equation}
for every $t\in[0,T]$, $\xi\in\mathbf S_2(\F)$. Clearly, when $\xi\equiv x\in C([0,T];\R^d)$ we have $V(t,x)=v(t,x)$.

\begin{Lemma}\label{L:DPP}
Suppose that Assumption \ref{AssA} holds. Let $t\in[0,T]$ and $\xi\in\mathbf S_2(\F)$, then
\begin{equation}\label{Value_Random}
V(t,\xi) \ = \ \E\big[v(t,\xi)\big].
\end{equation}
\end{Lemma}
\begin{proof}[\textbf{Proof.}]
We begin noting that for $t=0$ it is clear that equality \eqref{Value_Random} holds true, as a matter fact $\Fc_0$ is the family of $\P$-null sets, therefore $\xi$ is a.s. equal to a constant and \eqref{Value_Random} follows from the fact that $V(t,x)=v(t,x)$, for every $x\in C([0,T];\R^d)$. For this reason, in the sequel we suppose that $t>0$. We split the rest of the proof into four steps.

\vspace{1mm}

\noindent\textsc{Step I.} \emph{Additional notations.} Firstly, we fix some notations. For a fixed $t\in(0,T]$, let $\G^t=(\Gc_s^t)_{s\geq0}$ be given by
\[
\Gc_s^t \ := \ \Fc_t\vee\Fc_s^t, \qquad \forall\,s\geq0.
\]
Moreover, let $\mathbf S_2(\G^t)$ (resp. $\mathbf S_2(\Fc_t)$) be the set of $d$-dimensional continuous $\G^t$-progressively measurable (resp. $\Bc([0,T])\otimes\Fc_t$-measurable) processes $X\colon[0,T]\times\Omega\rightarrow\R^d$ satisfying the integrability condition \eqref{S_p}. Notice that Proposition \ref{P:SDE} extends to the case with initial condition $\xi\in\mathbf S_2(\Fc_t)$ rather than $\xi\in\mathbf S_2(\F)$. In particular, given $\xi\in\mathbf S_2(\Fc_t)$ and $\alpha\in\Ac$, equation \eqref{SDE} admits a unique solution $X^{t,\xi,\alpha}\in\mathbf S_2(\G^t)$. Then, for $\xi\in\mathbf S_2(\Fc_t)$ we define $V(t,\xi)$ as in \eqref{Value_V}.

\vspace{1mm}

\noindent\textsc{Step II.} \emph{Preliminary remarks.} We begin noting that $X^{t,\xi,\alpha}=X^{t,\xi_{\cdot\wedge t},\alpha}$, so that it is enough to prove equality \eqref{Value_Random} with $\xi_{\cdot\wedge t}$ in place of $\xi$. More generally, we shall prove the validity of \eqref{Value_Random} in the case when $\xi\in\mathbf S_2(\Fc_t)$.

Now, recall that $v$ is Lipschitz in the variable $x$ (see Proposition \ref{P:Value}) and observe that, by the same arguments, $V$ is also Lipschitz in its second argument. Furthermore, both $v$ and $V$ are bounded. Notice also that given $\xi\in\mathbf S_2(\Fc_t)$ there exists a sequence $\{\xi_k\}_k\subset\mathbf S_2(\Fc_t)$ converging to $\xi$, with $\xi_k$ taking only a finite number of values. As a consequence, from the continuity of $v$ and $V$, it is enough to prove \eqref{Value_Random} with $\xi\in\mathbf S_2(\Fc_t)$ taking only a finite number of values. Then, from now on, let us suppose that
\begin{equation}\label{xi_discr}
\xi \ = \ \sum_{i=1}^n x_i\,1_{E_i},
\end{equation}
for some $n\in\N$, $x_i\in C([0,T];\R^d)$, $E_i\in\Fc_t$, with $\{E_i\}_{i=1,\ldots,n}$ being a partition of $\Omega$.

\vspace{1mm}

\noindent\textsc{Step III.} \emph{Proof of the inequality $V(t,\xi)\leq\E[v(t,\xi)]$.} Since $\xi\in\mathbf S_2(\F)$ takes only a finite number of values, by \cite[Lemma B.3]{CKGPR20} (here we use that $t>0$, so in particular $\Fc_t$ has the property required by \cite[Lemma B.3]{CKGPR20}, namely there exists a $\Fc_t$-measurable random variable having uniform distribution on $[0,1]$) there exists a $\Fc_t$-measurable random variable $U\colon\Omega\rightarrow\R$, having uniform distribution on $[0,1]$ and being independent of $\xi$. As a consequence, from \cite[Lemma B.2]{CKGPR20} it follows that, for every $\alpha\in\Ac$, there exists a measurable function
\[
\mathrm a\colon\big([0,T]\times\Omega\times C([0,T];\R^d)\times[0,1],Prog(\F^t)\otimes\Bc(C([0,T];\R^d))\otimes\mathcal B([0,1])\big) \longrightarrow (A,\Bc(A))
\]
such that
\[
\beta_s \ := \ \alpha_s\,1_{[0,t)}(s) + \mathrm a_s(\xi,U)\,1_{[t,T]}(s), \qquad \forall\,s\in[0,T]
\]
belongs to $\Ac$ and
\[
\Big(\xi,(\mathrm a_s(\xi,U))_{s\in[t,T]},(B_s-B_t)_{s\in[t,T]}\Big) \ \overset{\mathscr L}{=} \ \Big(\xi,(\alpha_s)_{s\in[t,T]},(B_s-B_t)_{s\in[t,T]}\Big),
\]
where $\overset{\mathscr L}{=}$ means equality in law. Then, by the same arguments as in \cite[Proposition 1.137]{fabbrigozziswiech}, we get
\[
\big(X_s^{t,\xi,\alpha},\alpha_s\big)_{s\in[0,T]} \ \overset{\mathscr L}{=} \ \big(X_s^{t,\xi,\beta},\beta_s\big)_{s\in[0,T]}.
\]
Moreover, recalling \eqref{xi_discr}, define
\[
\beta_{i,s} \ := \ \alpha_s\,1_{[0,t)}(s) + \mathrm a_s(x_i,U)\,1_{[t,T]}(s), \qquad \forall\,s\in[0,T],\,i=1,\ldots,n.
\]
Since $X^{t,\xi,\beta}$ and $X^{t,x_1,\beta_1}\,1_{E_1}+\cdots+X^{t,x_n,\beta_n}\,1_{E_n}$ solve the same equation, they are $\P$-indistinguishable. Hence
\begin{align*}
\E\bigg[\int_t^T f\big(s,X^{t,\xi,\alpha},\alpha_s\big)\,ds + g\big(X^{t,\xi,\alpha}\big)\bigg] \ &= \ \E\bigg[\int_t^T f\big(s,X^{t,\xi,\beta},\beta_s\big)\,ds + g\big(X^{t,\xi,\beta}\big)\bigg] \\
&= \ \E\bigg[\sum_{i=1}^n\bigg(\int_t^T f\big(s,X^{t,x_i,\beta_i},\beta_{i,s}\big)\,ds + g\big(X^{t,x_i,\beta_i}\big)\bigg)1_{E_i}\bigg].
\end{align*}
Recalling that both $\{X^{t,x_i,\beta_i}\}_i$ and $\{\beta_i\}_i$ are independent of $\{E_i\}_i$, we have
\begin{align*}
&\E\bigg[\sum_{i=1}^n\bigg(\int_t^T f\big(s,X^{t,x_i,\beta_i},\beta_{i,s}\big)\,ds + g\big(X^{t,x_i,\beta_i}\big)\bigg)1_{E_i}\bigg] \\
&= \ \E\bigg[\sum_{i=1}^n\E\bigg[\int_t^T f\big(s,X^{t,x_i,\beta_i},\beta_{i,s}\big)\,ds + g\big(X^{t,x_i,\beta_i}\big)\bigg]1_{E_i}\bigg] \\
&= \sum_{i=1}^n\E\bigg[\E\bigg[\int_t^T f\big(s,X^{t,x_i,\beta_i},\beta_{i,s}\big)\,ds + g\big(X^{t,x_i,\beta_i}\big)\bigg]1_{E_i}\bigg] \ \leq \ \sum_{i=1}^n\E\Big[v(t,x_i)\,1_{E_i}\Big] \ = \ \E\big[v(t,\xi)\big].
\end{align*}
Then, the inequality $V(t,\xi)\leq\E[v(t,\xi)]$ follows from the arbitrariness of $\alpha$.

\vspace{1mm}

\noindent\textsc{Step IV.} \emph{Proof of the inequality $V(t,\xi)\geq\E[v(t,\xi)]$.} Take $\xi\in\mathbf S_2(\Fc_t)$ as in \eqref{xi_discr}. Then, from equality \eqref{Value_A_t} of Lemma \ref{L:A_t}, for every $\eps>0$ and $i=1,\ldots,n$, there exists $\beta_i^\eps\in\Ac_t$ such that
\[
v(t,x_i) \ \leq \ \E\bigg[\int_t^T f\big(s,X^{t,x_i,\beta_i^\eps},\beta_{i,s}^\eps\big)\,ds + g\big(X^{t,x_i,\beta_i^\eps}\big)\bigg] + \eps.
\]
Let
\[
\beta^\eps \ := \ \sum_{i=1}^n \beta_i\,1_{E_i}.
\]
We have $\beta^\eps\in\Ac$, moreover $X^{t,\xi,\beta^\eps}$ and $X^{t,x_1,\beta_1^\eps}\,1_{E_1}+\cdots+X^{t,x_n,\beta_n^\eps}\,1_{E_n}$ solve the same equation, therefore they are $\P$-indistinguishable. Therefore (exploiting the independence of both $\{X^{t,x_i,\beta_i^\eps}\}_i$ and $\{\beta_i^\eps\}_i$ from $\{E_i\}_i$)
\begin{align*}
\E\big[v(t,\xi)\big] \ &= \ \sum_{i=1}^n \E\Big[v(t,x_i)\,1_{E_i}\Big] \\
&\leq \ \sum_{i=1}^n\E\bigg[\E\bigg[\int_t^T f\big(s,X^{t,x_i,\beta_i^\eps},\beta_{i,s}^\eps\big)\,ds + g\big(X^{t,x_i,\beta_i^\eps}\big)\bigg]1_{E_i}\bigg] + \eps \\
&= \ \E\bigg[\sum_{i=1}^n\E\bigg[\int_t^T f\big(s,X^{t,x_i,\beta_i^\eps},\beta_{i,s}^\eps\big)\,ds + g\big(X^{t,x_i,\beta_i^\eps}\big)\bigg]1_{E_i}\bigg] + \eps \\
&= \ \E\bigg[\sum_{i=1}^n\bigg(\int_t^T f\big(s,X^{t,x_i,\beta_i^\eps},\beta_{i,s}^\eps\big)\,ds + g\big(X^{t,x_i,\beta_i^\eps}\big)\bigg)1_{E_i}\bigg] + \eps \\
&= \ \E\bigg[\int_t^T f\big(s,X^{t,\xi,\beta^\eps},\beta_s^\eps\big)\,ds + g\big(X^{t,\xi,\beta^\eps}\big)\bigg] + \eps \ \leq \ V(t,\xi) + \eps.
\end{align*}
From the arbitrariness of $\eps$, the inequality $\E[v(t,\xi)]\leq V(t,\xi)$ follows.
\end{proof}

\begin{Theorem}\label{T:DPP}
Suppose that Assumption \ref{AssA} holds.
Then the value function $v$ satisfies the \textbf{dynamic programming principle}: for every $t,s\in[0,T]$, with $t\leq s$, and every $x\in C([0,T];\R^d)$ it holds that
\[
v(t,x) \ = \ \sup_{\alpha\in\Ac}\E\bigg[\int_t^s f\big(r,X^{t,x,\alpha},\alpha_r\big)\,dr + v\big(s, X^{t,x,\alpha}\big)\bigg].
\]
\end{Theorem}

\begin{proof}[\textbf{Proof.}]
This follows directly from \cite[Theorem 3.4]{CKGPR20} and Lemma \ref{L:DPP}. As a matter of fact, let $V$ be the function given by \eqref{Value_V}. From \cite[Theorem 3.4]{CKGPR20} we get the dynamic programming principle for $V$:
\[
V(t,x) \ = \ \sup_{\alpha\in\Ac}\bigg\{\E\bigg[\int_t^s f\big(r,X^{t,x,\alpha},\alpha_r\big)\,dr\bigg] + V\big(s, X^{t,x,\alpha}\big)\bigg\}.
\]
Moreover, by Lemma \ref{L:DPP} we know that
\[
V(s,X^{t,x,\alpha}) \ = \ \E\big[v\big(s,X^{t,x,\alpha}\big)\big],
\]
from which the claim follows.
\end{proof}

\section{Path Dependent HJB equations and viscosity solutions}
\label{S:Visc}

\subsection{Definition of path-dependent viscosity solutions}

In the present paper we adopt the standard definitions of pathwise (or functional) derivatives of a map $u\colon[t_0,T]\times C([0,T];\R^d)\rightarrow\R$, $t_0\in[0,T)$, as they were introduced in the seminal paper \cite{dupire}, and further developed by \cite{contfournie10,ContFournieAP} and \cite[Section 2]{CR19}. We report in Appendix \ref{App:PathwiseDeriv} a coincise presentation of these tools. Just to fix notations, we recall here that the pathwise derivatives of a map $u\colon[t_0,T]\times C([0,T];\R^d)\rightarrow\R$ are given by the horizontal derivative $\partial_t^H u\colon[t_0,T]\times C([0,T];\R^d)\rightarrow\R$ and the vertical derivatives of first and second-order $\partial_x^V u\colon[t_0,T]\times C([0,T];\R^d)\rightarrow\R^d$ and $\partial_{xx}^V u\colon[t_0,T]\times C([0,T];\R^d)\rightarrow\R^{d\times d}$. We also refer to Definition \ref{D:C^1,2} (resp. Definition \ref{D:C_pol^1,2}) for the definition of the class $C^{1,2}([t_0,T]\times C([0,T];\R^d))$ (resp. $C_{\textup{pol}}^{1,2}([t_0,T]\times C([0,T];\R^d))$).
{The reason for which we consider $C^{1,2}([t_0,T]\times C([0,T];\R^d))$ rather than simply $C^{1,2}([0,T]\times C([0,T];\R^d))$ is due to the definition of viscosity solution adopted, for more details see Remark \ref{R:DefnViscosity}.} Finally, we recall that for a map $u\in C^{1,2}([t_0,T]\times C([0,T];\R^d))$ the so-called functional It\^o's formula holds, see Theorem \ref{T:Ito}.

\vspace{3mm}

\noindent Now, consider the following second-order path-dependent partial differential equation:
\begin{equation}\label{PPDE_general}
\hspace{-5mm}\begin{cases}
\vspace{2mm}
\partial_t^H u(t,x) = F\big(t,x,u(t,x),\partial^V_x u(t,x),\partial^V_{xx} u(t,x)\big), &(t,x)\in[0,T)\times C([0,T];\R^d), \\
u(T,x) = g(x), &\,x\in C([0,T];\R^d),
\end{cases}
\end{equation}
with $F\colon[0,T]\times C([0,T];\R^d)\times\R\times\R^d\times\mathcal S(d)\rightarrow\R$, where $\mathcal S(d)$ is the set of symmetric $d\times d$ matrices.

\begin{Definition}
We say that a function $u\colon[0,T]\times C([0,T];\R^d)\rightarrow\R$ is a \textbf{classical solution} of equation \eqref{PPDE_general} if it belongs to $C_{\textup{pol}}^{1,2}([0,T]\times C([0,T];\R^d))$ and satisfies \eqref{PPDE_general}.
\end{Definition}

\begin{Definition}\label{D:Visc}
We say that an upper semicontinuous function $u\colon[0,T]\times C([0,T];\R^d)\rightarrow\R$ is a (\textbf{path-dependent}) \textbf{viscosity subsolution} of equation \eqref{PPDE_general} if:
\begin{itemize}
\item $u(T,x)\leq g(x)$, for all $x\in C([0,T];\R^d)$;
\item for any $(t,x)\in[0,T)\times C([0,T];\R^d)$ and $\varphi\in C_{\textup{pol}}^{1,2}([t,T]\times C([0,T];\R^d))$, satisfying
\[
(u-\varphi)(t,x) \ = \ \sup_{(t',x')\in[t,T]\times C([0,T];\R^d)}(u-\varphi)(t',x'),
\]
with $(u-\varphi)(t,x)=0$, we have
\begin{equation}\label{Sub}
- \partial_t^H\varphi(t,x) + F\big(t,x,u(t,x),\partial^V_x\varphi(t,x),\partial^V_{xx}\varphi(t,x)\big) \ \leq \ 0.
\end{equation}
\end{itemize}

\noindent We say that a lower semicontinuous function $u\colon[0,T]\times C([0,T];\R^d)\rightarrow\R$ is a (\textbf{path-dependent}) \textbf{viscosity supersolution} of equation \eqref{PPDE_general} if:
\begin{itemize}
\item $u(T,x)\geq g(x)$, for all $x\in C([0,T];\R^d)$;
\item for any $(t,x)\in[0,T)\times C([0,T];\R^d)$ and $\varphi\in C_{\textup{pol}}^{1,2}([t,T]\times C([0,T];\R^d))$, satisfying
\[
(u-\varphi)(t,x) \ = \ \inf_{(t',x')\in[t,T]\times C([0,T];\R^d)}(u-\varphi)(t',x'),
\]
with $(u-\varphi)(t,x)=0$, we have
\begin{equation}\label{Super}
- \partial_t^H\varphi(t,x) + F\big(t,x,u(t,x),\partial^V_x\varphi(t,x),\partial^V_{xx}\varphi(t,x)\big) \ \geq \ 0.
\end{equation}
\end{itemize}

\noindent We say that a continuous map $u\colon[0,T]\times C([0,T];\R^d)\rightarrow\R$ is a (\textbf{path-dependent}) \textbf{viscosity solution} of equation \eqref{PPDE_general} if $u$ is both a (path-dependent) viscosity subsolution and a (path-dependent) viscosity supersolution of \eqref{PPDE_general}.
\end{Definition}

\begin{Remark}\label{R:DefnViscosity}
{Differently from the standard definition of viscosity solution usually adopted in the non-path-dependent case (see for instance \cite{CIL92}), notice that in Definition \ref{D:Visc} the maxima/minima are taken on $[t,T]\times C([0,T];\R^d)$ with the right-time interval $[t,T]$ in place of $[0,T]$ (i.e. the maxima/minima are ``one-sided'').\\
In the non-path-dependent case it is known that, even in infinite dimension, our ``one-sided'' definition is equivalent to the standard ``two-sided'' one (see e.g. \cite[Lemma 3.39]{fabbrigozziswiech}).
In addition, notice that the value function (say $v$) of our stochastic control problem is a viscosity solution of the HJB equation in both senses. As a matter of fact, the DPP, which is the main tool in order to prove the viscosity properties of the value function, only involves the values of $v=v(s,y)$ for $s\geq t$.

We observe that the fact of taking the maxima/minima on the right-time interval is generally adopted in the literature on viscosity solutions of path-dependent PDEs, as for instance in \cite{EKTZ,etzI,etzII,R16,rtz1,rtz3,CFGRT}, where the notion of viscosity solution introduced involves the maxima/minima of an expectation of future $($that is on $[t,T]$$)$ values of a suitable underlying process.\\
In our case, the reason for considering $[t,T]\times C([0,T];\R^d)$ rather than $[0,T]\times C([0,T];\R^d)$ is due to the proof of the comparison Theorem \ref{T:Comparison}. In particular, it is due to the gauge-type function implemented in that proof, which is introduced in Lemma \ref{L:SmoothGauge} and denoted by $\kappa_\infty$. More precisely, given a fixed point $(t,x)\in[0,T]\times C([0,T];\R^d)$, the map $(s,y)\mapsto\kappa_\infty((s,y),(t,x))$ is smooth only for $(s,y)\in[t,T]\times C([0,T];\R^d)$. However, if we would be able to find another gauge-type function $\Psi\colon([0,T]\times C([0,T];\R^d))^2\rightarrow[0,+\infty)$ such that, for every fixed $(t,x)$, the map $(s,y)\mapsto\Psi((s,y),(t,x))$ is smooth on the entire space $[0,T]\times C([0,T];\R^d)$, then the same proof of the comparison theorem would work for the more usual notion of viscosity solution where  maxima/minima are taken on $[0,T]\times C([0,T];\R^d)$.}
\end{Remark}

\subsection{The value function solves the path-dependent HJB equation}

Now, we focus on the path-dependent Hamilton-Jacobi-Bellman equation, namely on equation \eqref{PPDE_general} with
\begin{equation}\label{F_HJB}
F(t,x,r,p,M) \ = \ -\sup_{a\in A}\bigg\{\big\langle b(t,x,a),p\big\rangle + \frac{1}{2}\text{tr}\big[(\sigma\sigma\trans)(t,x,a)M\big] + f(t,x,a)\bigg\}.
\end{equation}
Therefore, equation \eqref{PPDE_general} becomes
\begin{equation}\label{HJB}
\begin{cases}
\vspace{2mm}
\partial_t^H u(t,x) + \sup_{a\in A}\bigg\{\big\langle b(t,x,a),\partial^V_x u(t,x)\big\rangle \\
\vspace{2mm}+\,\dfrac{1}{2}\text{tr}\big[(\sigma\sigma\trans)(t,x,a)\partial^V_{xx}u(t,x)\big]+f(t,x,a)\bigg\} = 0, \qquad &(t,x)\in[0,T)\times C([0,T];\R^d), \\
u(T,x) = g(x), &\,x\in C([0,T];\R^d).
\end{cases}
\end{equation}
We now prove that the value function $v$ is a viscosity solution to equation \eqref{HJB}.

\begin{Theorem}\label{T:Existence}
Suppose that Assumptions \ref{AssA} and \ref{AssB} hold. The value function $v$, defined by \eqref{Value}, is a viscosity solution to equation \eqref{HJB}.
\end{Theorem}
\begin{proof}[\textbf{Proof.}]
Recall from Proposition \ref{P:Value} that $v$ is continuous, moreover $v(T,\cdot)\equiv g(\cdot)$. Then it remains to prove both the subsolution and the supersolution property on $[0,T)\times C([0,T];\R^d)$.

\vspace{1mm}

\noindent\emph{Subsolution property.} Let $(t,x)\in[0,T)\times C([0,T];\R^d)$ and $\varphi\in C_{\textup{pol}}^{1,2}([t,T]\times C([0,T];\R^d))$ be such that
\[
(v-\varphi)(t,x) \ = \ \sup_{(t',x')\in[t,T]\times C([0,T];\R^d)}(v-\varphi)(t',x') \ = \ 0.
\]
From Theorem \ref{T:DPP} we know that, for every $h>0$ sufficiently small,
\[
0 \ = \ \sup_{\alpha\in\Ac}
\left\{
\E\bigg[\frac{1}{h}\int_t^{t+h}
f(r,X^{t,x,\alpha},\alpha_r)\,dr + \frac{1}{h}
\big(v(t+h, X^{t,x,\alpha}) - v(t,x)\big)\bigg]\right\}.
\]
Then, there exists $\alpha^h\in\Ac$ such that
\begin{align*}
- h \ &\leq \E \bigg[\frac{1}{h}\int_t^{t+h} f(r,X^{t,x,\alpha^h },\alpha^h _r)\,dr +
\frac{1}{h}\big(v(t+h, X^{t,x,\alpha^h }) - v(t,x)\big)\bigg] \\
&\leq \ \E\bigg[\frac{1}{h}\int_t^{t+h}
f(r, X^{t,x,\alpha^h },\alpha^h _r)\,dr
+ \frac{1}{h}\big(\varphi(t+h,X^{t,x,\alpha^h }) - \varphi(t,x)\big)\bigg],
\end{align*}
where the above inequality follows from the fact that $v(t,x)=\varphi(t,x)$ and $v\leq\varphi$ on $[t,T]\times C([0,T];\R^d)$. By the functional It\^o formula \eqref{Ito_formula}, we obtain
\begin{align*}
0 \ &\leq \ h + \frac{1}{h}\int_t^{t+h} \E\big[
\partial_t^H \varphi(r,X^{t,x,\alpha^h })\big]dr + \frac{1}{h}\int_t^{t+h} \E\big[\langle b(r,X^{t,x,\alpha^h },\alpha^h _r),
\partial_x^V \varphi(r,X^{t,x,\alpha^h })\rangle\big]dr \\
&+ \frac{1}{h}\int_t^{t+h}\frac{1}{2}\E
\big[\textup{tr}\big[(\sigma\sigma\trans)(r, X^{t,x,\alpha^h },\alpha^h _r) \partial_{xx}^V
\varphi(r, X^{t,x,\alpha^h })\big]\big]dr +
\frac{1}{h}\int_t^{t+h} \E\big[f(r, X^{t,x,\alpha^h })\big]dr.
\end{align*}
Recalling that $b$, $\sigma$, $f$ are uniformly continuous in their first two arguments, uniformly with respect to $a$, using \eqref{ContinuityViscProof}, we get
\begin{align*}
&\frac{1}{h}\int_t^{t+h} \E\big[
\partial_t^H \varphi(r,X^{t,x,\alpha^h })\big]dr + \frac{1}{h}\int_t^{t+h} \E\bigg[\langle
b(r,X^{t,x,\alpha^h },\alpha^h _r),\partial_x^V
\varphi(r,X^{t,x,\alpha^h })\rangle
\\
&\quad \ + \frac{1}{2}\textup{tr}\big[(\sigma\sigma\trans)
(r,X^{t,x,\alpha^h },\alpha^h _r)
\partial_{xx}^V \varphi(r,X^{t,x,\alpha^h })\big]
+ f(r,X^{t,x,\alpha^h },\alpha^h _r)\bigg]dr \\
&= \ \partial_t^H \varphi(t,x) + \frac{1}{h}\int_t^{t+h} \E\bigg[\langle b(t,x,\alpha^h _r),\partial_x^V \varphi(t,x)\rangle \\
&\quad \ + \frac{1}{2}\textup{tr}\big[(\sigma\sigma\trans)(t,x,\alpha^h _r) \partial_{xx}^V \varphi(t,x)\big] + f(t,x,\alpha^h _r)\bigg]dr + \rho(h),
\end{align*}
where $\rho(h)\rightarrow0$ as $h\rightarrow0^+$. Then, we obtain
\begin{align*}
&0  \leq h + \rho(h) + \partial_t^H \varphi(t,x) +\frac1h \int_{t}^{t+h} \sup_{a\in A}\Big\{\langle b(t,x,a),\partial_x^V \varphi(t,x)\rangle \\
&\quad \ + \frac{1}{2}\textup{tr}\big[(\sigma\sigma\trans)(t,x,a) \partial_{xx}^V \varphi(t,x)\big] + f(t,x,a)\Big\}dr .
\end{align*}
Sending $h\rightarrow0^+$, we conclude that \eqref{Sub} holds (with $F$ given by \eqref{F_HJB}).

\vspace{1mm}

\noindent\emph{Supersolution property.} Let $(t,x)\in[0,T)\times C([0,T];\R^d)$ and $\varphi\in C_{\textup{pol}}^{1,2}([t,T]\times C([0,T];\R^d))$ be such that
\[
(v-\varphi)(t,x) \ = \ \inf_{(t',x')\in[t,T]\times C([0,T];\R^d)}(v-\varphi)(t',x') \ = \ 0.
\]
From Theorem \ref{T:DPP} we have, for every $h>0$ sufficiently small, and for every constant control strategy $\alpha\equiv a\in A$,
\begin{align*}
0 \ &\geq \ \E\bigg[\frac{1}{h}\int_t^{t+h} f(r,X^{t,x,a},a)\,dr + \frac{1}{h}\big(v(t+h,X^{t,x,a}) - v(t,x)\big)\bigg] \\
&\geq \ \bar\E\bigg[\frac{1}{h}\int_t^{t+h} f(r,X^{t,x,a},a)\,dr + \frac{1}{h}\big(\varphi(t+h, X^{t,x,a}) - \varphi(t,x)\big)\bigg],
\end{align*}
where the above inequality follows from the fact that $v(t,x)=\varphi(t,x)$ and $v\geq\varphi$ on $[t,T]\times C([0,T];\R^d)$. Now, by the functional It\^o formula \eqref{Ito_formula}, we obtain
\begin{align*}
0 \ &\geq \ \frac{1}{h}\int_t^{t+h} \E\big[\partial_t^H
\varphi(r,X^{t,x,a})\big]dr + \frac{1}{h}\int_t^{t+h}
\E\big[\langle b(r,X^{t,x,a},a),\partial_x^V
\varphi(r, X^{t,x,a})\rangle\big]dr \\
&+ \frac{1}{h}\int_t^{t+h}\frac{1}{2}
\E\big[\textup{tr}\big[(\sigma\sigma\trans)
(r,X^{t,x,a},a) \partial_{xx}^V \varphi(r, X^{t,x,a})\big]\big]dr + \frac{1}{h}\int_t^{t+h}\E\big[f(r, X^{t,x,a},a)\big]dr.
\end{align*}
Letting $h\rightarrow0^+$, exploiting the regularity of $\varphi$ and the continuity of $b$, $\sigma$, $f$, we find
\begin{align*}
0 \ \geq \ \partial_t^H \varphi(t,x) + \big\langle b(t,x,a),\partial^V_x \varphi(t,x)\big\rangle + \frac{1}{2}\text{tr}\big[(\sigma\sigma\trans)(t,x,a)
\partial^V_{xx}\varphi(t,x)\big] + f(t,x,a).	
\end{align*}
From the arbitrariness of $a$, we conclude that \eqref{Super} holds (with $F$ given by \eqref{F_HJB}).
\end{proof}

\section{Uniqueness}
\label{S:Uniq}

\subsection{Smooth variational principle}

This section is devoted to state a smooth variational principle on $[0,T]\times C([0,T];\R^d)$ which will be an essential tool in the proof of the comparison theorem (Theorem \ref{T:Comparison}). Notice that such a smooth variational principle is obtained from \cite[Theorem 1]{LiShi} (see also \cite[Theorem 2.5.2]{BZ05}), which is a generalization of the Borwein-Preiss variant (\cite{BP87}) of Ekeland's variational principle (\cite{Ekeland}). More precisely, \cite[Theorem 1]{LiShi} extends Ekeland's principle to the concept of gauge-type function, that we now introduce.

\begin{Definition}\label{D:Gauge}
A map $\Psi\colon([0,T]\times C([0,T];\R^d))^2\rightarrow[0,+\infty]$ is called a \textbf{gauge-type function} if it satisfies the following properties.
\begin{enumerate}[\upshape a)]
\item $(t,x)\mapsto\Psi((t,x),(t_0,x_0))$ is lower semi-continuous on $[0,T]\times C([0,T];\R^d)$, for every fixed $(t_0,x_0)\in[0,T]\times C([0,T];\R^d)$.
\item $\Psi((t,x),(t,x))=0$, for every $(t,x)\in[0,T]\times C([0,T];\R^d)$.
\item For every $\eps>0$ there exists $\eta>0$ such that, for all $(t',x'),(t'',x'')\in[0,T]\times C([0,T];\R^d)$, the inequality $\Psi((t',x'),(t'',x''))\leq\eta$ implies $d_\infty((t',x'),(t'',x''))\leq\eps$.
\end{enumerate}
\end{Definition}

\noindent In the proof of the comparison theorem we need a gauge-type function $\Psi$ such that $(t,x)\mapsto\Psi((t,x),(t_0,x_0))$ is \emph{smooth} on $[t_0,T]\times C([0,T];\R^d)$, for every fixed $(t_0,x_0)$. Notice that $d_\infty$ is obviously a gauge-type function, however it is not smooth enough.

\begin{Lemma}\label{L:SmoothGauge}
Define the map $\kappa_\infty\colon([0,T]\times C([0,T];\R^d))^2\rightarrow[0,+\infty)$ as
\begin{align}
&\kappa_\infty\big((t,x),(t',x')\big) = \label{rho_infty} \\
&=
\begin{cases}
\vspace{2mm}\dfrac{\left(\|x_{\cdot\wedge t}-x_{\cdot\wedge t'}'\|_T^4-|x(t)-x'(t')|^4\right)^3}{\|x_{\cdot\wedge t}-x_{\cdot\wedge t'}'\|_T^8} + 3|x(t)-x'(t')|^4, \hspace{5mm} \|x_{\cdot\wedge t}-x_{\cdot\wedge t'}'\|_T\neq0, \\
0,\qquad \qquad \qquad
\hspace{7.9cm} \|x_{\cdot\wedge t}-x_{\cdot\wedge t'}'\|_T=0,
\end{cases}
\notag
\end{align}
for all $(t,x),(t',x')\in[0,T]\times C([0,T];\R^d)$. Then, $\kappa_\infty$ is continuous and satisfies the following inequalities:
\begin{equation}\label{Estimate_rho_infty}
\big\|x_{\cdot\wedge t}-x_{\cdot\wedge t'}'\big\|_T^4 \ \leq \ \kappa_\infty\big((t,x),(t',x')\big) \ \leq \ 3\big\|x_{\cdot\wedge t}-x_{\cdot\wedge t'}'\big\|_T^4.
\end{equation}
Moreover, for every fixed $(t',x')\in[0,T]\times C([0,T];\R^d)$, the map $[t',T]\times C([0,T];\R^d)\ni(t,x)\mapsto\kappa_\infty((t,x),(t',x'))$ belongs to $C^{1,2}([t',T]\times C([0,T];\R^d))$ and its horizontal derivative is identically equal to zero. Its vertical derivatives of first and second-order satisfy
\begin{align}
\big|\partial_{x_i}^V\kappa_\infty\big((t,x),(t',x')\big)\big| \ &\leq \ c\,\big\|x_{\cdot\wedge t}-x_{\cdot\wedge t'}'\big\|_T^3, \label{BoundFirstDeriv} \\
\big|\partial_{x_ix_j}^V\kappa_\infty\big((t,x),(t',x')\big)\big| \ &\leq \ c\,\big\|x_{\cdot\wedge t}-x_{\cdot\wedge t'}'\big\|_T^2, \label{BoundSecondDeriv}
\end{align}
for some constant $c>0$, for every $i,j=1,\ldots,d$. Finally, the following pseudo-triangle inequality holds:
\begin{equation}\label{TriangleIneq}
|t-t'|^4 + \kappa_\infty\big((t,x),(t',x')\big) \ \leq \ 3\cdot2^3\Big(|t-t_0|^4 + \kappa_\infty\big((t,x),(t_0,x_0)\big) + |t_0-t'|^4 +\kappa_\infty\big((t_0,x_0),(t',x')\big)\Big),
\end{equation}
for all $(t,x),(t',x'),(t_0,x_0)\in[0,T]\times C([0,T];\R^d)$.
\end{Lemma}
\begin{proof}[\textbf{Proof.}]
The claim follows from \cite[Lemma 3.1]{Zhou}, apart from \eqref{TriangleIneq}. More precisely, let $\Upsilon^{m,M}$ be the function defined at the beginning of \cite[Section 3]{Zhou}. Then, notice that $\kappa_\infty$ corresponds to $\Upsilon^{2,3}$. As a consequence, \eqref{Estimate_rho_infty} follows from inequalities (3.1) in \cite{Zhou}. In addition, the fact that, for every fixed $(t',x')\in[0,T]\times C([0,T];\R^d)$, the map $[t',T]\times C([0,T];\R^d)\ni(t,x)\mapsto\kappa_\infty((t,x),(t',x'))$ belongs to $C^{1,2}([t',T]\times C([0,T];\R^d))$ follows from \cite[Lemma 3.1]{Zhou}. Moreover, the fact that its horizontal derivative is identically equal to zero is proved at the beginning of the proof of \cite[Lemma 3.1]{Zhou}. Concerning estimate \eqref{BoundFirstDeriv}, this follows from the explicit expressions of the first-order vertical derivatives of $\kappa_\infty$ reported in (3.8) of \cite{Zhou}. Finally, estimate \eqref{BoundSecondDeriv} follows from the explicit expressions of the second-order vertical derivatives of $\kappa_\infty$ given in (3.14) of \cite{Zhou}.\\
Regarding \eqref{TriangleIneq}, by \eqref{Estimate_rho_infty} we have
\begin{align*}
&|t-t'|^4 + \kappa_\infty\big((t,x),(t',x')\big) \ \leq \ |t-t'|^4+3\big\|x(\cdot\wedge t)-x'(\cdot\wedge t')\big\|_T^4 \ \leq \ 3|t-t'|^4+3\big\|x(\cdot\wedge t)-x'(\cdot\wedge t')\big\|_T^4 \\
&\leq \ 3\big(|t-t_0|+|t'-t_0|\big)^4 + 3\Big(\big\|x(\cdot\wedge t)-x_0(\cdot\wedge t_0)\big\|_T+\big\|x'(\cdot\wedge t')-x_0(\cdot\wedge t_0)\big\|_T\Big)^4 \\
&\leq \ 3\cdot2^3\Big(|t-t_0|^4+|t'-t_0|^4\Big) + 3\cdot2^3\Big(\big\|x(\cdot\wedge t)-x_0(\cdot\wedge t_0)\big\|_T^4 + \big\|x'(\cdot\wedge t')-x_0(\cdot\wedge t_0)\big\|_T^4\Big) \\
&\leq \ 3\cdot2^3\Big(|t-t_0|^4+|t'-t_0|^4 + \kappa_\infty\big((t,x),(t_0,x_0)\big)+\kappa_\infty\big((t',x'),(t_0,x_0)\big)\Big).
\end{align*}
\end{proof}

\noindent Next result provides a gauge-type function with bounded derivatives, built starting from $\kappa_\infty$ in \eqref{rho_infty}.

\begin{Corollary}\label{C:SmoothGaugeBeta}
Let $\rho_\infty\colon([0,T]\times C([0,T];\R^d))^2\rightarrow[0,+\infty]$ be defined as
\[
\rho_\infty\big((t,x),(t',x')\big) \ = \
\begin{cases}
\vspace{2mm}\displaystyle|t-t'|^2 + \frac{\kappa_\infty\big((t,x),(t',x')\big)}{1 + \kappa_\infty\big((t,x),(t',x')\big)}, &\qquad t\geq t', \\
+\infty, &\qquad t<t',
\end{cases}
\]
for all $(t,x),(t',x')\in[0,T]\times C([0,T];\R^d)$. Then, $\rho_\infty$ is a gauge-type function. In addition, for every fixed $(t',x')\in[0,T]\times C([0,T];\R^d)$, the map $[t',T]\times C([0,T];\R^d)\ni(t,x)\mapsto\rho_\infty((t,x),(t',x'))$ belongs to $C^{1,2}([t',T]\times C([0,T];\R^d))$ and it has bounded derivatives.
\end{Corollary}
\begin{proof}[\textbf{Proof.}]
The claim follows directly from Lemma \ref{L:SmoothGauge}. As a matter of fact, from the continuity of $\kappa_\infty$ we deduce that, for every fixed $(t_0,x_0)\in[0,T]\times C([0,T];\R^d)$, the map $(t,x)\mapsto\rho_\infty((t,x),(t_0,x_0))$ is lower semi-continuous on $[0,T]\times C([0,T];\R^d)$. This proves item a) of Definition \ref{D:Gauge}. Moreover, item b) is obvious, while item c) follows from inequalities \eqref{Estimate_rho_infty}. Finally, the fact that, for every fixed $(t',x')\in[0,T]\times C([0,T];\R^d)$, the map $[t',T]\times C([0,T];\R^d)\ni(t,x)\mapsto\rho_\infty((t,x),(t',x'))$ belongs to $C^{1,2}([t',T]\times C([0,T];\R^d))$ and it has bounded derivatives follows from the regularity of $\kappa_\infty$ and the estimates on its derivatives stated in Lemma \ref{L:SmoothGauge}.
\end{proof}

\noindent We can finally state the smooth variational principle on $[0,T]\times C([0,T];\R^d)$.

\begin{Theorem}\label{T:SmoothVP}
Let $\lambda>0$, $\delta>0$, and let $G\colon[0,T]\times C([0,T];\R^d)\rightarrow\R$ be an upper semicontinuous map, bounded from above. Let also $(t_0,x_0)\in[0,T]\times C([0,T];\R^d)$ satisfy
\[
G(t_0,x_0) \ \geq \ \sup_{(t,x)\in[0,T]\times C([0,T];\R^d)} G(t,x) - \lambda.
\]
Then, there exist $\{(t_i,x_i)\}_{i\geq1}\subset[0,T]\times C([0,T];\R^d)$ converging to some $(\bar t,\bar x)\in[0,T]\times C([0,T];\R^d)$ and $\varphi\colon[0,T]\times C([0,T];\R^d)\rightarrow[0,+\infty]$ given by
\begin{equation}\label{varphi}
\varphi(t,x) \ := \ \sum_{i=0}^{+\infty} \frac{\delta}{2^i} \rho_\infty\big((t,x),(t_i,x_i)\big), \qquad \forall\,(t,x)\in[0,T]\times C([0,T];\R^d),
\end{equation}
fulfilling the following properties.
\begin{enumerate}[\upshape i)]
\item $\rho_\infty((\bar t,\bar x),(t_i,x_i))\leq\lambda/(2^i\delta)$, for every $i\geq1$, and $\rho_\infty((\bar t,\bar x),(t_0,x_0))\leq\lambda/\delta$.
\item $G(t_0,x_0)\leq G(\bar t,\bar x)-\varphi(\bar t,\bar x)$.
\item For every $(t,x)\neq(\bar t,\bar x)$, $G(t,x) - \varphi(t,x)<G(\bar t,\bar x) - \varphi(\bar t,\bar x)$.
\item It holds that $t_0\leq\bar t$ and $t_i\leq\bar t$, for every $i\geq1$.
\end{enumerate}
In addition, the restriction of $\varphi$ to $[\bar t,T]\times C([0,T];\R^d)$ belongs to $C^{1,2}([\bar t,T]\times C([0,T];\R^d))$ and its derivatives are bounded by $c\delta$, for some constant $c\geq0$, independent of $\lambda,\delta$.
\end{Theorem}
\begin{proof}[\textbf{Proof.}]
Items i)-ii)-iii) follow from the variational principle \cite[Theorem 1]{LiShi}, which applies to a generic gauge-type function $\Psi$ (just observe that \cite[Theorem 1]{LiShi} is formulated on a complete \emph{metric} space, while $[0,T]\times C([0,T];\R^d)$ is a complete \emph{pseudometric} space; however, this does not affect the result). Notice that the quantities $\eps$ and $\delta_i$ appearing in \cite[Theorem 1]{LiShi} here are taken equal respectively to $\eps=\lambda$ and $\delta_i=\delta/2^i$, for $i\geq0$.

Concerning item iv), this is a consequence of the fact that we set $\rho_\infty((t,x),(t',x'))$ equal to $+\infty$ for $t<t'$. More precisely, item iv) can be deduced looking at the proof of \cite[Theorem 1]{LiShi} (see in particular formula (18) in \cite{LiShi} where $(t_1,x_1)$ is introduced and, more generally, formula (21) where $(t_i,x_i)$ is introduced), from which we get the inequalities $t_\delta\leq t_1\leq t_2\leq\cdots\leq t_i\leq\cdots\leq\bar t$.

Finally, the properties of $\varphi$ follows from the properties of $\rho_\infty$ stated in Corollary \ref{C:SmoothGaugeBeta} and from item iv).
\end{proof}

\subsection{Comparison theorem and uniqueness}

\begin{Theorem}\label{T:Comparison}
Suppose that Assumptions \ref{AssA}, \ref{AssB}, \ref{AssC} hold. Let $u_1,u_2\colon[0,T]\times C([0,T];\R^d)\rightarrow\R$ be such that
\begin{equation}\label{u_1_u_2}
|u_i(t,x) - u_i(t',x')| \ \leq \ C_i\big(|t-t'|^{q_i} + \|x(\cdot\wedge t) - x'(\cdot\wedge t')\|_T^{p_i}\big),
\end{equation}
for every $i=1,2$, $(t,x),(t',x')\in[0,T]\times C([0,T];\R^d)$, for some constants $C_i\geq0$, $q_i,p_i\in(0,1]$. Suppose that $u_1$ $($resp. $u_2$$)$ is a (path-dependent) viscosity subsolution $($resp. supersolution$)$ of equation \eqref{HJB}. Then $u_1\leq u_2$ on $[0,T]\times C([0,T];\R^d)$.
\end{Theorem}
\begin{proof}[\textbf{Proof.}]
The proof consists in showing that $u_1\leq v$ and $v\leq u_2$ on $[0,T]\times C([0,T];\R^d)$, with $v$ given by \eqref{Value}.

\vspace{1mm}

\noindent\textsc{Step I.} \emph{Proof of $u_1\leq v$.} We proceed by contradiction and assume that $\sup(u_1-v)>0$. Then, there exists $(t_0,x_0)\in[0,T]\times C([0,T];\R^d)$ such that
\begin{equation}\label{Contradiction}
(u_1 - v)(t_0,x_0) \ > \ 0.
\end{equation}
Notice that $t_0<T$, since $u_1(T,x_0)\leq g(x_0)=v(T,x_0)$. Now, consider the sequences $\{b_n\}_n$, $\{f_n\}_n$, $\{g_n\}_n$ in \eqref{Coeff_n}. Moreover, for every $n$ and any $\eps\in(0,1)$, consider the functions $v_{n,\eps}\in C_{\textup{pol}}^{1,2}([0,T]\times C([0,T];\R^d))$ and $\bar v_{n,\eps}\in C^{1,2}([0,T]\times\R^d)$ introduced in Theorem \ref{T:CylindrApprox1}, with $v_{n,\eps}$ classical solution of the following equation:
\[
\begin{cases}
\vspace{2mm}
\partial_t^H v_{n,\eps}(t,x) + \dfrac{1}{2}\eps^2\textup{tr}\big[\partial_{yy}\bar v_{n,\eps}(t,y_n^{t,x})\big] + \sup_{a\in A}\bigg\{\big\langle b_n(t,x,a),\partial^V_x v_{n,\eps}(t,x)\big\rangle \\
\vspace{2mm}+\,\dfrac{1}{2}\textup{tr}\big[(\sigma\sigma\trans)(t,x,a)\partial^V_{xx}v_{n,\eps}(t,x)\big]+f_n(t,x,a)\bigg\} = 0, &\hspace{-2.4cm}(t,x)\in[0,T)\times C([0,T];\R^d), \\
v_{n,\eps}(T,x) = g_n(x), &\,\hspace{-2.4cm}x\in C([0,T];\R^d),
\end{cases}
\]
where $y_n^{t,x}$ is given by \eqref{y_n^t,x}. Notice that the term $\frac{1}{2}\eps^2\textup{tr}[\partial_{yy}\bar v_{n,\eps}(t,y_n^{t,x})]$ depends on the function $\bar v_{n,\eps}$ rather than on $v_{n,\eps}$, see Remark \ref{R:v_eps}.\\
We split the rest of the proof of \textsc{Step I} into four substeps.

\vspace{4mm}

\noindent\textsc{Substep I-a}. Given $k\in\N$, we set $\tilde u_1(t,x):=\text{e}^{t-t_0}\,u_1(t,x)$, for all $(t,x)\in[0,T]\times C([0,T];\R^d)$, and we define similarly $\tilde v_{n,\eps}$, $\tilde f$, $\tilde f_n$. We also define $\tilde g(x):=\text{e}^{T-t_0}\,g(x)$ and $\tilde g_n(x):=\text{e}^{T-t_0}\,g_n(x)$, for all $x\in C([0,T];\R^d)$. Notice that $\tilde u_1$ is a (path-dependent) viscosity subsolution of the following path-dependent partial differential equation:
\begin{equation}\label{HJB_exp}
\begin{cases}
\vspace{2mm}
\partial_t^H \tilde u_1(t,x) + \sup_{a\in A}\bigg\{\big\langle b(t,x,a),\partial^V_x \tilde u_1(t,x)\big\rangle + \dfrac{1}{2}\text{tr}\big[(\sigma\sigma\trans)(t,x,a)\partial^V_{xx}\tilde u_1(t,x)\big] \\
\vspace{2mm}+\,\tilde f(t,x,a)\bigg\} = \tilde u_1(t,x), &\hspace{-4.5cm}(t,x)\in[0,T)\times C([0,T];\R^d), \\
\tilde u_1(T,x) = \tilde g(x), &\hspace{-4.5cm}x\in C([0,T];\R^d).
\end{cases}
\end{equation}
Similarly, $\tilde v_{n,\eps}\in C_{\textup{pol}}^{1,2}([0,T]\times C([0,T];\R^d))$ and is a classical solution of the following equation:
\begin{equation}\label{HJB_exp_n,eps}
\hspace{-3cm}\begin{cases}
\vspace{2mm}
\partial_t^H \tilde v_{n,\eps}(t,x) + \dfrac{1}{2}\eps^2\text{e}^{t-t_0}\textup{tr}\big[\partial_{yy}\bar v_{n,\eps}(t,y_n^{t,x})\big] + \sup_{a\in A}\bigg\{\big\langle b_n(t,x,a),\partial^V_x \tilde v_{n,\eps}(t,x)\big\rangle \\
\vspace{2mm}+\,\dfrac{1}{2}\textup{tr}\big[(\sigma\sigma\trans)(t,x,a)\partial^V_{xx}\tilde v_{n,\eps}(t,x)\big] \\
\vspace{2mm}+\,\tilde f_n(t,x,a)\bigg\} = \tilde v_{n,\eps}(t,x), &\hspace{-7cm}(t,x)\in[0,T)\times C([0,T];\R^d), \\
\tilde v_{n,\eps}(T,x) = \tilde g_n(x), &\,\hspace{-7cm}x\in C([0,T];\R^d).
\end{cases}
\end{equation}

\vspace{4mm}

\noindent\textsc{Substep I-b}. We begin noting that by \eqref{u_1_u_2} and item 5) of Theorem \ref{T:CylindrApprox1}, there exist constants $M'\geq0$ and $p\in(0,1/2]$, independent of $n,\eps$, such that
\[
\big|\big(\tilde u_1-\tilde v_{n,\eps}\big)(t,x) - \big(\tilde u_1-\tilde v_{n,\eps}\big)(t',x')\big| \ \leq \ \frac{M'}{2}\big(|t-t'|^p + \|x(\cdot\wedge t) - x'(\cdot\wedge t')\|_T^p\big),
\]
for all $(t,x),(t',x')\in[0,T]\times C([0,T];\R^d)$, with $|t-t'|\leq1$ and $\|x(\cdot\wedge t)-x'(\cdot\wedge t')\|_T\leq1$. By \eqref{Estimate_rho_infty}, we find
\begin{align*}
\big|\big(\tilde u_1-\tilde v_{n,\eps}\big)(t,x) - \big(\tilde u_1-\tilde v_{n,\eps}\big)(t',x')\big| \ &\leq \ \frac{M'}{2}\big(|t-t'|^p + \|x(\cdot\wedge t) - x'(\cdot\wedge t')\|_T^p\big) \\
&= \ \frac{M'}{2}\Big(|t-t'|^p + \|x(\cdot\wedge t) - x'(\cdot\wedge t')\|_T^{4\frac{p}{4}}\Big) \\
&\leq \ \frac{M'}{2}\Big(|t-t'|^p + \kappa_\infty\big((t,x),(t',x')\big)^{\frac{p}{4}}\Big).
\end{align*}
Then, by the concavity of the map $y\mapsto y^{\frac{p}{4}}$, we obtain
\begin{align}\label{UnifCont_u-v_M'}
\big|\big(\tilde u_1-\tilde v_{n,\eps}\big)(t,x) - \big(\tilde u_1-\tilde v_{n,\eps}\big)(t',x')\big| \ &\leq \ \frac{M'}{2}2^{1-\frac{p}{4}}\Big(|t-t'|^4 + \kappa_\infty\big((t,x),(t',x')\big)\Big)^{\frac{p}{4}} \notag \\
&\leq \ M'\Big(|t-t'|^4 + \kappa_\infty\big((t,x),(t',x')\big)\Big)^{\frac{p}{4}},
\end{align}
for all $(t,x),(t',x')\in[0,T]\times C([0,T];\R^d)$, with $|t-t'|\leq1$ and $\|x(\cdot\wedge t)-x'(\cdot\wedge t')\|_T\leq1$. Let $M>0$ be such that
\begin{equation}\label{M_Bound}
M \ > \ 2^{2+p}3^{\frac{p}{4}}M'.
\end{equation}
Since in particular $M>M'$, from \eqref{UnifCont_u-v_M'} we deduce that
\begin{equation}\label{UnifCont_u-v}
\big|\big(\tilde u_1-\tilde v_{n,\eps}\big)(t,x) - \big(\tilde u_1-\tilde v_{n,\eps}\big)(t',x')\big| \ \leq \ M\Big(|t-t'|^4 + \kappa_\infty\big((t,x),(t',x')\big)\Big)^{\frac{p}{4}}.
\end{equation}
Now, let
\begin{equation}\label{gamma}
\gamma \ = \ \frac{(M/M')^{\frac{4}{p}}}{3\cdot2^{3+\frac{8}{p}}}-1.
\end{equation}
Notice that by \eqref{M_Bound} it follows that $\gamma$ is strictly positive (in fact $\gamma>1$). Moreover, let $\bar r\in(0,1)$ be such that
\begin{equation}\label{bar_r}
\bar r^4 \ = \ \frac{1}{3\cdot2^3(1+\gamma)}.
\end{equation}
Now, for each $r\in(0,\bar r)$, let $C_r>0$ be such that the map $y\mapsto r^\frac{4}{y}-\frac{r^p}{\sqrt{y}}$ is strictly positive for $y\geq C_r$. Then, given $r\in(0,\bar r)$ and $k\geq C_r$, we construct a map $G_{n,\eps,k,r}\colon[0,T]\times C([0,T];\R^d)\rightarrow\R$ of the form
\begin{align*}
G_{n,\eps,k,r}(t,x) \ &= \ \big(\tilde u_1-\tilde v_{n,\eps}\big)(t,x) - h_{k,r}\Big(|t-t_0|^4 + \kappa_\infty\big((t,x),(t_0,x_0)\big)\Big) \\
&\quad \ - \psi_{k,r}\Big(|t-t_0|^4 + \kappa_\infty\big((t,x),(t_0,x_0)\big)\Big),
\end{align*}
for some smooth functions $h_{k,r}$ and $\psi_{k,r}$, such that $G_{n,\eps,k,r}$ satisfies the following properties:
\begin{enumerate}[1)]
\item $G_{n,\eps,k,r}(t_0,x_0)=(\tilde u_1-\tilde v_{n,\eps})(t_0,x_0)$;
\item $G_{n,\eps,k,r}(t,x)\leq(\tilde u_1-\tilde v_{n,\eps})(t_0,x_0) + Mr^p$, for $|t-t_0|^4+\kappa_\infty((t,x),(t_0,x_0))\leq r^4$;
\item $G_{n,\eps,k,r}(t,x)\leq(\tilde u_1-\tilde v_{n,\eps})(t_0,x_0) - \frac{1}{2}Mr^p$, for $|t-t_0|^4+\kappa_\infty((t,x),(t_0,x_0))= r^4$;
\item $G_{n,\eps,k,r}(t,x)\leq(\tilde u_1-\tilde v_{n,\eps})(t_0,x_0) - \frac{1}{4}Mr^p$, for $|t-t_0|^4+\kappa_\infty((t,x),(t_0,x_0))>r^4$;
\item the horizontal and vertical derivatives of first and second-order of the map $[t_0,T]\times C([0,T];\R^d)\ni(t,x)\mapsto h_{k,r}(|t-t_0|^4+\kappa_\infty((t,x),(t_0,x_0)))$ are bounded by $\frac{D_r}{\sqrt{k}}$ on the set $\{|t-t_0|^4+\kappa_\infty((t,x),(t_0,x_0))\leq r^4\}$, for some constant $D_r>0$, possibly depending on $r$, but independent of $n,\eps,k$.
\end{enumerate}
Notice that from properties 2) and 4) it follows that
\begin{equation}\label{G_Propriety4}
\sup_{(t,x)\in[0,T]\times C([0,T];\R^d)} G_{n,\eps,k,r}(t,x) \ \leq \ \big(\tilde u_1-\tilde v_{n,\eps}\big)(t_0,x_0) + Mr^p.
\end{equation}
In order to construct $G_{n,\eps,k,r}$, we consider a function $\psi_{k,r}\colon[0,\infty)\rightarrow[0,\infty)$, such that:
\begin{enumerate}[a)]
\item $\psi_{k,r}\in C^2([0,\infty))$ and has polynomial growth together with its derivatives of first and second-order;
\item $\psi_{k,r}(s)=0$ for $0\leq s\leq r^4$;
\item for $s\geq\gamma r^4$, with $\gamma>1$ as in \eqref{gamma}, $\psi_{k,r}$ satisfies
\begin{equation}\label{property_c}
\psi_{k,r}(s) \ \geq \ 2C'\sum_{i=1}^4\big(3\cdot2^3s + 3\cdot2^3r^4\big)^{\frac{\beta_i}{4}} - \frac{1}{4}Mr^p,
\end{equation}
where $C'\geq0$ is as in \eqref{C'} and $\beta_1:=q_1$, $\beta_2:=p_1$, $\beta_3:=1/2$, $\beta_4:=1$.
\end{enumerate}
Then, we define $G_{n,\eps,k,r}$ as follows:
\begin{align*}
G_{n,\eps,k,r}(t,x) \ &:= \ \big(\tilde u_1-\tilde v_{n,\eps}\big)(t,x) - 3M\sqrt{k}\bigg\{\bigg(\frac{1}{\ell_{k,r}}+|t-t_0|^4+\kappa_\infty\big((t,x),(t_0,x_0)\big)\bigg)^{\frac{1}{k}} - \frac{1}{\ell_{k,r}^{\frac{1}{k}}}\bigg\} \\
&\quad \ - \psi_{k,r}\Big(|t-t_0|^4 + \kappa_\infty\big((t,x),(t_0,x_0)\big)\Big),
\end{align*}
for every $(t,x)\in[0,T]\times C([0,T];\R^d)$, with
\begin{equation}\label{ell}
\ell_{k,r} \ = \ \Bigg(\frac{1}{r^{\frac{4}{k}}-\frac{r^p}{\sqrt{k}}}\Bigg)^k.
\end{equation}
Since we are assuming $k\geq C_r$, then $r^\frac{4}{k}-\frac{r^{p}}{\sqrt{k}}>0$.\\
We begin noting that
\begin{align}\label{G_Property5}
G_{n,\eps,k,r}(t,x) \ &= \ \big(\tilde u_1-\tilde v_{n,\eps}\big)(t,x) \\
&\quad \ - 3M\sqrt{k}\bigg\{\bigg(\frac{1}{\ell_{k,r}}+|t-t_0|^4+\kappa_\infty\big((t,x),(t_0,x_0)\big)\bigg)^{\frac{1}{k}} - \frac{1}{\ell_{k,r}^{\frac{1}{k}}}\bigg\}, \notag
\end{align}
for $|t-t_0|^4+\kappa_\infty((t,x),(t_0,x_0))\leq r^4$. As a matter of fact, if $|t-t_0|^4+\kappa_\infty((t,x),(t_0,x_0))\leq r^4$ then $\psi_{k,r}(|t-t_0|^4+\kappa_\infty((t,x),(t_0,x_0)))$ is equal to zero by property b) of $\psi_{k,r}$.

\vspace{2mm}

\noindent\emph{Property 1).} Notice that property 1) above holds true:
\[
G_{n,\eps,k,r}(t_0,x_0) \ = \ \big(\tilde u_1-\tilde v_{n,\eps}\big)(t_0,x_0).
\]

\vspace{2mm}

\noindent\emph{Property 2).} This property follows from \eqref{UnifCont_u-v}, the fact that $h_{k,r}\geq0$ and $\psi_{k,r}\geq0$:
\begin{align*}
G_{n,\eps,k,r}(t,x) - \big(\tilde u_1-\tilde v_{n,\eps}\big)(t_0,x_0) \ &\leq \ \big(\tilde u_1-\tilde v_{n,\eps}\big)(t,x) - \big(\tilde u_1-\tilde v_{n,\eps}\big)(t_0,x_0) \\
&\leq \ M\Big(|t-t_0|^4 + \kappa_\infty\big((t,x),(t_0,x_0)\big)\Big)^{\frac{p}{4}} \ \leq \ Mr^p.
\end{align*}

\vspace{2mm}

\noindent\emph{Property 3).} Regarding property 3), by \eqref{UnifCont_u-v} and \eqref{G_Property5}, for all $(t,x)\in[0,T]\times C([0,T];\R^d)$, with $|t-t_0|^4+\kappa_\infty((t,x),(t_0,x_0))=r^4$,
\begin{align*}
&G_{n,\eps,k,r}(t,x) \ = \ \big(\tilde u_1-\tilde v_{n,\eps}\big)(t,x) - \big(\tilde u_1-\tilde v_{n,\eps}\big)(t_0,x_0) + \big(\tilde u_1-\tilde v_{n,\eps}\big)(t_0,x_0) \\
&- 3M\sqrt{k}\bigg\{\bigg(\frac{1}{\ell_{k,r}}+|t-t_0|^4+\kappa_\infty\big((t,x),(t_0,x_0)\big)\bigg)^{\frac{1}{k}} - \frac{1}{\ell_{k,r}^{\frac{1}{k}}}\bigg\} \\
&\leq \ M\Big(|t-t_0|^4 + \kappa_\infty\big((t,x),(t_0,x_0)\big)\Big)^{\frac{p}{4}} + \big(\tilde u_1-\tilde v_{n,\eps}\big)(t_0,x_0) - 3M\sqrt{k}\bigg\{\bigg(\frac{1}{\ell_{k,r}}+r^4\bigg)^{\frac{1}{k}} - \frac{1}{\ell_{k,r}^{\frac{1}{k}}}\bigg\} \\
&= \ Mr^p + \big(\tilde u_1-\tilde v_{n,\eps}\big)(t_0,x_0) - 3M\sqrt{k}\bigg\{\bigg(\frac{1}{\ell_{k,r}}+r^4\bigg)^{\frac{1}{k}} - \frac{1}{\ell_{k,r}^{\frac{1}{k}}}\bigg\}.
\end{align*}
Now, notice that
\begin{equation}\label{Property3_Proof}
Mr^p - 3M\sqrt{k}\bigg\{\bigg(\frac{1}{\ell_{k,r}}+r^4\bigg)^{\frac{1}{k}} - \frac{1}{\ell_{k,r}^{\frac{1}{k}}}\bigg\} \ \leq \ - \frac{1}{2}Mr^p,
\end{equation}
from which property 3) follows. As a matter of fact, \eqref{Property3_Proof} can be written as
\[
\frac{3M}{2}r^p \ \leq \ 3M\sqrt{k}\bigg\{\bigg(\frac{1}{\ell_{k,r}}+r^4\bigg)^{\frac{1}{k}} - \frac{1}{\ell_{k,r}^{\frac{1}{k}}}\bigg\},
\]
which becomes
\begin{equation}\label{Property3_Proof_2}
r^p + \frac{2\sqrt{k}}{\ell_{k,r}^{\frac{1}{k}}} \ \leq \ 2\sqrt{k}\bigg(\frac{1}{\ell_{k,r}}+r^4\bigg)^{\frac{1}{k}}.
\end{equation}
On the other hand, using the concavity of the map $y\mapsto y^{1/k}$, we have
\[
2\bigg(\frac{1}{\ell_{k,r}}+r^4\bigg)^{\frac{1}{k}} \ \geq \ 2^{\frac{1}{k}}\bigg(\frac{1}{\ell_{k,r}^{\frac{1}{k}}}+r^{\frac{4}{k}}\bigg) \ \geq \ \frac{1}{\ell_{k,r}^{\frac{1}{k}}}+r^{\frac{4}{k}},
\]
where the last inequality follows from the fact that $2^{1/k}\geq1$. Therefore, \eqref{Property3_Proof_2} holds true if
\[
r^p + \frac{2\sqrt{k}}{\ell_{k,r}^{\frac{1}{k}}} \ \leq \ \sqrt{k}\bigg(\frac{1}{\ell_{k,r}^{\frac{1}{k}}} + r^{\frac{4}{k}}\bigg).
\]
This yields
\[
r^{\frac{4}{k}} - \frac{r^p}{\sqrt{k}} \ \geq \ \frac{1}{\ell_{k,r}^{\frac{1}{k}}},
\]
which is indeed an equality due to the expression of $\ell_{k,r}$ in \eqref{ell}.

\vspace{2mm}

\noindent\emph{Property 4).} Regarding property 4), we have, for every $|t-t_0|^4+\kappa_\infty((t,x),(t_0,x_0))>r^4$,
\begin{align}\label{G_Property4}
G_{n,\eps,k,r}(t,x) \ &= \ G_{n,\eps,k,r}(t,x) - G_{n,\eps,k,r}(t',x') + G_{n,\eps,k,r}(t',x') \notag \\
&= \ \big(\tilde u_1-\tilde v_{n,\eps}\big)(t,x) - \big(\tilde u_1-\tilde v_{n,\eps}\big)(t',x') + G_{n,\eps,k,r}(t',x') \notag \\
&\quad \ - 3M\sqrt{k}\bigg\{\bigg(\frac{1}{\ell_{k,r}}+|t-t_0|^4+\kappa_\infty\big((t,x),(t_0,x_0)\big)\bigg)^{\frac{1}{k}} - \frac{1}{\ell_{k,r}^{\frac{1}{k}}}\bigg\} \notag \\
&\quad \ + 3M\sqrt{k}\bigg\{\bigg(\frac{1}{\ell_{k,r}}+|t'-t_0|^4+\kappa_\infty\big((t',x'),(t_0,x_0)\big)\bigg)^{\frac{1}{k}} - \frac{1}{\ell_{k,r}^{\frac{1}{k}}}\bigg\} \notag \\
&\quad \ - \psi_{k,r}\Big(|t-t_0|^4 + \kappa_\infty\big((t,x),(t_0,x_0)\big)\Big) \notag \\
&\leq \ \big(\tilde u_1-\tilde v_{n,\eps}\big)(t,x) - \big(\tilde u_1-\tilde v_{n,\eps}\big)(t',x') + G_{n,\eps,k,r}(t',x') \\
&\quad \ - \psi_{k,r}\Big(|t-t_0|^4 + \kappa_\infty\big((t,x),(t_0,x_0)\big)\Big), \notag
\end{align}
for any $|t'-t_0|^4+\kappa_\infty((t',x'),(t_0,x_0))=r^4$, where we used \eqref{G_Property5}, namely that $\psi_{k,r}$ does not appear in the expression of $G_{n,\eps,k,r}(t',x')$, and also that the map
\[
s \ \longmapsto \ 3M\sqrt{k}\bigg\{\bigg(\frac{1}{\ell_{k,r}}+s\bigg)^{\frac{1}{k}} - \frac{1}{\ell_{k,r}^{\frac{1}{k}}}\bigg\}
\]
is non-decreasing.\\
Now, we prove that the set $A:=\{(t,x)\colon|t-t_0|^4+\kappa_\infty((t,x),(t_0,x_0))>r^4\}$ can be written as $A_1\cup A_2$, with $A_1$ and $A_2$ disjoint sets, where:
\begin{itemize}
\item $A_1$ is the set of points $(t,x)\in A$ for which there exists $(t',x')$ (with $|t'-t_0|^4+\kappa_\infty((t',x'),(t_0,x_0))=r^4$) satisfying $|t-t'|\leq1$, $\|x(\cdot\wedge t)-x'(\cdot\wedge t')\|_T\leq1$, $|t-t'|^4 + \kappa_\infty((t,x),(t',x'))\leq\big(\frac{1}{4}r^pM/M'\big)^{4/p}$;
\item $A_2$ is the set of points $(t,x)\in A$ such that
\begin{equation}\label{(t,x)(t',x')}
|t-t_0|^4 + \kappa_\infty\big((t,x),(t_0,x_0)\big) \ > \ \gamma r^4,
\end{equation}
with $\gamma>1$ given by \eqref{gamma}.
\end{itemize}
Let us prove that $A_1$, $A_2$ form a partition of $A$. To this end, consider a point $(t,x)\in A\backslash A_1$ and suppose that $|t-t'|^4 + \kappa_\infty((t,x),(t',x'))>\big(\frac{1}{4}r^pM/M'\big)^{4/p}$ for any $(t',x')$ with $|t'-t_0|^4+\kappa_\infty((t',x'),(t_0,x_0))=r^4$. By \eqref{TriangleIneq}, we have (recalling the definition of $\gamma$ in \eqref{gamma})
\[
|t-t_0|^4 + \kappa_\infty\big((t,x),(t_0,x_0)\big) \ > \ \frac{1}{3\cdot2^3}\bigg(\frac{1}{4}\frac{M}{M'}r^p\bigg)^{\frac{4}{p}} - r^4 \ = \ \bigg(\frac{(M/M')^{\frac{4}{p}}}{3\cdot2^{3+\frac{8}{p}}} - 1\bigg)r^4 \ = \ \gamma r^4.
\]
This shows that $(t,x)\in A_2$. Now, suppose on the other hand that $|t-t'|>1$ (in the case $\|x(\cdot\wedge t)-x'(\cdot\wedge t')\|_T>1$ we proceed along the same lines). Then $|t-t'|^4 + \kappa_\infty((t,x),(t',x'))>1$. Therefore, by \eqref{TriangleIneq}, we get
\[
|t-t_0|^4 + \kappa_\infty\big((t,x),(t_0,x_0) \ > \ \frac{1}{3\cdot2^3} - r^4 \ > \ \gamma r^4,
\]
which shows that $(t,x)\in A_2$, where the last inequality follows from the fact that $r\in(0,\bar r)$ and the definition of $\bar r$ in \eqref{bar_r}.\\
For the rest of the proof of \emph{Property 4)} we distinguish two cases depending on the point $(t,x)$. 

\vspace{2mm}

\noindent\emph{Case 1: $(t,x)\in A_1$.} By \eqref{UnifCont_u-v_M'} and property 3) of $G_{n,\eps,k,r}$, from \eqref{G_Property4} we find
\begin{align}\label{Case1}
G_{n,\eps,k,r}(t,x) \ &\leq \ M'\Big(|t-t'|^4 + \kappa_\infty\big((t,x),(t',x')\big)\Big)^{\frac{p}{4}} + \big(\tilde u_1-\tilde v_{n,\eps}\big)(t_0,x_0) \\
&\quad \ - \frac{1}{2}Mr^p - \psi_{k,r}\Big(|t-t_0|^4 + \kappa_\infty\big((t,x),(t_0,x_0)\big)\Big). \notag
\end{align} Then, by \eqref{Case1}, recalling that $\psi_{k,r}\geq0$,
\[
G_{n,\eps,k,r}(t,x) \ \leq \ \big(\tilde u_1-\tilde v_{n,\eps}\big)(t_0,x_0) - \frac{1}{4}Mr^p,
\]
which proves property 4) for such $(t,x)$.

\vspace{2mm}

\noindent\emph{Case 2: $(t,x)\in A_2$.} We begin noting that by \eqref{u_1_u_2} and item 5) of Theorem \ref{T:CylindrApprox1}, there exists a constant $C'\geq0$, independent of $n,\eps$, such that
\begin{align}\label{C'}
&\big|\big(\tilde u_1-\tilde v_{n,\eps}\big)(t,x) - \big(\tilde u_1-\tilde v_{n,\eps}\big)(t',x')\big| \notag \\
&\leq \ C'\big(|t-t'|^{q_1} + \|x(\cdot\wedge t) - x'(\cdot\wedge t')\|_T^{p_1} + |t-t'|^{1/2} + \|x(\cdot\wedge t) - x'(\cdot\wedge t')\|_T\big) \notag \\
&\leq \ C'\sum_{i=1}^4\big(|t-t'|^{\beta_i} + \|x(\cdot\wedge t) - x'(\cdot\wedge t')\|_T^{\beta_i}\big),
\end{align}
for all $(t,x),(t',x')\in[0,T]\times C([0,T];\R^d)$, where $\beta_1:=q_1$, $\beta_2:=p_1$, $\beta_3:=1/2$, $\beta_4:=1$. Hence (recalling \eqref{Estimate_rho_infty})
\begin{align*}
\big|\big(\tilde u_1-\tilde v_{n,\eps}\big)(t,x) - \big(\tilde u_1-\tilde v_{n,\eps}\big)(t',x')\big| \ &\leq \ 2C'\sum_{i=1}^4\big(|t-t'|^4+\|x(\cdot\wedge t) - x'(\cdot\wedge t')\|_T^4\big)^{\frac{\beta_i}{4}} \\
&\leq \ 2C'\sum_{i=1}^4\big(|t-t'|^4+\kappa_\infty\big((t,x),(t',x')\big)\big)^{\frac{\beta_i}{4}}.
\end{align*}
Now, by \eqref{TriangleIneq} we have
\[
|t-t'|^4 + \kappa_\infty\big((t,x),(t',x')\big) \ \leq \ 3\cdot2^3\Big(|t-t_0|^4 + \kappa_\infty\big((t,x),(t_0,x_0)\big)\Big) + 3\cdot2^3r^4.
\]
Hence
\[
\big|\big(\tilde u_1-\tilde v_{n,\eps}\big)(t,x) - \big(\tilde u_1-\tilde v_{n,\eps}\big)(t',x')\big| \ \leq \ 2C'\sum_{i=1}^4\bigg(3\cdot2^3\Big(|t-t_0|^4 + \kappa_\infty\big((t,x),(t_0,x_0)\big)\Big) + 3\cdot2^3r^4\bigg)^{\frac{\beta_i}{4}}.
\]
Using the last inequality, together with \eqref{G_Property4} and property 3) of $G_{n,\eps,k,r}$, we get (recall that $(t',x')$ satisfies $|t'-t_0|^4+\kappa_\infty((t',x'),(t_0,x_0))=r^4$)
\begin{align*}
G_{n,\eps,k,r}(t,x) \ &\leq \ \big(\tilde u_1-\tilde v_{n,\eps}\big)(t,x) - \big(\tilde u_1-\tilde v_{n,\eps}\big)(t',x') + G_{n,\eps,k,r}(t',x') - \psi_{k,r}\Big(|t-t_0|^4 + \kappa_\infty\big((t,x),(t_0,x_0)\big)\Big) \\
&\leq \ 2C'\sum_{i=1}^4\bigg(3\cdot2^3\Big(|t-t_0|^4 + \kappa_\infty\big((t,x),(t_0,x_0)\big)\Big) + 3\cdot2^3r^4\bigg)^{\frac{\beta_i}{4}} + \big(\tilde u_1-\tilde v_{n,\eps}\big)(t_0,x_0) - \frac{1}{2}Mr^p \\
&\quad \ - \psi_{k,r}\Big(|t-t_0|^4 + \kappa_\infty\big((t,x),(t_0,x_0)\big)\Big) \ \leq \ \big(\tilde u_1-\tilde v_{n,\eps}\big)(t_0,x_0) - \frac{1}{4}Mr^p,
\end{align*}
where the latter inequality follows from property c) of $\psi_{k,r}$ (namely from \eqref{property_c} with $s:=|t-t_0|^4+\kappa_\infty((t,x),(t_0,x_0))$) and the fact that $(t,x)$ satisfies \eqref{(t,x)(t',x')}.

\vspace{2mm}

\noindent\emph{Property 5).} We recall that the function $h_{k,r}$ takes the following form:
\[
h_{k,r}(t,x) \ = \ 3M\sqrt{k}\bigg\{\bigg(\frac{1}{\ell_{k,r}}+|t-t_0|^4+\kappa_\infty\big((t,x),(t_0,x_0)\big)\bigg)^{\frac{1}{k}} - \frac{1}{\ell_{k,r}^{\frac{1}{k}}}\bigg\}.
\]
Its derivatives are given by:
\begin{align*}
\partial_t^H h_{k,r}(t,x) \ &= \ \frac{3M}{\sqrt{k}}\frac{4(t-t_0)^3}{\big(\frac{1}{\ell_{k,r}}+|t-t_0|^4+\kappa_\infty((t,x),(t_0,x_0))\big)^{1-\frac{1}{k}}}, \\
\partial_{x_i}^V h_{k,r}(t,x) \ &= \ \frac{3M}{\sqrt{k}}\frac{\partial_{x_i}\kappa_\infty((t,x),(t_0,x_0))}{\big(\frac{1}{\ell_{k,r}}+|t-t_0|^4+\kappa_\infty((t,x),(t_0,x_0))\big)^{1-\frac{1}{k}}}, \\
\partial_{x_ix_j}^V h_{k,r}(t,x) \ &= \ \frac{3M}{\sqrt{k}}\frac{\partial_{x_ix_j}^V\kappa_\infty((t,x),(t_0,x_0))}{\big(\frac{1}{\ell_{k,r}}+|t-t_0|^4+\kappa_\infty((t,x),(t_0,x_0))\big)^{1-\frac{1}{k}}} \\
&\quad \ + \frac{3M}{\sqrt{k}}\bigg(\frac{1}{k}-1\bigg)\frac{\partial_{x_i}^V\kappa_\infty((t,x),(t_0,x_0))\partial_{x_j}^V\kappa_\infty((t,x),(t_0,x_0))}{\big(\frac{1}{\ell_{k,r}}+|t-t_0|^4+\kappa_\infty((t,x),(t_0,x_0))\big)^{2-\frac{1}{k}}}.
\end{align*}
The claim follows if we prove that the quantities multiplying $\frac{1}{\sqrt{k}}$ are bounded on the set $\{|t-t_0|^4+\kappa_\infty((t,x),(t_0,x_0))\leq r^4\}$ by some constant $D_r>0$, independent of $k$, but possibly depending on $r$. Consider for instance the first term appearing in the second-order derivative (the other terms can be treated in a similar way):
\[
\frac{|\partial_{x_ix_j}^V\kappa_\infty((t,x),(t_0,x_0))|}{\big(\frac{1}{\ell_{k,r}}+|t-t_0|^4+\kappa_\infty((t,x),(t_0,x_0))\big)^{1-\frac{1}{k}}}.
\]
By \eqref{Estimate_rho_infty} and \eqref{BoundSecondDeriv}, we have
\begin{align*}
&\frac{|\partial_{x_ix_j}^V\kappa_\infty((t,x),(t_0,x_0))|}{\big(\frac{1}{\ell_{k,r}}+|t-t_0|^4+\kappa_\infty((t,x),(t_0,x_0))\big)^{1-\frac{1}{k}}} \\
&\leq \ \ell_{k,r}^{1-\frac{1}{k}}|\partial_{x_ix_j}^V\kappa_\infty((t,x),(t_0,x_0))| \ \leq \ c\ell_{k,r}^{1-\frac{1}{k}}\big(\kappa_\infty((t,x),(t_0,x_0))\big)^{\frac{1}{2}} \\
&\leq \ c\ell_{k,r}^{1-\frac{1}{k}}\big(|t-t_0|^4+\kappa_\infty((t,x),(t_0,x_0))\big)^{\frac{1}{2}} \ \leq \ c\ell_{k,r}^{1-\frac{1}{k}}r^2.
\end{align*}
Recalling the expression of $\ell_{k,r}$ in \eqref{ell}, we find
\[
c\ell_{k,r}^{1-\frac{1}{k}}r^2 \ = \ c\Bigg(\frac{1}{r^{\frac{4}{k}}-\frac{r^p}{\sqrt{k}}}\Bigg)^{k-1}r^2 \ \overset{k\rightarrow\infty}{\longrightarrow} \ \frac{c}{r^2}.
\]
We conclude that the quantity $c\ell_{k,r}^{1-\frac{1}{k}}r^2$ is bounded in $k$ by some constant depending on $r$.

\vspace{4mm}

\noindent\textsc{Substep I-c.}
By \eqref{G_Propriety4} and the smooth variational principle (Theorem \ref{T:SmoothVP}) with $G:=G_{n,\eps,k,r}$ and $\lambda:=Mr^p$, we deduce that for every $\delta>0$ there exist $(\bar t,\bar x)\in[t_0,T]\times C([0,T];\R^d)$ and $\varphi\colon[0,T]\times C([0,T]\times\R^d)\rightarrow[0,+\infty]$ in \eqref{varphi} satisfying items i)-ii)-iii)-iv) of Theorem \ref{T:SmoothVP}. Moreover, the restriction of $\varphi$ to $[\bar t,T]\times C([0,T];\R^d)$ belongs to $C^{1,2}([\bar t,T]\times C([0,T];\R^d))$ and its derivatives are bounded by $c\delta$, for some constant $c\geq0$, independent of $n,\eps,k,r,\delta$. Moreover, from properties 1)-3)-4) of $G_{n,\eps,k,r}$ and item ii) of Theorem \ref{T:SmoothVP}, it follows that
\begin{equation}\label{Estimate}
|\bar t-t_0|^4+\kappa_\infty\big((\bar t,\bar x),(t_0,x_0)\big) \ \leq \ r^4.
\end{equation}
In particular
\[
|\bar t - t_0| \ \leq \ r.
\]
Recalling that $t_0<T$, we see that $\bar t<T$ for $r$ small enough.\\
Using again \eqref{Estimate} and \eqref{Estimate_rho_infty}, we obtain
\begin{equation}\label{x-x_0}
\big\|\bar x(\cdot\wedge\bar t) - x_0(\cdot\wedge t_0)\big\|_T \ \leq \ r.
\end{equation}

\vspace{4mm}

\noindent\textsc{Substep I-d}. We apply the definition of viscosity subsolution of \eqref{HJB_exp} to $\tilde u_1$ at the point $(\bar t,\bar x)$ with test function $(t,x)\mapsto\tilde v_{n,\eps}(t,x)+h_{k,r}(t,x)+\varphi(t,x)$, where (recall from \eqref{G_Property5} that $\psi_{k,r}$ does not appear because of \eqref{Estimate})
\[
h_{k,r}(t,x) \ = \ 3M\sqrt{k}\bigg\{\bigg(\frac{1}{\ell_{k,r}}+|t-t_0|^4+\kappa_\infty\big((t,x),(t_0,x_0)\big)\bigg)^{\frac{1}{k}} - \frac{1}{\ell^{\frac{1}{k}}}\bigg\}.
\]
Then, we obtain
\begin{align*}
&\tilde u_1(\bar t,\bar x) \\
&\leq \ \partial_t^H h_{k,r}(\bar t,\bar x) + \sup_{a\in A}\bigg\{\big\langle b(\bar t,\bar x,a),\partial^V_x h_{k,r}(\bar t,\bar x)\big\rangle + \frac{1}{2}\text{tr}\big[(\sigma\sigma\trans)(\bar t,\bar x,a)\partial^V_{xx}h_{k,r}(\bar t,\bar x)\big]\bigg\} \\
&+ \partial_t^H \varphi(\bar t,\bar x) + \sup_{a\in A}\bigg\{\big\langle b(\bar t,\bar x,a),\partial^V_x \varphi(\bar t,\bar x)\big\rangle + \frac{1}{2}\text{tr}\big[(\sigma\sigma\trans)(\bar t,\bar x,a)\partial^V_{xx}\varphi(\bar t,\bar x)\big]\bigg\} \\
&+ \partial_t^H \tilde v_{n,\eps}(\bar t,\bar x) + \sup_{a\in A}\bigg\{\big\langle b(\bar t,\bar x,a),\partial^V_x \tilde v_{n,\eps}(\bar t,\bar x)\big\rangle + \frac{1}{2}\text{tr}\big[(\sigma\sigma\trans)(\bar t,\bar x,a)\partial^V_{xx}\tilde v_{n,\eps}(\bar t,\bar x)\big] + \tilde f(\bar t,\bar x,a)\bigg\}.
\end{align*}
Recalling that $\tilde v_{n,\eps}$ is a classical solution of equation \eqref{HJB_exp_n,eps}, we find
\begin{align}
&(\tilde u_1 - \tilde v_{n,\eps})(\bar t,\bar x) \ \leq \ - \frac{1}{2}\eps^2\text{e}^{\bar t-t_0}\textup{tr}\big[\partial_{yy}\bar v_{n,\eps}(\bar t,y_n^{\bar t,\bar x})\big]
\notag \\
&+ \partial_t^H h_{k,r}(\bar t,\bar x) + \sup_{a\in A}\bigg\{\big\langle b(\bar t,\bar x,a),\partial^V_x h_{k,r}(\bar t,\bar x)\big\rangle + \frac{1}{2}\text{tr}\big[(\sigma\sigma\trans)(\bar t,\bar x,a)\partial^V_{xx}h_{k,r}(\bar t,\bar x)\big]\bigg\} \notag \\
&+ \partial_t^H \varphi(\bar t,\bar x) + \sup_{a\in A}\bigg\{\big\langle b(\bar t,\bar x,a),\partial^V_x \varphi(\bar t,\bar x)\big\rangle + \frac{1}{2}\text{tr}\big[(\sigma\sigma\trans)(\bar t,\bar x,a)\partial^V_{xx}\varphi(\bar t,\bar x)\big]\bigg\}
\notag\\
&+ \sup_{a\in A}\Big\{\big\langle b(\bar t,\bar x,a) - b_n(\bar t,\bar x,a),\partial^V_x \tilde v_{n,\eps}(\bar t,\bar x)\big\rangle + \tilde f(\bar t,\bar x,a) - \tilde f_n(\bar t,\bar x,a)\Big\},
\label{eq:visc1new}
\end{align}
where $y_n^{\bar t,\bar x}$ is given by \eqref{y_n^t,x} with $t$ and $x$ replaced respectively by $\bar t$ and $\bar x$. By item ii) of Theorem \ref{T:SmoothVP} and the fact that $\varphi\geq0$, we have
\begin{align*}
(u_1 - v_{n,\eps})(t_0,x_0) \ = \ (\tilde u_1 - \tilde v_{n,\eps})(t_0,x_0) \ &\leq \ \big(\tilde u_1 - \tilde v_{n,\eps}-\varphi\big)(\bar t,\bar x) \ \leq \ \big(\tilde u_1 - \tilde v_{n,\eps}\big)(\bar t,\bar x).
\end{align*}
In addition, using the boundedness of $b$ and $\sigma$, and recalling from Theorem \ref{T:SmoothVP} that the derivatives of $\varphi$ are bounded by $c\delta$, we deduce that there exists a constant $\Lambda\geq0$, independent of $n,\eps,k,r,\delta$, such that
\[
\partial_t^H \varphi(\bar t,\bar x) + \sup_{a\in A}\bigg\{\big\langle b(\bar t,\bar x,a),\partial^V_x \varphi(\bar t,\bar x)\big\rangle + \frac{1}{2}\text{tr}\big[(\sigma\sigma\trans)(\bar t,\bar x,a)\partial^V_{xx}\varphi(\bar t,\bar x)\big]\bigg\} \ \leq \ \Lambda\delta.
\]
Similarly, recall from property 5) of \textsc{Substep I-b} that the derivatives of $h_{k,r}$ are bounded by $\frac{D_r}{\sqrt{k}}$ on the set $\{|t-t_0|^4+\kappa_\infty((t,x),(t_0,x_0))\leq r^4\}$. Since by \eqref{Estimate} it holds that $|\bar t-t_0|^4+\kappa_\infty((\bar t,\bar x),(t_0,x_0))\leq r^4$, there exists some constant $\Xi_r\geq0$, independent of $n,\eps,k,\delta$, but possibly depending on $r$, such that
\[
\partial_t^H h_{k,r}(\bar t,\bar x) + \sup_{a\in A}\bigg\{\big\langle b(\bar t,\bar x,a),\partial^V_x h_{k,r}(\bar t,\bar x)\big\rangle + \frac{1}{2}\text{tr}\big[(\sigma\sigma\trans)(\bar t,\bar x,a)\partial^V_{xx}h_{k,r}(\bar t,\bar x)\big]\bigg\} \ \leq \ \frac{\Xi_r}{\sqrt{k}}.
\]
Plugging the last three estimates into \eqref{eq:visc1new} we get, using also estimates \eqref{EstimateKrylov} and \eqref{EstimateLipschitz1},
\begin{align}\label{Proof_u_1}
&(u_1 - v_{n,\eps})(t_0,x_0) \ \leq \ \frac{1}{2}\eps^2\text{e}^{\bar t-t_0}\,d\,d_n\,\bar C_n\,\textup{e}^{\bar C_n(T-\bar t)}\,\big(1 + \big|y_n^{\bar t,\bar x}\big|\big)^{3q} + \Lambda\delta+\frac{\Xi_r}{\sqrt{k}} \\
&+ \text{e}^{\bar t-t_0}\Big\{\bar L\sup_{a\in A}\big\{\big|b(\bar t,\bar x,a) - b_n(\bar t,\bar x,a)\big|\big\} + \sup_{a\in A}\big\{\big|f(\bar t,\bar x,a) - f_n(\bar t,\bar x,a)\big|\big\}\Big\}. \notag
\end{align}
Recall that $y_n^{\bar t,\bar x}$ is given by \eqref{y_n^t,x}. Then, from \eqref{Dy-x1}-\eqref{Dy-x2} we see that there exists a constant $c\geq0$, independent of $n,\eps,\delta$, such that
\[
\big|(y_n^{\bar t,\bar x})_j\big| \ \leq \ c\,\big\|\bar x(\cdot\wedge\bar t)\big\|_T, \qquad \forall\,j=1,\ldots,d_n.
\]
Hence
\begin{align*}
&\big|y_n^{\bar t,\bar x}\big| = \sqrt{\sum_j \big|(y_n^{\bar t,\bar x})_j\big|^2} \leq c\sqrt{d_n}\big\|\bar x(\cdot\wedge\bar t)\big\|_T \leq c\sqrt{d_n}\big\|x_0(\cdot\wedge t_0)\big\|_T \\
&+ c\sqrt{d_n}\big\|\bar x(\cdot\wedge\bar t) - x_0(\cdot\wedge t_0)\big\|_T \ \leq c\sqrt{d_n}\big\|x_0(\cdot\wedge t_0)\big\|_T + c\sqrt{d_n}\frac{1}{k},
\end{align*}
where the last inequality follows from \eqref{x-x_0}. Therefore, from \eqref{Proof_u_1} we obtain
\begin{align*}
&(u_1 - v_{n,\eps})(t_0,x_0) \ \leq \frac{1}{2}\eps^2\text{e}^{\bar t-t_0}dd_n\bar C_n\textup{e}^{\bar C_n T}\bigg(1 + c\sqrt{d_n}\big\|x_0(\cdot\wedge t_0)\big\|_T + c\sqrt{d_n}\frac{1}{k}\bigg)^{3q} \\
&+\Lambda\delta+\frac{\Xi_r}{\sqrt{k}}+ \text{e}^{\bar t-t_0}\Big\{\bar L\sup_{a\in A}\big\{\big|b(\bar t,\bar x,a) - b_n(\bar t,\bar x,a)\big|\big\} + \sup_{a\in A}\big\{\big|f(\bar t,\bar x,a) - f_n(\bar t,\bar x,a)\big|\big\}\Big\}.
\end{align*}
Now, notice that
\begin{align*}
&\sup_{|s-s'|\leq2^{-n}} \big|\bar x(s\wedge\bar t) - \bar x(s'\wedge\bar t)\big| \ \leq \ \sup_{|s-s'|\leq2^{-n}} \big|\bar x(s\wedge\bar t) - x_0(s\wedge t_0) + x_0(s\wedge t_0) - x_0(s'\wedge t_0) \notag \\
&+ x_0(s'\wedge t_0) - \bar x(s'\wedge\bar t)\big| \ \leq \ 2\big\|\bar x(\cdot\wedge \bar t) - x_0(\cdot\wedge t_0)\big\|_T + \sup_{|s-s'|\leq2^{-n}} \big|x_0(s\wedge t_0) - x_0(s'\wedge t_0)\big| \notag \\
&\leq \ 2r + \sup_{|s-s'|\leq2^{-n}} \big|x_0(s\wedge t_0) - x_0(s'\wedge t_0)\big|,
\end{align*}
where the last inequality follows from \eqref{x-x_0}. Then, using estimate \eqref{h_n-h_proof5} with $h$ and $h_n$ replaced respectively by $b$ and $b_n$ or $f$ and $f_n$, we get
\begin{align}\label{Proof_u_2}
(u_1 - v_{n,\eps})(t_0,x_0) &\leq \frac{1}{2}\eps^2\text{e}^{\bar t-t_0}dd_n\bar C_n\textup{e}^{\bar C_n T}\bigg(1 + c\sqrt{d_n}\big\|x_0(\cdot\wedge t_0)\big\|_T + c\sqrt{d_n}\frac{1}{k}\bigg)^{3q} + \Lambda\delta+\frac{\Xi_r}{\sqrt{k}} \notag \\
&\quad \ + \text{e}^{\bar t-t_0}(1+\bar L)\bigg(6Kr + 3K\sup_{|s-s'|\leq2^{-n}} \big|x_0(s\wedge t_0) - x_0(s'\wedge t_0)\big| \\
&\quad \ + \frac{2K}{\sqrt{d(2^n+1)}} + \int_0^\infty \frac{2}{\sqrt{2\pi}}\,\textup{e}^{-\frac{1}{2}r^2}\,w\big((\bar t+r/n)\wedge T-\bar t\big)\,dr\bigg), \notag
\end{align}
where we recall that $w$ is the modulus of continuity of $b$ and $f$ with respect to the time variable.\\
Now we pass to the limit. First, we send $\eps\rightarrow0^+$ so the first term on the right-hand side of \eqref{Proof_u_2} goes to zero. Then, we send $n\rightarrow+\infty$ so the second term on the second line in \eqref{Proof_u_2} goes to zero, together with the two terms on the third line. Afterwards, we let $k\rightarrow\infty$ and $\delta\rightarrow0^+$. Finally, we send $r\rightarrow0^+$, obtaining
\[
(u_1 - v)(t_0,x_0) \ \leq \ 0,
\]
which gives a contradiction to \eqref{Contradiction}.

\vspace{4mm}

\noindent\textsc{Step II.} \emph{Proof of $v\leq u_2$.} It is enough to show that
\begin{equation}\label{ClaimStepII}
u_2(t,x) \ \geq \ \sup_{s\in[t,T],a\in A}\E\bigg[\int_t^s f\big(r,X^{t,x,a},a\big)\,dr + v\big(s,X^{t,x,a}\big)\bigg],
\end{equation}
for every $(t,x)\in[0,T]\times C([0,T];\R^d)$, where $X^{t,x,a}$ corresponds to the process $X^{t,x,\alpha}$ with $\alpha\equiv a$. As a matter of fact, it holds that
\[
\sup_{s\in[t,T],a\in A}\E\bigg[\int_t^s f\big(r,X^{t,x,a},a\big)\,dr + v\big(s,X^{t,x,a}\big)\bigg] \ \geq \ v(t,x),
\]
where the validity of the above inequality can be shown simply taking $s=t$ in the left-hand side. For every fixed $s\in[0,T]$, $a\in A$, set
\[
v^{s,a}(t,x) \ := \ \E\bigg[\int_t^s f\big(r,X^{t,x,a},a\big)\,dr + v\big(s,X^{t,x,a}\big)\bigg], \qquad \forall\,(t,x)\in[0,s]\times C([0,T];\R^d).
\]
Notice that applying Proposition \ref{P:Value} with $g,T,A$ replaced by $v(s,\cdot),s,\{a\}$, respectively, we deduce that $v^{s,a}$ is bounded, jointly continuous on $[0,s]\times C([0,T];\R^d)$, and there exists a constant $\hat c\geq0$ $($depending only on $T$ and $K$$)$ such that
\[
|v^{s,a}(t,x) - v^{s,a}(t',x')| \ \leq \ \hat c\big(|t-t'|^{1/2} + \|x(\cdot\wedge t) - x'(\cdot\wedge t')\|_T\big),
\]
for all $(s,a)\in[0,T]\times A$, $(t,x),(t',x')\in[0,s]\times C([0,T];\R^d)$. By the boundedness of $f$ and \eqref{Value_Lipschitz}, we also have
\begin{equation}\label{s-Continuity}
|v^{s,a}(t,x) - v^{s',a}(t,x)| \ \leq \ K|s'-s| + c|s'-s|^{1/2},
\end{equation}
for all $a\in A$, $s,s'\in[0,T]$, $(t,x)\in[0,s\wedge s']\times C([0,T];\R^d)$, with $c$ being the same constant appearing in \eqref{Value_Lipschitz}.

In order to prove \eqref{ClaimStepII}, we proceed by contradiction and suppose that there exist $(t_0,x_0)\in[0,T]\times C([0,T];\R^d)$, $s_0\in[t_0,T]$, $a_0\in A$, such that
\[
(v^{s_0,a_0} - u_2)(t_0,x_0) \ > \ 0.
\]
It holds that $t_0<T$, otherwise $t_0=s_0=T$ and $u_2(T,x_0)\geq g(x_0)=v^{T,a_0}(T,x_0)$. Moreover, we can suppose, without loss of generality, that $t_0<s_0$. As a matter of fact, by \eqref{s-Continuity} and the fact that $t_0<T$, there exists $s_1\in(t_0,T]$ such that $(v^{s_1,a_0}-u_2)(t_0,x_0)>0$. Therefore, it is enough to consider $s_1$ in place of $s_0$. For this reason, in the sequel we assume that $t_0<s_0$.

Now, consider the sequences $\{b_n\}_n$, $\{f_n\}_n$, $\{\hat v_n(s_0,\cdot)\}_n$, $\{v_n^{s_0,a_0}\}_n$ introduced in Theorem \ref{T:CylindrApprox2}, with $v_n^{s_0,a_0}$ being a classical solution of equation \eqref{HJB_n^s_0,a_0}.

We set $\tilde u_2(t,x):=\text{e}^{t-t_0}\,u_2(t,x)$, for all $(t,x)\in[0,T]\times C([0,T];\R^d)$, and we define similarly $\tilde v_n^{s_0,a_0}$, $\tilde f$, $\tilde f_n$. We also define $\tilde g(x):=\text{e}^{t-t_0}\, g(x)$ and $\tilde v_n(s_0,x):=\text{e}^{s_0-t_0}\,\hat v_n(s_0,x)$, for all $x\in C([0,T];\R^d)$. Notice that $\tilde u_2$ is a (path-dependent) viscosity supersolution of the following path-dependent partial differential equation:
\[
\begin{cases}
\vspace{2mm}
\partial_t^H \tilde u_2(t,x) + \big\langle b(t,x,a_0),\partial^V_x \tilde u_2(t,x)\big\rangle + \dfrac{1}{2}\text{tr}\big[(\sigma\sigma\trans)(t,x,a_0)\partial^V_{xx}\tilde u_2(t,x)\big] \\
\vspace{2mm}+\,\tilde f(t,x,a_0) = \tilde u_2(t,x), &\hspace{-4.5cm}(t,x)\in[0,T)\times C([0,T];\R^d), \\
\tilde u_2(T,x) = \tilde g(x), &\hspace{-4.5cm}x\in C([0,T];\R^d).
\end{cases}
\]
Similarly, $\tilde v_n^{s_0,a_0}\in C_{\textup{pol}}^{1,2}([0,T]\times C([0,T];\R^d))$ and is a classical solution of the following equation:
\[
\hspace{-.4cm}\begin{cases}
\vspace{2mm}
\partial_t^H \tilde v_n^{s_0,a_0}(t,x) + \big\langle b_n(t,x,a_0),\partial^V_x \tilde v_n^{s_0,a_0}(t,x)\big\rangle + \dfrac{1}{2}\textup{tr}\big[(\sigma\sigma\trans)(t,x,a_0)\partial^V_{xx}\tilde v_n^{s_0,a_0}(t,x)\big] \\
\vspace{2mm}+\,\tilde f_n(t,x,a_0) = \tilde v_n^{s_0,a_0}(t,x), &\hspace{-5cm}(t,x)\in[0,s_0)\times C([0,T];\R^d), \\
\tilde v_n^{s_0,a_0}(s_0,x) = \tilde v_n(s_0,x), &\,\hspace{-5cm}x\in C([0,T];\R^d).
\end{cases}
\]
We begin noting that by \eqref{u_1_u_2} and item 4) of Theorem \ref{T:CylindrApprox2}, there exist constants $M\geq0$ and $p\in(0,1/2]$, independent of $n,\eps$, such that
\[
\big|\big(\tilde v_n^{s_0,a_0}-\tilde u_2\big)(t,x) - \big(\tilde v_n^{s_0,a_0}-\tilde u_2\big)(t',x')\big| \ \leq \ M\big(|t-t'|^p + \|x(\cdot\wedge t) - x'(\cdot\wedge t')\|_T^p\big),
\]
for all $(t,x),(t',x')\in[0,s_0]\times C([0,T];\R^d)$, with $|t-t'|\leq1$ and $\|x(\cdot\wedge t)-x'(\cdot\wedge t')\|_T\leq1$. To alleviate notation, we use the same symbols $M$ and $p$ as in \textsc{Step I}. Now, we proceed along the same lines as in \textsc{Substep I-b}. In particular, we consider the function
\begin{align*}
G_{n,k,r}(t,x) &:= \big(\tilde v_n^{s_0,a_0}-\tilde u_2\big)(t,x) - 3M\sqrt{k}\bigg\{\bigg(\frac{1}{\ell_{k,r}}+|t-t_0|^4+\kappa_\infty\big((t,x),(t_0,x_0)\big)\bigg)^{\frac{1}{k}} - \frac{1}{\ell_{k,r}^{\frac{1}{k}}}\bigg\} \\
&\quad \ - \psi_{k,r}\Big(|t-t_0|^4 + \kappa_\infty\big((t,x),(t_0,x_0)\big)\Big).
\end{align*}
As in \textsc{Substep I-b} we can prove that, for $k\geq C_r$,
\[
\sup_{(t,x)\in[0,s_0]\times C([0,T];\R^d)} G_{n,k,r}(t,x) \ \leq \ \big(\tilde v_n^{s_0,a_0}-\tilde u_2\big)(t_0,x_0) + Mr^p.
\]
Then, by the smooth variational principle (Theorem \ref{T:SmoothVP}) on $[0,s_0]\times C([0,T];\R^d)$ with $G:=G_{n,k,r}$ and $\lambda:=Mr^p$, we deduce that for every $\delta>0$ there exist $(\bar t,\bar x)\in[t_0,s_0]\times C([0,T];\R^d)$ and $\varphi\colon[0,T]\times C([0,T]\times\R^d)\rightarrow[0,+\infty]$ in \eqref{varphi} satisfying items i)-ii)-iii)-iv) of Theorem \ref{T:SmoothVP}. Moreover, the restriction of $\varphi$ to $[\bar t,T]\times C([0,T];\R^d)$ belongs to $C^{1,2}([\bar t,T]\times C([0,T];\R^d))$ and its derivatives are bounded by $c\delta$, for some constant $c\geq0$, independent of $n,k,r,\delta$. Moreover, from the properties of $G_{n,k,r}$ and item ii) of Theorem \ref{T:SmoothVP}, it follows that
\[
|\bar t - t_0| \ \leq \ r, \qquad \big\|\bar x(\cdot\wedge\bar t) - x_0(\cdot\wedge t_0)\big\|_T \ \leq \ r.
\]
Recalling that $t_0<s_0$, by the first inequality we deduce that $\bar t<s_0$ for $r$ small enough.

Now, we can proceed along the same lines as in \textsc{Substep I-d} to get a contradiction and conclude the proof. We only notice that in order to use the viscosity supersolution property of $\tilde u_2$ at the point $(\bar t,\bar x)$ we need to extend $v_n^{s_0,a_0}$ from $[0,s_0]\times C([0,T];\R^d)$ to $[0,T]\times C([0,T];\R^d)$ in such a way that the extension is still smooth. We can do this extending by reflection (see \cite{Hestenes}), namely defining the function
\[
v_n^{1,s_0,a_0}(t,x) =
\begin{cases}
v_n^{s_0,a_0}(t,x), \qquad &(t,x)\in[0,s_0]\times C([0,T];\R^d), \\
4v_n^{s_0,a_0}((3s_0-t)/2,x)-3v_n^{s_0,a_0}(2s_0-t,x), \qquad &(t,x)\in[s_0,2s_0]\times C([0,T];\R^d).
\end{cases}
\]
Notice that $v_n^{1,s_0,a_0}$ is non-anticipative. If $2s_0\geq T$ we have finished, otherwise we extend again $v_n^{1,s_0,a_0}$ and after a finite number of extensions we find a map defined on the entire space $[0,T]\times C([0,T];\R^d)$.
\end{proof}

\noindent Finally, we can state the following uniqueness result.

\begin{Corollary}\label{C:Uniq}
Suppose that Assumptions \ref{AssA}, \ref{AssB}, \ref{AssC} hold. Then, the value function $v$ in \eqref{Value} is the unique (path-dependent) viscosity solution of equation \eqref{HJB} in the class of functions satisfying \eqref{u_1_u_2}.
\end{Corollary}
\begin{proof}[\textbf{Proof.}]
If $u$ is another (path-dependent) viscosity solution of equation \eqref{HJB} satisfying \eqref{u_1_u_2}, then, by Theorem \ref{T:Comparison}, we get the two following inequalities:
\[
u \ \leq \ v, \qquad\qquad v \ \leq \ u,
\]
from which the claim follows.
\end{proof}

\begin{appendices}

\section{Pathwise derivatives and functional It\^o's\\ formula}
\label{App:PathwiseDeriv}

In the present appendix, we briefly recall the definitions of pathwise (or functional) derivatives following \cite[Section 2]{CR19}, for which we refer for more details.\\
As we follow the standard approach (as it was introduced in the seminal paper \cite{dupire}), in order to introduce the pathwise derivatives for a map $u\colon[t_0,T]\times C([0,T];\R^d)\rightarrow\R$, $t_0\in[0,T)$, we firstly define them for a map $\hat u\colon[t_0,T]\times D([0,T];\R^d)\rightarrow\R$, with $D([0,T];\R^d)$ being the set of c\`adl\`ag paths, endowed with the supremum norm $\|\hat x\|_T=\sup_{s\in[0,T]}|\hat x(s)|$, for every $\hat x\in D([0,T];\R^d)$. We also define on $[0,T]\times D([0,T];\R^d)$ the pseudometric $\hat d_\infty\colon([0,T]\times D([0,T];\R^d))^2\rightarrow[0,\infty)$ as
\[
\hat d_\infty\big((t,\hat x),(t',\hat x')\big) \ := \ |t - t'| + \big\|\hat x(\cdot\wedge t) - \hat x'(\cdot\wedge t')\big\|_T.
\]
On $[t_0,T]\times D([0,T];\R^d)$ we consider the restriction of $\hat d_\infty$, which we still denote by the same symbol.

\begin{Definition}
Consider a map $\hat u\colon[t_0,T]\times D([0,T];\R^d)\rightarrow\R$, $t_0,\in[0,T)$.
\begin{enumerate}[\upshape (i)]
\item For every $(t,\hat x)\in[t_0,T]\times D([0,T];\R^d)$, with $t<T$, the horizontal derivative of $\hat u$ at $(t,\hat x)$ is defined as (if the limit exists)
\[
\partial_t^H\hat u(t,\hat x) \ := \ \lim_{\delta\rightarrow0^+}\frac{\hat u(t+\delta,\hat x(\cdot\wedge t)) - \hat u(t,\hat x(\cdot\wedge t))}{\delta}.
\]
At $t=T$, it is defined as
\[
\partial_t^H\hat u(T,\hat x) \ := \ \lim_{t\rightarrow T^-} \partial_t^H\hat u(t,\hat x).
\]
\item For every $(t,\hat x)\in[t_0,T]\times D([0,T];\R^d)$, the vertical derivatives of first and second-order of $\hat u$ at $(t,\hat x)$ are defined as (if the limits exist)
\begin{align*}
\partial_{x_i}^V\hat u(t,\hat x) \ &:= \ \lim_{h\rightarrow0}\frac{\hat u(t,\hat x + he_i1_{[t,T]}) - \hat u(t,\hat x)}{h}, \\
\partial_{x_i x_j}^V\hat u(t,\hat x) \ &:= \ \partial_{x_j}^V(\partial_{x_i}^V\hat u)(t,\hat x),
\end{align*}
where $e_1,\ldots,e_d$ is the standard orthonormal basis of $\R^d$. We also denote $\partial_x^V\hat u=(\partial_{x_1}^V\hat u,\ldots,\partial_{x_d}^V\hat u)$ and $\partial_{xx}^V\hat u=(\partial_{x_i x_j}^V\hat u)_{i,j=1,\ldots,d}$.
\end{enumerate}
\end{Definition}

\begin{Definition}
$C^{1,2}([t_0,T]\times D([0,T];\R^d))$, $t_0,\in[0,T)$, is the set of continuous real-valued maps $\hat u$ defined on $([t_0,T]\times D([0,T];\R^d),\hat d_\infty)$, such that $\partial_t^H\hat u$, $\partial_x^V\hat u$, $\partial_{xx}^V\hat u$ exist everywhere on $([t_0,T]\times D([0,T];\R^d),\hat d_\infty)$ and are continuous.
\end{Definition}

\noindent We can now define the pathwise derivatives for a map $u\colon[t_0,T]\times C([0,T];\R^d)\rightarrow\R$. To this end, the following consistency property plays a crucial role.

\begin{Lemma}\label{L:Consistency}
If $\hat u_1,\hat u_2\in C^{1,2}([t_0,T]\times D([0,T];\R^d))$ coincide on continuous paths, namely
\[
\hat u_1(t,x) \ = \ \hat u_2(t,x), \qquad \forall\,(t,x)\in[t_0,T]\times C([0,T];\R^d),
\]
then the same holds for their pathwise derivatives: for every $(t,x)\in[t_0,T]\times C([0,T];\R^d)$,
\begin{align*}
\partial_t^H\hat u_1(t,x) \ &= \ \partial_t^H\hat u_2(t,x), \\
\partial^V_x \hat u_1(t,x) \ &= \ \partial^V_x \hat u_2(t,x), \\
\partial^V_{xx} \hat u_1(t,x) \ &= \ \partial^V_{xx} \hat u_2(t,x).
\end{align*}
\end{Lemma}
\begin{proof}[\textbf{Proof.}]
See \cite[Lemma 2.1]{CR19}.
\end{proof}

\noindent We can now given the following definition.

\begin{Definition}\label{D:C^1,2}
$C^{1,2}([t_0,T]\times C([0,T];\R^d))$, $t_0,\in[0,T)$, is the set of continuous real-valued maps $u$ defined on $([t_0,T]\times C([0,T];\R^d),d_\infty)$, for which there exists $\hat u\in C^{1,2}([t_0,T]\times D([0,T];\R^d))$ such that
\[
u(t,x) \ = \ \hat u(t,x), \qquad \forall\,(t,x)\in[t_0,T]\times C([0,T];\R^d).
\]
We also define, for every $(t,x)\in[t_0,T]\times C([0,T];\R^d)$,
\begin{align*}
\partial_t^H u(t,x) \ &:= \ \partial_t^H\hat u(t,x), \\
\partial^V_x u(t,x) \ &:= \ \partial^V_x \hat u(t,x), \\
\partial^V_{xx} u(t,x) \ &:= \ \partial^V_{xx} \hat u(t,x).
\end{align*}
\end{Definition}

\begin{Remark}
Thanks to the consistency property stated in Lemma \ref{L:Consistency}, the definition of pathwise derivatives of $u$ does not depend on the map $\hat u$ appearing in Definition \ref{D:C^1,2}.
\end{Remark}

\noindent In the present paper we also need to consider the following subset of $C^{1,2}([t_0,T]\times C([0,T];\R^d))$.

\begin{Definition}\label{D:C_pol^1,2}
We denote by $C_{\textup{pol}}^{1,2}([t_0,T]\times C([0,T];\R^d))$ the set of $u\in C^{1,2}([t_0,T]\times C([0,T];\R^d))$ such that $u$, $\partial_t^H u$, $\partial^V_x u$, $\partial^V_{xx} u$ satisfy a polynomial growth condition: there exist constants $M\geq0$ and $q\geq0$ such that
\[
\big|\partial_t^H u(t,x)\big| + \big|\partial_x^V u(t,x)\big| + \big|\partial_{xx}^V u(t,x)\big| \ \leq \ M\,\big(1 + \|x\|_t^q\big),
\]
for all $(t,x)\in[t_0,T]\times C([0,T];\R^d)$.
\end{Definition}

\noindent Finally, we state the so-called functional It\^o formula.

\begin{Theorem}[Functional It\^o's formula]\label{T:Ito}
Let $u\in C^{1,2}([t_0,T]\times C([0,T];\R^d))$. Let also $(\Omega,\mathcal F,(\mathcal F_t)_{t\in[t_0,T]},\P)$ be a filtered probability space, with $(\mathcal F_t)_{t\in[t_0,T]}$ satisfying the usual conditions, on which a $d$-dimensional continuous semimartingale $X=(X_t)_{t\in[t_0,T]}$ is defined, with $X=(X^1,\ldots,X^d)$. Then, it holds that
\begin{align}\label{Ito_formula}
u(t,X) \ &= \ u(t_0,X) + \int_{t_0}^t \partial_t^H u(s,X)\,ds + \frac{1}{2}\sum_{i,j=1}^d \int_{t_0}^t \partial^V_{x_i x_j} u(s,X)\,d[X^i,X^j]_s \\
&\quad \ + \sum_{i=1}^d \int_{t_0}^t \partial^V_{x_i}u(s,X)\,dX_s^i, \hspace{2.1cm} \text{for all }\,t_0\leq t\leq T,\,\,\P\text{-a.s.} \notag
\end{align}
\end{Theorem}
\begin{proof}[\textbf{Proof.}]
See \cite[Theorem 2.2]{CR19}.
\end{proof}

\section{Cylindrical approximations}
\label{App:CylindricalApprox}

\subsection{The deterministic calculus via regularization}

In the present appendix, we need to consider ``cylindrical'' maps defined on $C([0,T];\R^d)$, namely maps depending on a path $x\in C([0,T];\R^d)$ only through a finite number of integrals with respect to $x$. An integral with respect to $x$ can be formally written as ``$\int_{[0,T]}\varphi(s)dx(s)$''. In order to give a meaning to the latter notation, it is useful to notice that we look for a deterministic integral which coincides with the It\^o integral when $x$ is replaced by an It\^o process (such a property will be exploited in the sequel). This is the case if we interpret ``$\int_{[0,T]}\varphi(s)dx(s)$'' as the deterministic version of the forward integral, which we now introduce and denote by $\int_{[0,T]}\varphi(s)d^-x(s)$. For more details on such an integral and, more generally, on the deterministic calculus via regularization we refer to \cite[Section 3.2]{digirfabbrirusso13} and \cite[Section 2.2]{cosso_russoStrict}. The only difference with respect to \cite{digirfabbrirusso13} and \cite{cosso_russoStrict} being that here we consider $d$-dimensional paths (with $d$ possibly greater than $1$), even though, as usual, we work component by component, therefore relying on the one-dimensional theory.

\begin{Definition}\label{D:DetInt}
Let $x\colon[0,T]\rightarrow\R^d$ and $\varphi\colon[0,T]\rightarrow\R$ be c\`{a}dl\`{a}g functions. When the following limit
\[
\int_{[0,T]}\varphi(s)\,d^- x(s) \ := \ \lim_{\eps\rightarrow0^+}\int_0^T \varphi(s)\,\frac{x(T\wedge(s+\eps))-x(s)}{\eps}\,ds
\]
exists and it is finite, we denote it by $\int_{[0,T]} \varphi(s)\,d^- x(s)$ and call it \textbf{forward integral of $\varphi$ with respect to $x$}.
\end{Definition}

When $\varphi$ is continuous and of bounded variation, an integration by parts formula provides an explicit representation of the forward integral of $\varphi$ with respect to $x$.

\begin{Proposition}
\label{P:IntbyParts}
Let $x\colon[0,T]\rightarrow\R^d$ be a c\`{a}dl\`{a}g function and let $\varphi\colon[0,T]\rightarrow\R$ be c\`{a}dl\`{a}g and of bounded variation. The following \textbf{integration by parts formula} holds:
\begin{equation}\label{IntbyParts}
\int_{[0,T]}  \varphi(s)\,d^- x(s) \ = \ \varphi(T-) \,x(T) - \varphi(0)x(0)  - \int_{(0,T]} x(s)\,d\varphi(s),
\end{equation}
where $\int_{(0,T]} x(s)\,d\varphi(s)$ is a Lebesgue-Stieltjes integral on $(0,T]$.
\end{Proposition}
\begin{proof}[\textbf{Proof.}]
We have
\[
\int_0^T \varphi(s)\frac{x(T\wedge(s+\eps))-x(s)}{\eps}ds = \int_0^{T-\eps} \varphi(s)\frac{x(s+\eps)-x(s)}{\eps}ds + \int_{T-\eps}^T \varphi(s)\frac{x(T)-x(s)}{\eps}ds
\]
Notice that
\[
\int_0^{T-\eps} \varphi(s)x(s+\eps)\,ds \ = \ \int_\eps^T \varphi(s-\eps)x(s)\,ds \ = \ \int_0^T \varphi(s-\eps)x(s)\,ds,
\]
where in the last equality we set $\varphi(s):=0$ for $s<0$. Hence
\begin{align*}
&\int_0^T \varphi(s)\,\frac{x(T\wedge(s+\eps))-x(s)}{\eps}\,ds \\
&= \ \int_0^T \frac{\varphi(s-\eps)-\varphi(s)}{\eps}x(s)\,ds + \frac{1}{\eps}\int_{T-\eps}^T \varphi(s)x(s)\,ds  + \frac{1}{\eps}\int_{T-\eps}^T \varphi(s)\big(x(T)-x(s)\big)\,ds \\
&= \ \int_0^T \frac{\varphi(s-\eps)-\varphi(s)}{\eps}x(s)\,ds + \frac{1}{\eps}\int_{T-\eps}^T \varphi(s)x(T)\,ds.
\end{align*}
Since $\varphi$ is c\`{a}dl\`{a}g, we have
\begin{equation}\label{IbP_1}
\frac{1}{\eps}\int_{T-\eps}^T \varphi(s)x(T)\,ds \ \overset{\eps\rightarrow0^+}{\longrightarrow} \ \varphi(T-)x(T). \qquad\qquad
\end{equation}
On the other hand, by Fubini's theorem, we have
\begin{align*}
&\int_0^T \frac{\varphi(s-\eps)-\varphi(s)}{\eps}x(s)\,ds \ = \ -\int_0^T \bigg(\frac{1}{\eps}\int_{(s-\eps,s]}d\varphi(r)\bigg)x(s)\,ds \\
&= \ - \int_{(-\eps,T]}\bigg(\frac{1}{\eps}\int_r^{r+\eps} x(s)\,ds\bigg)d\varphi(r) \ = \ - \int_{(-\eps,0]}\bigg(\frac{1}{\eps}\int_r^{r+\eps} x(s)\,ds\bigg)d\varphi(r) - \int_{(0,T]}\bigg(\frac{1}{\eps}\int_r^{r+\eps} x(s)\,ds\bigg)d\varphi(r) \\
&= \ - \varphi(0)\frac{1}{\eps}\int_0^\eps x(s)\,ds - \int_{(0,T]}\bigg(\frac{1}{\eps}\int_r^{r+\eps} x(s)\,ds\bigg)d\varphi(r).
\end{align*}
Since $x$ is right-continuous we have that $\frac{1}{\eps}\int_r^{r+\eps}x(s)\,ds\rightarrow x(r)$ as $\eps\rightarrow0^+$. Moreover, since $x$ is bounded (being a c\`{a}dl\`{a}g function), by Lebesgue's dominated convergence theorem we conclude that
\[
\int_{(0,T]}\bigg(\frac{1}{\eps}\int_r^{r+\eps} x(s)\,ds\bigg)d\varphi(r) \ \overset{\eps\rightarrow0^+}{\longrightarrow} \ \int_{(0,T]}x(r)\,d\varphi(r).
\]
Hence
\begin{equation}\label{IbP_2}
- \varphi(0)\frac{1}{\eps}\int_0^\eps x(s)\,ds - \int_{(0,T]}\bigg(\frac{1}{\eps}\int_r^{r+\eps} x(s)\,ds\bigg)d\varphi(r) \ \overset{\eps\rightarrow0^+}{\longrightarrow} \ - \varphi(0)x(0) - \int_{(0,T]}x(r)d\varphi(r).
\end{equation}
From \eqref{IbP_1} and \eqref{IbP_2} we see that \eqref{IntbyParts} follows.
\end{proof}

\subsection{Cylindrical approximations}

\begin{Lemma}\label{L:CylindrApprox1}
Let $h\colon[0,T]\times C([0,T];\R^d)\times A\rightarrow\R$ be continuous and satisfying, for some constant $K\geq0$, items \textup{a)}-\textup{b)}-\textup{d)} or, alternatively, items \textup{a)}-\textup{c)}-\textup{d):}
\begin{enumerate}[\upshape a)]
\item $|h(t,x,a) - h(t,x',a)|\leq K\,\|x-x'\|_t$, for all $t\in[0,T]$, $x,x'\in C([0,T];\R^d)$, $a\in A$;
\item $|h(t,0,a)|\leq K$, for all $(t,a)\in[0,T]\times A$;
\item $|h(t,x,a)|\leq K$, for all $(t,x,a)\in[0,T]\times C([0,T];\R^d)\times A$;
\item $h$ is uniformly continuous in $t$, uniformly with respect to the other variables, namely there exists a modulus of continuity $w\colon[0,\infty)\rightarrow[0,\infty)$ such that
\begin{equation}\label{h_UnifCont}
|h(t,x,a) - h(s,x,a)| \ \leq \ w(|t - s|),
\end{equation}
for all $t,s\in[0,T],x\in C([0,T];\R^d),a\in A$.
\end{enumerate}
Then, there exists a sequence $\{h_n\}_n$ with $h_n\colon[0,T]\times C([0,T];\R^d)\times A\rightarrow\R$ continuous and satisfying the following.
\begin{enumerate}[\upshape 1)]
\item $h_n$ converges pointwise to $h$ uniformly with respect to $a$, namely: for every $(t,x)\in[0,T]\times C([0,T];\R^d)$, it holds that
\[
\lim_{n\rightarrow+\infty}\sup_{a\in A}\big|h_n(t,x,a) - h(t,x,a)\big| \ = \ 0.
\]
More precisely, \eqref{h_n-h_proof5} holds.
\item If $h$ satisfies items \textup{a)} and \textup{b)} $($resp. \textup{a)} and \textup{c)}$)$ then $h_n$ also satisfies the same items. In particular, $h_n$ satisfies item \textup{a)} with constant $2K$ and item \textup{b)} with a constant $\check K\geq0$, depending only on $K$ $($resp. item \textup{c)} with the same constant $K$$)$.
\item For every $n$, there exist $d_n\in\N$, a continuous function $\bar h_n\colon[0,T]\times\R^{dd_n}\times A\rightarrow\R$, and some continuously differentiable functions $\phi_{n,1},\ldots,\phi_{n,d_n}\colon[0,T]\rightarrow\R$ such that
\[
h_n(t,x,a) \ = \ \bar h_n\bigg(t,\int_{[0,t]}\phi_{n,1}(s)\,d^-x(s)+ x(0),\ldots,\int_{[0,t]}\phi_{n,d_n}(s)\,d^-x(s)+ x(0),a\bigg),
\]
for every $(t,x,a)\in[0,T]\times C([0,T];\R^d)\times A$. Moreover, $d_n$ and $\phi_{n,1},\ldots,\phi_{n,d_n}$ do not depend on $h$. In addition, $y_n^{t,x}$ given by \eqref{y_n^t,x} is such that \eqref{Dy-x1} and \eqref{Dy-x2} hold.
\item If $h$ satisfies items \textup{a)} and \textup{b)} $($resp. \textup{a)} and \textup{c)}$)$ then $\bar h_n$ satisfies items \textup{i)} and \textup{ii)} $($resp. \textup{i)} and \textup{iii)}$)$ below:
\begin{enumerate}[\upshape i)]
\item $|\bar h_n(t,y,a)-\bar h_n(t,y',a)|\leq K\,|y-y'|$, for all $t\in[0,T]$, $y,y'\in\R^{dd_n}$, $a\in A$;
\item $|\bar h_n(t,0,a)|\leq \check K$, for all $(t,a)\in[0,T]\times A$, for some constant $\check K\geq0$, depending only on $K$;
\item $|\bar h_n(t,y,a)|\leq K$, for all $(t,y,a)\in[0,T]\times\R^{dd_n}\times A$.
\end{enumerate}
\item For every $n$ and any $a\in A$, the function $\bar h_n(\cdot,\cdot,a)$ is $C^{1,2}([0,T]\times\R^{dd_n})$. Moreover, there exist constants $K_n\geq0$ and $q\in\{0,1\}$ such that
\begin{equation}\label{Estimate2ndDeriv}
\big|\partial_t\bar h_n(t,y,a)\big| + \big|\partial_y\bar h_n(t,y,a)\big| + \big|\partial_{yy}^2\bar h_n(t,y,a)\big| \ \leq \ K_n\,\big(1 + |y|\big)^q,
\end{equation}
for all $(t,y,a)\in[0,T]\times\R^{dd_n}\times A$. The constant $q$ is equal to $1$ if $h$ satisfies item \textup{b)}, while it is equal to $0$ if $h$ satisfies item \textup{c)}.
\end{enumerate}
\end{Lemma}
\begin{proof}[\textbf{Proof.}] We split the proof into six steps.

\vspace{1mm}

\noindent\emph{\textbf{Step I. Definitions of} $x_{n,y}^{\textup{pol}}$ \textbf{and} $x_n^{t,\textup{pol}}$.} For every $n\in\N$, consider the $n$-th dyadic mesh of the time interval $[0,T]$, that is
\[
0 \ = \ t_0^n \ < \ t_1^n \ < \ \ldots \ < \ t_{2^n}^n \ = \ T, \qquad\quad \text{with }t_j^n \ := \ \frac{j}{2^n}T,\text{ for every }j=0,\ldots,2^n.
\]
For every $y=(y_0,\ldots,y_{2^n})\in\R^{d\cdot(2^n+1)}$, we consider the corresponding $n$-th polygonal, denoted $x_{n,y}^{\textup{pol}}$, which is an element of $C([0,T];\R^d)$ and is characterized by the following properties:
\begin{itemize}
\item $x_{n,y}^{\textup{pol}}(t_n^j)=y_j$, for every $j=0,\ldots,2^n$;
\item $x_{n,y}^{\textup{pol}}$ is linear on every interval $[t_{j-1}^n,t_j^n]$, for any $j=1,\ldots,2^n$.
\end{itemize}
So, in particular, $x_{n,y}^{\textup{pol}}$ is given by the following formula:
\[
x_{n,y}^{\textup{pol}}(s) \ = \ \frac{y_j - y_{j-1}}{t_j^n - t_{j-1}^n} s + \frac{t_j^n\,y_{j-1} - t_{j-1}^n\,y_j}{t_j^n - t_{j-1}^n},
\]
for every $s\in[t_{j-1}^n,t_j^n]$ and any $j=1,\ldots,2^n$.\\
Now, given $t\in[0,T]$ and $x\in C([0,T];\R^d)$, we denote
\[
x_n^{t,\textup{pol}} \ := \ x_{n,\hat y_n^{t,x}}^{\textup{pol}},
\]
with
\[
\hat y_n^{t,x} \ := \ \big(x(t_0^n\wedge t),\ldots,x(t_{2^n}^n\wedge t)\big) \ = \ \bigg(\int_{[0,t]} 1_{[0,t_0^n)}(s)\,d^-x(s)+x(0),\ldots,\int_{[0,t]} 1_{[0,t_{2^n}^n)}(s)\,d^-x(s)+x(0)\bigg),
\]
where the second equality follows from the integration by parts formula \eqref{IntbyParts}.\\
It is easy to see that (in the following formulae we use the same symbol, that is $|\cdot|$, to denote the Euclidean norms on $\R^d$ and $\R^{d\cdot(2^n+1)}$)
\begin{equation}\label{Estimate_y}
\|x_{n,y}^{\textup{pol}}\|_T \ \leq \ \max_j |y_j| \ \leq \ |y|, \qquad\quad \|x_{n,y}^{\textup{pol}} - x_{n,\tilde y}^{\textup{pol}}\|_T \ \leq \ \max_j |y_j - \tilde y_j| \ \leq \ |y - \tilde y|,
\end{equation}
for every $y=(y_0,\ldots,y_{2^n}),\tilde y=(\tilde y_0,\ldots,\tilde y_{2^n})\in\R^{d\cdot(2^n+1)}$. Similarly, we have
\begin{equation}\label{estimate_x_pol-x}
\|x_n^{t,\textup{pol}}\|_t \leq \|x\|_t, \qquad \|x_n^{t,\textup{pol}} - x\|_t \leq \sup_{|r-s|\leq2^{-n}} |x(r\wedge t) - x(s\wedge t)| \leq \sup_{|r-s|\leq2^{-n}} |x(r) - x(s)|.
\end{equation}

\vspace{1mm}

\noindent\emph{\textbf{Step II. Definitions of $\phi_{n,j}$ and $y_n^{t,x}$}.} Let $\chi\colon\R\rightarrow\R$ be given by
\[
\chi(r) \ = \
\begin{cases}
\vspace{1mm}0, \qquad & r\leq0, \\
\vspace{1mm}\dfrac{1}{1+\frac{\textup{e}^{1/r}}{\textup{e}^{1/(1-r)}}}, \qquad & 0<r<1, \\
1, \qquad & r\geq1.
\end{cases}
\]
Notice that $\chi$ belongs to $C^\infty(\R)$ and is strictly increasing on $[0,1]$. Then, for every $n\in\N$, we define the functions $\phi_{n,0},\ldots,\phi_{n,2^n}\colon[0,T]\rightarrow\R$ as follows:
\begin{equation}\label{phi_n,j}
\phi_{n,j}(r) \ = \
\begin{cases}
\vspace{1mm}1, \qquad & \qquad 0\leq r<t_j^n, \\
1 - \chi(2^{2n} (r - t_j^n)), \qquad & \qquad t_j^n\leq r\leq T,
\end{cases}
\end{equation}
for $j=0,\ldots,2^n$; so, in particular,
\[
\phi_{n,0}(r) \ = \ 1 - \chi(2^{2n}r), \qquad\qquad 0\leq r\leq T.
\]
Moreover, for every $(t,x)\in[0,T]\times C([0,T];\R^d)$, let $y_n^{t,x}\in\R^{d\cdot(2^n+1)}$ be defined as
\begin{equation}\label{y_n^t,x}
y_n^{t,x} := \bigg(\int_{[0,t]}\phi_{n,0}(s)\,d^-x(s)+x(0),\ldots,\int_{[0,t]}\phi_{n,2^n}(s)\,d^-x(s)+x(0)\bigg).
\end{equation}
In the rest of this step we prove that the following estimate holds:
\begin{equation}\label{Dy-x}
\big\|x_{n,y_n^{t,x}}^{\textup{pol}} - x_n^{t,\textup{pol}}\big\|_t \ \leq \ 2\sup_{|r-s|\leq2^{-2n}} |x(r\wedge t) - x(s\wedge t)| \ \leq \ 2\sup_{|r-s|\leq2^{-n}} |x(r) - x(s)|.
\end{equation}
We begin noting that
\[
\big\|x_{n,y_n^{t,x}}^{\textup{pol}} - x_n^{t,\textup{pol}}\big\|_t \ = \ \max_{j=0,\ldots,2^n}\big|(y_n^{t,x})_j - x(t_j^n\wedge t)\big|.
\]
where $y_n^{t,x}=((y_n^{t,x})_0,\ldots,(y_n^{t,x})_{2^n})\in\R^{d\cdot(2^n+1)}$. By the integration by parts formula \eqref{IntbyParts}, we have
\begin{align*}
&(y_n^{t,x})_j - x(t_j^n\wedge t) \ = \ \int_{[0,t]}\phi_{n,j}(s)\,d^-x(s) + x(0) - x(t_j^n\wedge t) \\
&= \ \int_{[0,(t_j^n+2^{-2n})\wedge t]}\phi_{n,j}(s)\,d^-x(s) + x(0) - x(t_j^n\wedge t) \\
&= \ \phi_{n,j}((t_j^n+2^{-2n})\wedge t)\,x((t_j^n+2^{-2n})\wedge t) - x(t_j^n\wedge t) - \int_0^{(t_j^n+2^{-2n})\wedge t} x(s)\,\phi_{n,j}'(s)\,ds.
\end{align*}
Notice that $\phi_{n,j}(s)=1$ and $\phi_{n,j}'(s)=0$, for $0\leq s\leq t_j^n$. Now, we distinguish two cases.
\begin{enumerate}
\item If $t\leq t_j^n$, we have
\begin{equation}\label{Dy-x1}
(y_n^{t,x})_j - x(t_j^n\wedge t) \ = \ 0.
\end{equation}
\item If $t>t_j^n$, we have
\begin{align*}
(y_n^{t,x})_j - x(t_j^n\wedge t) \ &= \ \phi_{n,j}((t_j^n+2^{-2n})\wedge t)\,x((t_j^n+2^{-2n})\wedge t) - x(t_j^n) \\
&\quad \ - \int_{t_j^n}^{(t_j^n+2^{-2n})\wedge t} x(s)\,\phi_{n,j}'(s)\,ds.
\end{align*}
Observe that $\int_{t_j^n}^{(t_j^n+2^{-2n})\wedge t} \phi_{n,j}'(s)\,ds=\phi_{n,j}((t_j^n+2^{-2n})\wedge t)-1$, then
\begin{align*}
(y_n^{t,x})_j - x(t_j^n\wedge t) \ &= \ x((t_j^n+2^{-2n})\wedge t) - x(t_j^n) \\
&\quad \ + \int_{t_j^n}^{(t_j^n+2^{-2n})\wedge t} \big(x((t_j^n+2^{-2n})\wedge t) - x(s)\big)\,\phi_{n,j}'(s)\,ds \\
&\leq \ \sup_{|r-s|\leq2^{-2n}} |x(r\wedge t) - x(s\wedge t)| \bigg(1 + \int_{t_j^n}^{(t_j^n+2^{-2n})\wedge t} \big|\phi_{n,j}'(s)\big|\,ds\bigg) \\
&\leq \ \sup_{|r-s|\leq2^{-n}} |x(r\wedge t) - x(s\wedge t)| \bigg(1 + \int_{t_j^n}^{(t_j^n+2^{-2n})\wedge t} \big|\phi_{n,j}'(s)\big|\,ds\bigg).
\end{align*}
Since $\int_{t_j^n}^{(t_j^n+2^{-2n})\wedge t} |\phi_{n,j}'(s)|\,ds=-\int_{t_j^n}^{(t_j^n+2^{-2n})\wedge t} \phi_{n,j}'(s)\,ds=1-\phi_{n,j}((t_j^n+2^{-2n})\wedge t)$ and $1-\phi_{n,j}((t_j^n+2^{-2n})\wedge t)\leq1$, we get
\begin{equation}\label{Dy-x2}
\big|(y_n^{t,x})_j - x(t_j^n\wedge t)\big| \ \leq \ 2\sup_{|r-s|\leq2^{-n}} |x(r\wedge t) - x(s\wedge t)| \ \leq \ 2\sup_{|r-s|\leq2^{-n}} |x(r) - x(s)|.
\end{equation}
\end{enumerate}
From \eqref{Dy-x1} and \eqref{Dy-x2} we conclude that \eqref{Dy-x} holds.

\vspace{1mm}

\noindent\emph{\textbf{Step III. Definitions of $\bar h_n,h_n$ and proof of item 3)}.} For every $n\in\N$, let $\eta_n\colon\R\rightarrow\R$ be given by
\[
\eta_n(s) \ = \ \frac{2n}{\sqrt{2\pi}}\,\textup{e}^{-\frac{n^2}{2}s^2}, \qquad \forall\,s\in\R.
\]
Notice that $\int_0^\infty\eta_n(s)ds=1$. Moreover, for every $n\in\N$, let $\zeta_n\colon\R^{d\cdot(2^n+1)}\rightarrow\R$ be the probability density function of the multivariate normal distribution $\Nc(0,(d(2^n+1))^{-2}\,I_{d(2^n+1)})$, where $I_{d(2^n+1)}$ denotes the identity matrix of order $d(2^n+1)$:
\[
\zeta_n(z) \ = \ \frac{(d(2^n+1))^{d(2^n+1)}}{(2\pi)^{d(2^n+1)/2}}\,\textup{e}^{-\frac{(d(2^n+1))^2}{2}|z|^2}, \qquad \forall\,z\in\R^{d\cdot(2^n+1)}.
\]
Now, define $\bar h_n\colon[0,T]\times\R^{d\cdot(2^n+1)}\times A\rightarrow\R$ as follows:
\[
\bar h_n(t,y,a) \ = \ \int_0^\infty \int_{\R^{d\cdot(2^n+1)}} \eta_n(s)\,\zeta_n(z)\,h\big((t+s)\wedge T,x_{n,y+z}^{\textup{pol}},a\big)\,ds\,dz,
\]
for all $(t,y,a)\in[0,T]\times\R^{d\cdot(2^n+1)}\times A$. Finally, let $h_n\colon[0,T]\times C([0,T];\R^d)\times A\rightarrow\R$ be given by
\[
h_n(t,x,a) \ = \ \bar h_n\bigg(t,\int_{[0,t]}\phi_{n,0}(s)\,d^-x(s)+x(0),\ldots,\int_{[0,t]}\phi_{n,2^n}(s)\,d^-x(s)+x(0),a\bigg),
\]
with $\phi_{n,j}$ as in \eqref{phi_n,j}. From the continuity of $h$, we see that both $h_n$ and $\bar h_n$ are continuous.

\vspace{1mm}

\noindent\emph{\textbf{Step IV. Proof of item 1)}.} For every $(t,x)\in[0,T]\times C([0,T];\R^d)$, let $y_n^{t,x}\in\R^{d\cdot(2^n+1)}$ be given by \eqref{y_n^t,x}. Then, we have (using also the equality $h(t,x,a)=h(t,x_{\cdot\wedge t},a)$)
\begin{align*}
&|h_n(t,x,a) - h(t,x,a)| \\
&\leq \ \int_0^\infty \int_{\R^{d\cdot(2^n+1)}} \eta_n(s)\,\zeta_n(z)\,\big|h\big((t+s)\wedge T,x_{n,y_n^{t,x}+z}^{\textup{pol}},a\big) - h(t,x_{\cdot\wedge t},a)\big|\,ds\,dz \\
&\leq \ \int_0^\infty \int_{\R^{d\cdot(2^n+1)}} \eta_n(s)\,\zeta_n(z)\,\big|h\big((t+s)\wedge T,x_{n,y_n^{t,x}+z}^{\textup{pol}},a\big) - h\big((t+s)\wedge T,x_{\cdot\wedge t},a\big)\big|\,ds\,dz \\
&\quad \ + \int_0^\infty \eta_n(s)\,\big|h\big((t+s)\wedge T,x_{\cdot\wedge t},a\big) - h(t,x_{\cdot\wedge t},a)\big|\,ds \\
&\leq \ K\int_0^\infty \int_{\R^{d\cdot(2^n+1)}} \eta_n(s)\,\zeta_n(z)\,\big\|x_{n,y_n^{t,x}+z}^{\textup{pol}} - x_{\cdot\wedge t}\big\|_{(t+s)\wedge T}\,ds\,dz \\
&\quad \ + \int_0^\infty \eta_n(s)\,\big|h\big((t+s)\wedge T,x_{\cdot\wedge t},a\big) - h(t,x_{\cdot\wedge t},a)\big|\,ds.
\end{align*}
Since both paths $x_{n,y_n^{t,x}+z}^{\textup{pol}}$ and $x_{\cdot\wedge t}$ are constant after time $t$, moreover recalling \eqref{Estimate_y}, \eqref{estimate_x_pol-x} and the linearity of the map $y\mapsto x_{n,y}^{\textup{pol}}$, we obtain
\begin{align*}
\big\|x_{n,y_n^{t,x}+z}^{\textup{pol}} - x_{\cdot\wedge t}\big\|_{(t+s)\wedge T} \ &= \ \big\|x_{n,y_n^{t,x}}^{\textup{pol}} + x_{n,z}^{\textup{pol}} - x_{\cdot\wedge t}\big\|_t \ \leq \ \big\|x_{n,y_n^{t,x}}^{\textup{pol}} - x_{\cdot\wedge t}\big\|_t + |z| \\
&\leq \ \big\|x_{n,y_n^{t,x}}^{\textup{pol}} - x_n^{\textup{pol}}\big\|_t + \|x_n^{\textup{pol}} - x\|_t + |z|.
\end{align*}
Then, by \eqref{estimate_x_pol-x} and \eqref{Dy-x}, we get
\begin{align}\label{h_n-h_proof2}
|h_n(t,x,a) - h(t,x,a)| \ &\leq \ 3K\sup_{|r-s|\leq2^{-n}} |x(r) - x(s)| + K\int_{\R^{d\cdot(2^n+1)}} \zeta_n(z)\,|z|\,dz \\
&\quad \ + \int_0^\infty \eta_n(s)\,\big|h\big((t+s)\wedge T,x_{\cdot\wedge t},a\big) - h(t,x_{\cdot\wedge t},a)\big|\,ds. \notag
\end{align}
Now, by \eqref{h_UnifCont} we have
\begin{align}\label{h_n-h_proof3}
&\int_0^\infty \eta_n(s)\,\big|h\big((t+s)\wedge T,x_{\cdot\wedge t},a\big) - h(t,x_{\cdot\wedge t},a)\big|\,ds \notag \\
&\leq \ \int_0^\infty \eta_n(s)\,w\big((t+s)\wedge T-t\big)\,ds \ = \ \int_0^\infty \frac{2}{\sqrt{2\pi}}\,\textup{e}^{-\frac{1}{2}r^2}\,w\big((t+r/n)\wedge T-t\big)\,dr.
\end{align}
Finally, consider the integral $\int_{\R^{d\cdot(2^n+1)}} \zeta_n(z)\,|z|\,dz$. Since the integrand is a radial function, it is more convenient to rewrite it in terms of spherical coordinates (see for instance \cite[Appendix C.3]{Evans}). In particular, denoting by $S_{d(2^n+1)-1}(R)$ the surface area of the sphere $\{|z|=R\}$, which is equal to $2\pi^{d(2^n+1)/2}R^{d(2^n+1)-1}/\Gamma(d(2^n+1)/2)$, with $\Gamma(\cdot)$ being the Gamma function, we get
\begin{align*}
&\int_{\R^{d\cdot(2^n+1)}} \zeta_n(z)\,|z|\,dz \ = \ \int_{\R^{d\cdot(2^n+1)}}\frac{(d(2^n+1))^{d(2^n+1)}}{(2\pi)^{d(2^n+1)/2}}\,\textup{e}^{-\frac{(d(2^n+1))^2}{2}|z|^2}\,|z|\,dz \notag \\
&= \ \frac{1}{d(2^n+1)}\int_{\R^{d\cdot(2^n+1)}} \frac{1}{(2\pi)^{d(2^n+1)/2}}\,\textup{e}^{-\frac{1}{2}|y|^2}\,|y|\,dy \notag \\
&= \ \frac{1}{d(2^n+1)}\frac{1}{(2\pi)^{(d(2^n+1)-1)/2}}\int_0^\infty \frac{1}{\sqrt{2\pi}}\,\textup{e}^{-\frac{1}{2}R^2}\,R\,S_{d(2^n+1)-1}(R)\,dR \notag \\
&= \ \frac{1}{d(2^n+1)}\frac{2\pi^{d(2^n+1)/2}}{(2\pi)^{(d(2^n+1)-1)/2}}\frac{1}{\Gamma(d(2^n+1)/2)}\int_0^\infty \frac{1}{\sqrt{2\pi}}\,\textup{e}^{-\frac{1}{2}R^2}\,R^{d(2^n+1)}\,dR.
\end{align*}
Since $\int_0^\infty\frac{1}{\sqrt{2\pi}}\textup{e}^{-\frac{1}{2}R^2}R^{d(2^n+1)}dR=2^{d(2^n+1)/2-1}\Gamma((d(2^n+1)+1)/2)/\sqrt{\pi}$, we find
\[
\int_{\R^{d\cdot(2^n+1)}} \zeta_n(z)\,|z|\,dz \ = \ \sqrt{2}\,\frac{\Gamma((d(2^n+1)+1)/2)}{d(2^n+1)\,\Gamma(d(2^n+1)/2)}.
\]
By \cite{Kershaw} we know that the following inequality holds:
\[
\frac{\Gamma(z+1/2)}{\Gamma(z)} \ \leq \ \bigg(z+\frac{\sqrt{3}}{2}\bigg)^{1/2} \ \leq \ 2\sqrt{z}, \qquad \forall\,z>\frac{1}{2}.
\]
Hence
\begin{equation}\label{h_n-h_proof4}
\int_{\R^{d\cdot(2^n+1)}} \zeta_n(z)\,|z|\,dz \ \leq \ \frac{2}{\sqrt{d(2^n+1)}}.
\end{equation}
In conclusion, plugging \eqref{h_n-h_proof3} and \eqref{h_n-h_proof4} into \eqref{h_n-h_proof2}, we obtain
\begin{align}\label{h_n-h_proof5}
|h_n(t,x,a) - h(t,x,a)| \ &\leq \ 3K\sup_{|r-s|\leq2^{-n}} |x(r) - x(s)| + \frac{2K}{\sqrt{d(2^n+1)}} \\
&\quad \ + \int_0^\infty \frac{2}{\sqrt{2\pi}}\,\textup{e}^{-\frac{1}{2}r^2}\,w\big((t+r/n)\wedge T-t\big)\,dr. \notag
\end{align}
Then, letting $n\rightarrow\infty$ in \eqref{h_n-h_proof5}, using Lebesgue's dominated convergence theorem, we deduce that item 1) holds.

\vspace{1mm}

\noindent\emph{\textbf{Step V. Proof of items 2) and 4).}} It is clear that $h_n$ (resp. $\bar h_n$) satisfies item c) (resp. 4)-iii)) with the same constant $K$. If $h$ satisfies item b) then $|h(t,x,a)|\leq K(1+\|x\|_t)$, therefore, by \eqref{Estimate_y},
\begin{align*}
|\bar h_n(t,y,a)| \ &\leq \ K\int_{\R^{d\cdot(2^n+1)}}\zeta_n(z)\,\big(1 + \big\|x_{n,y+z}^{\textup{pol}}\big\|_t\big)\,dz \\
&\leq \ K\,(1 + |y|) + K\int_{\R^{d\cdot(2^n+1)}}\zeta_n(z)\,|z|\,dz \ = \ K\,(1 + |y|) + \frac{2K}{\sqrt{d(2^n+1)}},
\end{align*}
where the last equality follows from \eqref{h_n-h_proof4}. Since $d(2^n+1)\geq1$, we get\[
|\bar h_n(t,y,a)| \ \leq \ K\,(1 + |y|) + 2\,K,
\]
which proves item 4)-ii). Concerning $h_n$, we have
\begin{align*}
|h_n(t,x,a)| \ &\leq \ K\,\big(1 + \big\|y_n^{t,x}\big\|_t\big) + 2K\\
&\leq \ K\,\big(1 + \big\|y_n^{t,x} - x_n^{t,\textup{pol}}\big\|_t + \big\|x_n^{t,\textup{pol}}\big\|_t\big) + 2\,K
\end{align*}
By \eqref{Dy-x} and \eqref{estimate_x_pol-x}, we obtain
\begin{align*}
|h_n(t,x,a)| \ &\leq \ K\Big(1 + 2\sup_{|r-s|\leq2^{-n}} |x(r\wedge t) - x(s\wedge t)| + \|x\|_t\Big) + 2K \\
&\leq \ K\big(1 + 3\|x\|_t\big) + 2K,
\end{align*}
which proves that $h_n$ satisfies item b) with a constant $\check K$, depending only on $K$.\\
Let us now prove that $h_n$ and $\bar h_n$ satisfy respectively item a) and item 4)-i). We have
\begin{align*}
&|\bar h_n(t,y,a) - \bar h_n(t,\tilde y,a)| \\
&\leq \ \int_0^\infty \int_{\R^{d\cdot(2^n+1)}} \eta_n(s)\,\zeta_n(z)\,\big|h\big((t+s)\wedge T,x_{n,y+z}^{\textup{pol}},a\big) - h\big((t+s)\wedge T,x_{n,\tilde y+z}^{\textup{pol}},a\big)\big|\,ds\,dz \\
&\leq \ K\int_0^\infty \int_{\R^{d\cdot(2^n+1)}} \eta_n(s)\,\zeta_n(z)\,\big\|x_{n,y+z}^{\textup{pol}} - x_{n,\tilde y+z}^{\textup{pol}}\big\|_{(t+s)\wedge T}\,ds\,dz \\
&= \ K\int_0^\infty \eta_n(s)\,\big\|x_{n,y}^{\textup{pol}} - x_{n,\tilde y}^{\textup{pol}}\big\|_{(t+s)\wedge T}\,ds,
\end{align*}
where the last equality follows from the linearity of the map $y\mapsto x_{n,y}^{\textup{pol}}$. Hence, recalling \eqref{Estimate_y}, we obtain
\begin{equation}\label{Lipschitz_bar_h_n}
|\bar h_n(t,y,a) - \bar h_n(t,\tilde y,a)| \ \leq \ K\max_j|y_j - \tilde y_j| \ \leq \ K\,|y - \tilde y|,
\end{equation}
which proves that $\bar h_n$ satisfies item 4)-i). Let us now prove that $h_n$ satisfies item a). From \eqref{Lipschitz_bar_h_n} we have
\[
|h_n(t,x,a) - h_n(t,\tilde x,a)| \ \leq \ K\max_j\big|(y_n^{t,x})_j - (y_n^{t,\tilde x})_j\big|.
\]
By the integration by parts formula \eqref{IntbyParts}, we get
\begin{align*}
(y_n^{t,x})_j - (y_n^{t,\tilde x})_j \ &= \ \int_{[0,t]}\phi_{n,j}(s)\,d^-x(s) - \int_{[0,t]}\phi_{n,j}(s)\,d^- \tilde x(s) \\
&= \ \phi_{n,j}(t)\,\big(x(t) - \tilde x(t)\big) - \int_0^t \big(x(s) - \tilde x(s)\big)\,\phi_{n,j}'(s)\,ds.
\end{align*}
Hence
\[
\big|(y_n^{t,x})_j - (y_n^{t,\tilde x})_j\big| \ \leq \ |x(t) - \tilde x(t)| + \|x - \tilde x\|_t\int_0^t |\phi_{n,j}'(s)|\,ds.
\]
Since $\int_0^t |\phi_{n,j}'(s)|\,ds=-\int_0^t \phi_{n,j}'(s)\,ds=1-\phi_{n,j}(t)$ and $1-\phi_{n,j}(t)\leq1$, we conclude that
\[
|h_n(t,x,a) - h_n(t,\tilde x,a)| \ \leq \ 2K\|x - \tilde x\|_t,
\]
which proves that $h_n$ satisfies item a) with constant $2K$.

\vspace{1mm}

\noindent\emph{\textbf{Step VI. Proof of item 5)}.} Recall that
\begin{align*}
\bar h_n(t,y,a) \ &= \ \int_0^\infty \int_{\R^{d\cdot(2^n+1)}} \eta_n(s)\,\zeta_n(z)\,h\big((t+s)\wedge T,x_{n,y+z}^{\textup{pol}},a\big)\,ds\,dz \\
&= \ \int_t^\infty \int_{\R^{d\cdot(2^n+1)}} \eta_n(s-t)\,\zeta_n(z-y)\,h\big(s\wedge T,x_{n,z}^{\textup{pol}},a\big)\,ds\,dz,
\end{align*}
for all $(t,y,a)\in[0,T]\times\R^{d\cdot(2^n+1)}\times A$. Then, it is clear that, for every $a\in A$, the function $\bar h_n(\cdot,\cdot,a)\in C^{1,2}([0,T]\times\R^{dd_n})$. Moreover, by direct calculation, we have
\begin{align*}
\partial_t \bar h_n(t,y,a) \ &= \ - \int_t^\infty \int_{\R^{d\cdot(2^n+1)}} \eta_n'(s-t)\,\zeta_n(z-y)\,h\big(s\wedge T,x_{n,z}^{\textup{pol}},a\big)\,ds\,dz \\
&\quad \ - \int_{\R^{d\cdot(2^n+1)}} \eta_n(0)\,\zeta_n(z-y)\,h\big(t\wedge T,x_{n,z}^{\textup{pol}},a\big)\,dz, \\
\partial_y \bar h_n(t,y,a) \ &= \ - \int_t^\infty \int_{\R^{d\cdot(2^n+1)}} \eta_n(s-t)\,\partial_y\zeta_n(z-y)\,h\big(s\wedge T,x_{n,z}^{\textup{pol}},a\big)\,ds\,dz, \\
\partial_{yy}^2 \bar h_n(t,y,a) \ &= \ \int_t^\infty \int_{\R^{d\cdot(2^n+1)}} \eta_n(s-t)\,\partial_{yy}^2\zeta_n(z-y)\,h\big(s\wedge T,x_{n,z}^{\textup{pol}},a\big)\,ds\,dz.
\end{align*}
By item b) or, alternatively, item c), we conclude that \eqref{Estimate2ndDeriv} holds.
\end{proof}

\noindent Under Assumptions \ref{AssA} and \ref{AssB},
{the coefficients $b$, $f$, $g$ satisfy items a), c), d) of Lemma \ref{L:CylindrApprox1}, therefore from this lemma} we get sequences $\{b_n\}_n$, $\{f_n\}_n$, $\{g_n\}_n$, with
\begin{equation}\label{Coeff_n}
b_n,\,f_n\colon[0,T]\times C([0,T];\R^d)\times A \, \longrightarrow \, \R^d,\,\R, \qquad g_n\colon C([0,T];\R^d) \, \longrightarrow \, \R,
\end{equation}
satisfying items \textup{1)-2)-3)-4)-5)} of Lemma \ref{L:CylindrApprox1}. We also recall from Lemma \ref{L:CylindrApprox1} that $d_n$ and $\phi_{n,1},\ldots,\phi_{n,d_n}$ are the same for $b$, $f$, $g$.\\
Finally, let
\begin{equation}\label{Value_n}
v_n\colon[0,T]\times C([0,T];\R^d) \ \longrightarrow \ \R	
\end{equation}
denote the value function of the optimal control problem with coefficients $b_n$, $\sigma$, $f_n$, $g_n$.

\begin{Lemma}\label{L:CylindrApprox2}
Let Assumptions \ref{AssA}, \ref{AssB}, \ref{AssC} hold. Consider the sequences $\{b_n\}_n$, $\{f_n\}_n$, $\{g_n\}_n$, $\{v_n\}_n$ in \eqref{Coeff_n}-\eqref{Value_n}. Then, $v_n$ converges pointwise to $v$ in \eqref{Value} as $n\rightarrow+\infty$.
\end{Lemma}
\begin{proof}[\textbf{Proof.}]
For every $n\in\N$, $t\in[0,T]$, $x\in C([0,T];\R^d)$, $\alpha\in\Ac$, let $X^{n,t,x,\alpha}\in\mathbf S_2(\F)$ be the unique solution to the following system of controlled stochastic differential equations:
\[
\begin{cases}
dX_s \ = \ b_n(s,X,\alpha_s)\,ds + \sigma(s,X,\alpha_s)\,dB_s, \qquad &\quad s\in(t,T], \\
X_s \ = \ x(s), &\quad s\in[0,t].
\end{cases}
\]
Then
\begin{align*}
&|v_n(t,x) - v(t,x)| \\
&\leq \ \sup_{\alpha\in\Ac}\E\bigg[\int_t^T \big|f_n\big(s,X^{n,t,x,\alpha},\alpha_s\big) - f\big(s,X^{t,x,\alpha},\alpha_s\big)\big|\,ds + \big|g_n\big(X^{n,t,x,\alpha}\big) - g\big(X^{t,x,\alpha}\big)\big|\bigg] \\
&\leq \ \sup_{\alpha\in\Ac}\bigg\{K\,(T + 1)\,\big\|X^{n,t,x,\alpha} - X^{t,x,\alpha}\big\|_{\mathbf S_2} + \E\bigg[\int_t^T \big|f_n\big(s,X^{t,x,\alpha},\alpha_s\big) - f\big(s,X^{t,x,\alpha},\alpha_s\big)\big|\,ds\bigg] \\
&\quad \ + \E\big[\big|g_n\big(X^{t,x,\alpha}\big) - g\big(X^{t,x,\alpha}\big)\big|\big]\bigg\}.
\end{align*}
By standard calculations (as for instance in \cite[Theorem 2.5.9]{Krylov80}), we find
\[
\big\|X^{n,t,x,\alpha} - X^{t,x,\alpha}\big\|_{\mathbf S_2}^2 \ \leq \ C_K T\textup{e}^{C_K T}\E\bigg[\int_t^T \big|b_n\big(s,X^{t,x,\alpha},\alpha_s\big) - b\big(s,X^{t,x,\alpha},\alpha_s\big)\big|^2\,ds\bigg].
\]
for some constant $C_K$, depending only on the constant $K$. It remains to prove that
\begin{align*}
&\sup_{\alpha\in\Ac}\bigg\{\E\bigg[\int_t^T \big|b_n\big(s,X^{t,x,\alpha},\alpha_s\big) - b\big(s,X^{t,x,\alpha},\alpha_s\big)\big|^2\,ds\bigg] \\
&+ \E\bigg[\int_t^T \big|f_n\big(s,X^{t,x,\alpha},\alpha_s\big) - f\big(s,X^{t,x,\alpha},\alpha_s\big)\big|\,ds\bigg] + \E\big[\big|g_n\big(X^{t,x,\alpha}\big) - g\big(X^{t,x,\alpha}\big)\big|\big]\bigg\} \ \overset{n\rightarrow+\infty}{\longrightarrow} \ 0.
\end{align*}
Let us address the term with $f$ and $f_n$, the other three terms can be treated in a similar way. From the proof of Lemma \ref{L:CylindrApprox1}, and in particular from estimate \eqref{h_n-h_proof5} with $h$ and $h_n$ replaced respectively by $f$ and $f_n$, we get
\begin{align}\label{LemmaB2Proof}
&\E\bigg[\int_t^T \big|f_n\big(s,X^{t,x,\alpha},\alpha_s\big) - f\big(s,X^{t,x,\alpha},\alpha_s\big)\big|\,ds\bigg] \ \leq \ 3KT\,\E\Big[\sup_{|r-s|\leq2^{-n}}\big|X_r^{t,x,\alpha} - X_s^{t,x,\alpha}\big|\Big] \notag \\
&+ \frac{2KT}{\sqrt{d(2^n+1)}} + \int_t^T\bigg(\int_0^\infty \frac{2}{\sqrt{2\pi}}\,\textup{e}^{-\frac{1}{2}r^2}\,w\big((s+r/n)\wedge T-s\big)\,dr\bigg)ds,
\end{align}
where we recall that $w$ is the modulus of continuity of $f$ with respect to the time variable. By Lebesgue's dominated convergence theorem, we see that the last integral in \eqref{LemmaB2Proof} goes to zero as $n\rightarrow+\infty$. Moreover, by standard calculations, we get
\[
\E\Big[\sup_{|r-s|\leq2^{-n}}\big|X_r^{t,x,\alpha} - X_s^{t,x,\alpha}\big|^2\Big] \ \leq \ \sup_{|r-s|\leq2^{-n}} |x(r\wedge t) - x(s\wedge t)|^2 + C_{K,T}\,2^{-n}\,\big(1 + \|x\|_t^2\big),
\]
for some constant $C_{K,T}\geq0$, depending only on $K$ and $T$. This allows to prove that right-hand side of \eqref{LemmaB2Proof} goes to zero as $n\rightarrow+\infty$ and concludes the proof.
\end{proof}

\begin{Lemma}\label{L:CylindrApprox3}
Let Assumptions \ref{AssA}, \ref{AssB}, \ref{AssC} hold. Suppose also that there exist $\hat d\in\N$ and functions
\[
\bar b,\;\bar\sigma,\;\bar f\colon[0,T]\times\R^{d\hat d}\times A \ \longrightarrow \ \R^d,\;\R^{d\times m},\;\R, \qquad\quad \bar g\colon\R^{d\hat d} \ \longrightarrow \ \R,
\]
satisfying the following conditions. $($Notice that items \textup{i)-ii)-iii)} below are not true assumptions for $\sigma$, since Assumption \ref{AssC} holds.
Indeed here we are only assuming, without loss of generality, that $\bar d$ from Assumption \ref{AssC} coincides with $\hat d$ and that the functions $\varphi_1,\ldots,\varphi_{\bar d}$ appearing in \ref{AssC}\textup{-(i)} coincide with $\phi_1,\ldots,\phi_{\hat d}$.$)$
\begin{enumerate}[\upshape i)]
\item There exist some continuously differentiable functions $\phi_1,\ldots,\phi_{\hat d}\colon[0,T]\rightarrow\R$ such that:
\begin{align*}
b(t,x,a) \ &= \ \bar b\bigg(t,\int_{[0,t]}\phi_1(s)\,d^-x(s)+x(0),\ldots,\int_{[0,t]}\phi_{\hat d}(s)\,d^-x(s)+x(0),a\bigg), \\
\sigma(t,x,a) \ &= \ \bar\sigma\bigg(t,\int_{[0,t]}\phi_1(s)\,d^-x(s)+x(0),\ldots,\int_{[0,t]}\phi_{\hat d}(s)\,d^-x(s)+x(0),a\bigg), \\
f(t,x,a) \ &= \ \bar f\bigg(t,\int_{[0,t]}\phi_1(s)\,d^-x(s)+x(0),\ldots,\int_{[0,t]}\phi_{\hat d}(s)\,d^-x(s)+x(0),a\bigg), \\
g(x) \ &= \ \bar g\bigg(\int_{[0,T]}\phi_1(s)\,d^-x(s)+x(0),\ldots,\int_{[0,T]}\phi_{\hat d}(s)\,d^-x(s)+x(0)\bigg),
\end{align*}
for every $(t,x,a)\in[0,T]\times C([0,T];\R^d)\times A$.
\item There exist a constant $\hat K\geq0$ such that
\begin{align*}
|\bar b(t,y,a) - \bar b(t,y',a)| + |\bar\sigma(t,y,a) - \bar\sigma(t,y',a)| \, + \ & \\
+ \, |\bar f(t,y,a) - \bar f(t,y',a)| + |\bar g(y) - \bar g(y')| \ &\leq \ \hat K\,|y - y'|, \\
|\bar b(t,y,a)| + |\bar\sigma(t,y,a)| + |\bar f(t,y,a)| + |\bar g(y)| \ &\leq \ \hat K,
\end{align*}
for all $(t,a)\in[0,T]\times A$, $y,y'\in\R^{d\hat d}$.
\item For every $a\in A$, the functions $\bar b(\cdot,\cdot,a),\bar\sigma(\cdot,\cdot,a),\bar f(\cdot,\cdot,a),\bar g(\cdot)$ are $C^{1,2}([0,T]\times\R^{d\hat d})$. Moreover, there exist constants $\bar K\geq0$ and $\bar q\geq0$ such that
\begin{align*}
\big|\partial_t\bar b(t,y,a)\big| + \big|\partial_t\bar\sigma(t,y,a)\big| + \big|\partial_t\bar f(t,y,a)\big| \ &\leq \ \bar K\,\big(1 + |y|\big)^{\bar q}, \\
\big|\partial_y\bar b(t,y,a)\big| + \big|\partial_y\bar\sigma(t,y,a)\big| + \big|\partial_y\bar f(t,y,a)\big| + \big|\partial_y\bar g(y)\big| \ &\leq \ \bar K\,\big(1 + |y|\big)^{\bar q}, \\
\big|\partial_{yy}^2\bar b(t,y,a)\big| + \big|\partial_{yy}^2\bar\sigma(t,y,a)\big| + \big|\partial_{yy}^2\bar f(t,y,a)\big| + \big|\partial_{yy}^2\bar g(y)\big| \ &\leq \ \bar K\,\big(1 + |y|\big)^{\bar q},
\end{align*}
for all $(t,y,a)\in[0,T]\times\R^{d\hat d}\times A$.
\end{enumerate}
Then, for every $\eps\in(0,1)$, there exist $v_\eps\colon[0,T]\times C([0,T];\R^d)\rightarrow\R$ and $\bar v_\eps\colon[0,T]\times\R^{d\hat d}\rightarrow\R$, with
\[
v_\eps(t,x) \ = \ \bar v_\eps\bigg(t,\int_{[0,t]}\phi_1(s)\,d^-x(s)+x(0),\ldots,\int_{[0,t]}\phi_{\hat d}(s)\,d^-x(s)+x(0)\bigg),
\]
for all $(t,x)\in[0,T]\times C([0,T];\R^d)$, such that the following holds.
\begin{enumerate}[\upshape 1)]
\item $v_\eps\in C_{\textup{pol}}^{1,2}([0,T]\times C([0,T];\R^d))$ and $\bar v_\eps\in C^{1,2}([0,T]\times\R^{d\hat d})$.
\item $v_\eps$ is a classical solution of the following equation in the unknown $u$ (see Remark \ref{R:v_eps}):
\begin{equation}\label{HJB_eps}
\hspace{-.8cm}\begin{cases}
\vspace{2mm}
\partial_t^H u(t,x) + \dfrac{1}{2}\eps^2\textup{tr}\big[\partial_{yy}\bar v_\eps(t,y^{t,x})\big] + \sup_{a\in A}\bigg\{\big\langle b(t,x,a),\partial^V_x u(t,x)\big\rangle \\
\vspace{2mm}+\,\dfrac{1}{2}\textup{tr}\big[(\sigma\sigma\trans)(t,x,a)\partial^V_{xx}u(t,x)\big]+f(t,x,a)\bigg\} = 0, &\hspace{-2.7cm}(t,x)\in[0,T)\times C([0,T];\R^d), \\
u(T,x) = g(x), &\,\hspace{-2.7cm}x\in C([0,T];\R^d),
\end{cases}
\end{equation}
where, for every $(t,x)\in[0,T]\times C([0,T];\R^d)$, $y^{t,x}\in\R^{d\hat d}$ is defined as
\begin{equation}\label{y^t,x}
y^{t,x} \ := \ \bigg(\int_{[0,t]}\phi_1(s)\,d^-x(s),\ldots,\int_{[0,t]}\phi_{\hat d}(s)\,d^-x(s)\bigg).
\end{equation}
\item There exists a constant $\bar C'\geq0$, depending only on $\hat K,\bar K,\bar q$, such that
\[
- \bar C'\,\textup{e}^{\bar C'(T-t)}\,\big(1 + |y|\big)^{3\bar q} \ \leq \ \partial_{y_iy_j}^2 \bar v_\eps(t,y) \ \leq \ \frac{1}{\eps^2}\bar C'\,\textup{e}^{\bar C'(T-t)}\,\big(1 + |y|\big)^{3\bar q},
\]
for all $(t,y)\in[0,T]\times\R^{d\hat d}$ and every $i,j=1,\ldots,d\hat d$.
\item There exists a constant $\bar L'\geq0$, depending only on $T$ and $\hat K$, such that
\[
\big|\partial_x^V v_\eps(t,x)\big| \ \leq \ \bar L',
\]
for every $(t,x)\in[0,T]\times C([0,T];\R^d)$.
\item Finally, $v_\eps$ converges pointwise to $v$ in \eqref{Value} as $\eps\rightarrow0^+$.
\end{enumerate}
\end{Lemma}
\begin{Remark}\label{R:v_eps}
{In equation \eqref{HJB_eps} with unknown $u$, the term $\frac{1}{2}\eps^2\textup{tr}[\partial_{yy}\bar v_\eps(t,y^{t,x})]$ is known as it does not depend on $u$ but it involves the function $\bar v_\eps$. The reason for the presence of this term is due to the fact that we first derive the HJB equation for the function $\bar v_\eps$, then we use equalities \eqref{EqualitiesDerivH}-\eqref{EqualitiesDerivV}-\eqref{EqualitiesDerivVV} to derive equation \eqref{HJB_eps} for $v_\eps$. However, from those equalities we are not able to rewrite the term $\frac{1}{2}\eps^2\textup{tr}[\partial_{yy}\bar v_\eps(t,y^{t,x})]$ in terms of $v_\eps$, therefore we have left it as it is since this is not relevant for the sequel $($actually we could work with the HJB equation satisfied by $\bar v_\eps$$)$.}
\end{Remark}
\begin{proof}[\textbf{Proof (of Lemma \ref{L:CylindrApprox3}).}]
Let $\boldsymbol\phi\colon[0,T]\rightarrow\R^{(d\hat d)\times d}$ be given by
\begin{equation}\label{bold_phi}
\boldsymbol\phi(t) \ = \
\left[\begin{array}{c}
\phi_1(t)I_d \\
\vdots \\
\phi_{\hat d}(t)I_d	
\end{array}
\right]
\end{equation}
for all $t\in[0,T]$, where $I_d$ denotes the $d\times d$ identity matrix. Let $\bar b_\phi\colon[0,T]\times\R^{d\hat d}\times A\rightarrow\R^{d\hat d}$ and $\bar\sigma_\phi\colon[0,T]\times\R^{d\hat d}\times A\rightarrow\R^{(d\hat d)\times m}$ be given by
\begin{equation}\label{b_sigma_phi}
\bar b_\phi(t,y,a) \ = \ \boldsymbol\phi(t)\,\bar b(t,y,a), \qquad\qquad
\bar\sigma_\phi(t,y,a) \ = \ \boldsymbol\phi(t)\,\bar\sigma(t,y,a),
\end{equation}
for all $(t,y,a)\in[0,T]\times\R^{d\hat d}\times A$, with $\bar b(t,y,a)$ being a column vector of dimension $d$. For every $\eps\in(0,1)$, consider the following Hamilton-Jacobi-Bellman equation on $[0,T]\times\R^{d\hat d}$:
\begin{equation}\label{HJB_fin_dim}
\begin{cases}
\vspace{2mm}
\partial_t \bar u(t,y) + \sup_{a\in A}\bigg\{\big\langle\bar b_\phi(t,y,a),\partial_y \bar u(t,y)\big\rangle + \dfrac{1}{2}\eps^2\text{tr}\big[\partial^2_{yy}\bar u(t,y)\big] \\
\vspace{2mm}+\,\dfrac{1}{2}\text{tr}\big[(\bar\sigma_\phi\bar\sigma_\phi\trans)(t,y,a)\partial^2_{yy}\bar u(t,y)\big]+\bar f(t,y,a)\bigg\} = 0, &\hspace{-1cm}(t,y)\in[0,T)\times\R^{d\hat d}, \\
\bar u(T,y) = \bar g(y), &\,\hspace{-1cm}y\in\R^{d\hat d}.
\end{cases}
\end{equation}
\emph{Proof of item 3).} From the assumptions on $\bar b,\bar\sigma,\bar f,\bar g$, it follows that there exists a unique classical solution $\bar v_\eps\in C^{1,2}([0,T]\times\R^{\hat d})$ of equation \eqref{HJB_fin_dim} (see for instance \cite[Theorem 14.15]{Lieberman} and, in particular, the comments after Theorem 14.15 concerning the case when the operators ``$L_\nu$'' are linear). Moreover, by \cite[Theorems 4.1.1 and 4.7.4]{Krylov80} we have that item 3) holds. Furthermore, by \cite[Theorem 4.6.2]{Krylov80} there exists some constant $\hat C\geq0$, depending only on $\hat K$, such that
\begin{equation}\label{v_esp-v}
|\bar v_\eps(t,y) - \bar v(t,y)| \ \leq \ \eps\,\hat C\,\textup{e}^{\hat C(T-t)},
\end{equation}
for every $(t,y)\in[0,T]\times\R^{d\hat d}$, where $\bar v\colon[0,T]\times\R^{d\hat d}\rightarrow\R$ is defined as
\[
\bar v(t,y) \ = \ \sup_{\alpha\in\Ac} \E\bigg[\int_t^T \bar f(s,Y_s^{t,y,\alpha},\alpha_s)\,ds + \bar g(Y_T^{t,y,\alpha})\bigg],
\]
with $Y^{t,y,\alpha}=(Y_s^{t,y,\alpha})_{s\in[t,T]}$ solving the following system of controlled stochastic differential equations:
\[
\begin{cases}
dY_s^{t,y,\alpha} \ = \ \bar b_\phi(s,Y_s^{t,y,\alpha},\alpha_s)\,ds + \bar\sigma_\phi(s,Y_s^{t,y,\alpha},\alpha_s)\,dB_s, \qquad &\quad s\in(t,T], \\
Y_t^{t,y,\alpha} \ = \ y.
\end{cases}
\]
\emph{Proof of item 1).} For every $(t,x)\in[0,T]\times C([0,T];\R^d)$, consider $y^{t,x}\in\R^{d\hat d}$ given by \eqref{y^t,x}. Then, proceeding as in the proof of \cite[Theorem 3.15]{cosso_russoStrict} (see, in particular, equalities (3.16)), we obtain
\begin{align}\label{EqualitiesYX}
Y_r^{t,y^{t,x},\alpha} &= \bigg(\int_{[0,t]}\phi_1(s)\,d^-x(s)+\int_t^r\phi_1(s)\,dX_s^{t,x,\alpha},\ldots,\int_{[0,t]}\phi_{\hat d}(s)\,d^-x(s)+\int_t^r\phi_{\hat d}(s)\,dX_s^{t,x,\alpha}\bigg) \notag \\
&= \bigg(\int_{[0,r]}\phi_1(s)\,d^-X_s^{t,x,\alpha},\ldots,\int_{[0,r]}\phi_{\hat d}(s)\,d^-X_s^{t,x,\alpha}\bigg),
\end{align}
for all $r\in[t,T]$, $\P$-a.s., where, for each $i=1,\ldots,d$, $\int_{[0,r]}\phi_i(s)\,d^-X_s^{t,x,\alpha}$ is intended $\P$-a.s. as a deterministic forward integral. From \eqref{EqualitiesYX} we get
\begin{align*}
\bar v\bigg(t,\int_{[0,t]}\phi_1(s)\,d^-x(s),\ldots,\int_{[0,t]}\phi_{\hat d}(s)\,d^-x(s)\bigg) = \sup_{\alpha\in\Ac} \E\bigg[\int_t^T \bar f(s,Y_s^{t,y^{t,x},\alpha},\alpha_s)ds + \bar g(Y_T^{t,y^{t,x},\alpha})\bigg] \\
= \sup_{\alpha\in\Ac} \E\bigg[\int_t^T f(s,X^{t,x,\alpha},\alpha_s)ds + g(X^{t,x,\alpha})\bigg] = v(t,x),
\end{align*}
where $v$ is the value function defined in \eqref{Value}.\\
Now, let $v_\eps\colon[0,T]\times C([0,T];\R^d)\rightarrow\R$ be defined as
\begin{equation}\label{Value_eps}
v_\eps(t,x) \ := \ \bar v_\eps\bigg(t,\int_{[0,t]}\phi_1(s)\,d^-x(s),\ldots,\int_{[0,t]}\phi_{\hat d}(s)\,d^-x(s)\bigg),
\end{equation}
for every $(t,x)\in[0,T]\times C([0,T];\R^d)$. Then, by direct calculations (proceeding as in the proof of \cite[Lemma D.1]{CR19}), we deduce that $v_\eps\in C_{\textup{pol}}^{1,2}([0,T]\times C([0,T];\R^d))$ and that the following equalities hold:
\begin{align}
\partial_t^H v_\eps(t,x) \ &= \ \partial_t \bar v_\eps\bigg(t,\int_{[0,t]}\phi_1(s)\,d^-x(s),\ldots,\int_{[0,t]}\phi_{\hat d}(s)\,d^-x(s)\bigg), \label{EqualitiesDerivH} \\
\partial_{x_i}^V v_\eps(t,x) \ &= \ \bigg\langle\partial_y \bar v_\eps\bigg(t,\int_{[0,t]}\phi_1(s)\,d^-x(s),\ldots,\int_{[0,t]}\phi_{\hat d}(s)\,d^-x(s)\bigg),\boldsymbol\phi_i(t)\bigg\rangle, \label{EqualitiesDerivV} \\
\partial_{xx}^V v_\eps(t,x) \ &= \ \boldsymbol\phi\trans(t)\,\partial_{yy}^2\bar v_\eps\bigg(t,\int_{[0,t]}\phi_1(s)\,d^-x(s),\ldots,\int_{[0,t]}\phi_{\hat d}(s)\,d^-x(s)\bigg)\boldsymbol\phi(t), \label{EqualitiesDerivVV}
\end{align}
for every $i=1,\ldots,d$, where $\boldsymbol\phi_i(t)$ denotes the $i$-th column of the matrix $\boldsymbol\phi(t)$.

\vspace{1mm}

\noindent\emph{Proof of item 2).} Since $\bar v_\eps$ is a classical solution of equation \eqref{HJB_fin_dim}, it follows from equalities \eqref{EqualitiesDerivH}-\eqref{EqualitiesDerivV}-\eqref{EqualitiesDerivVV} that $v_\eps$ is a classical solution of equation \eqref{HJB_eps}.

\vspace{1mm}

\noindent\emph{Proof of item 5).} From \eqref{v_esp-v}, we have
\[
|v_\eps(t,x) - v(t,x)| \ = \ |\bar v_\eps(t,y^{t,x}) - \bar v(t,y^{t,x})| \ \leq \ \eps\,\hat C\,\textup{e}^{\hat C(T-t)},
\]
for every $(t,x)\in[0,T]\times C([0,T];\R^d)$. This shows the validity of item 5).

\vspace{1mm}

\noindent\emph{Proof of item 4).} Following \cite[Section 6 of Chapter 4]{Krylov80}, we now formulate a stochastic optimal control problem with value function $\bar v_\eps$. To simplify notation we still consider the same probability space $(\Omega,\Fc,\P)$, on which we suppose that another Brownian motion $W=(W_t)_{t\geq0}$, $d$-dimensional and independent of $B$, is defined. For every $t\in[0,T]$, we denote by $\hat\F^t=(\hat\Fc_s^t)_{s\geq0}$ the $\P$-completion of the filtration generated by $(B_{s\vee t}-B_t)_{s\geq0}$ and $(W_{s\vee t}-W_t)_{s\geq0}$. We also denote by $\hat\Ac_t$ the family of all $\hat\F^t$-progressively measurable processes $\hat\alpha\colon[0,T]\times\Omega\rightarrow A$. Then, $\bar v_\eps$ admits the following stochastic control representation:
\[
\bar v_\eps(t,y) \ = \ \sup_{\hat\alpha\in\hat\Ac_t} \E\bigg[\int_t^T \bar f(s,Y_s^{\eps,t,y,\hat\alpha},\hat\alpha_s)\,ds + \bar g(Y_T^{\eps,t,y,\hat\alpha})\bigg],
\]
with $Y^{\eps,t,y,\hat\alpha}=(Y_s^{\eps,t,y,\hat\alpha})_{s\in[t,T]}$ solving the following system of controlled stochastic differential equations:
\[
\begin{cases}
dY_s^{\eps,t,y,\hat\alpha} = \bar b_\phi(s,Y_s^{\eps,t,y,\hat\alpha},\hat\alpha_s)\,ds + \bar\sigma_\phi(s,Y_s^{\eps,t,y,\hat\alpha},\hat\alpha_s)\,dB_s + \eps\,\boldsymbol\phi(s)\,dW_s, \qquad & s\in(t,T], \\
Y_t^{\eps,t,y,\hat\alpha} = y,
\end{cases}
\]
with $\boldsymbol\phi$ as in \eqref{bold_phi} and $\bar b_\phi,\bar\sigma_\phi$ as in \eqref{b_sigma_phi}, respectively. Now, given $t\in[0,T]$, $x\in C([0,T];\R^d)$, $\hat\alpha\in\hat\Ac_t$, $\eps\in(0,1)$, consider the solution $X^{\eps,t,x,\hat\alpha}$ to the following system of controlled stochastic differential equations:
\[
\begin{cases}
dX_s^{\eps,t,x,\hat\alpha} = b(s,X^{\eps,t,x,\hat\alpha},\hat\alpha_s)\,ds + \sigma(s,X^{\eps,t,x,\hat\alpha},\hat\alpha_s)\,dB_s + \eps\,dW_s, \qquad &\quad s\in(t,T], \\
X_s^{\eps,t,x,\hat\alpha} = x(s), &\quad s\in[0,t].
\end{cases}
\]
From similar calculations as in \eqref{EqualitiesYX}, we deduce that (recall that $y^{t,x}$ is given by \eqref{y^t,x})
\[
Y_r^{\eps,t,y^{t,x},\hat\alpha} \ = \ \bigg(\int_{[0,r]}\phi_1(s)\,d^-X_s^{\eps,t,x,\hat\alpha},\ldots,\int_{[0,r]}\phi_{\hat d}(s)\,d^-X_s^{\eps,t,x,\hat\alpha}\bigg),
\]
for all $r\in[t,T]$, $\P$-a.s.. Then
\begin{align*}
\bar v_\eps(t,y^{t,x}) \ &= \ \sup_{\hat\alpha\in\hat\Ac_t} \E\bigg[\int_t^T \bar f(s,Y_s^{\eps,t,y^{t,x},\hat\alpha},\hat\alpha_s)ds + \bar g(Y_T^{\eps,t,y^{t,x},\hat\alpha})\bigg] \\
&= \ \sup_{\hat\alpha\in\hat\Ac_t} \E\bigg[\int_t^T f(s,X^{\eps,t,x,\hat\alpha},\hat\alpha_s)ds + g(X^{\eps,t,x,\hat\alpha})\bigg].
\end{align*}
Recalling \eqref{Value_eps} and the definition of $y^{t,x}$ in \eqref{y^t,x}, we get
\begin{equation}\label{Value_eps_bis}
v_\eps(t,x) \ = \ \sup_{\hat\alpha\in\hat\Ac_t} \E\bigg[\int_t^T f(s,X^{\eps,t,x,\hat\alpha},\hat\alpha_s)ds + g(X^{\eps,t,x,\hat\alpha})\bigg],
\end{equation}
for every $(t,x)\in[0,T]\times C([0,T];\R^d)$. Proceeding along the same lines as for the proof of \eqref{Value_eps_bis}, we obtain
\[
\hat v_\eps(t,\hat x) \ = \ \sup_{\hat\alpha\in\hat\Ac_t} \E\bigg[\int_t^T f(s,X^{\eps,t,\hat x,\hat\alpha},\hat\alpha_s)ds + g(X^{\eps,t,\hat x,\hat\alpha})\bigg],
\]
for every $(t,\hat x)\in[0,T]\times D([0,T];\R^d)$. Then, from the Lipschitz property of $f$ and $g$, we derive the following Lipschitz property of $\hat v_\eps$
\[
|\hat v_\eps(t,\hat x) - \hat v_\eps(t,\hat x')| \ \leq \ \bar L'\,\|\hat x - \hat x'\|_t,
\]
for all $t\in[0,T]$, $\hat x,\hat x'\in D([0,T];\R^d)$, for some constant $\bar L'$, depending only on $T$ and $K$. As a consequence, from the definition of vertical derivative of $v_\eps$, we see that item 4) holds.
\end{proof}

\noindent We can now state the following result, which plays a crucial role in the proof the comparison theorem (Theorem \ref{T:Comparison}), in order to show that $u_1\leq v$.

\begin{Theorem}\label{T:CylindrApprox1}
Let Assumptions \ref{AssA}, \ref{AssB}, \ref{AssC} hold. Consider the sequences $\{b_n\}_n$, $\{f_n\}_n$, $\{g_n\}_n$, $\{v_n\}_n$ in \eqref{Coeff_n}-\eqref{Value_n} $($recall from Lemma \ref{L:CylindrApprox1} that $d_n$ and $\phi_{n,1},\ldots,\phi_{n,d_n}$ are the same for $b$, $f$, $g$$)$. We also assume, without loss of generality, that in Assumption \ref{AssC}\textup{-(i)}, $\bar d=d_n$ and that the functions $\varphi_1,\ldots,\varphi_{\bar d}$ coincide with $\phi_{n,1},\ldots,\phi_{n,d_n}$.\\
Then, for every $n$ and any $\eps\in(0,1)$, there exist $v_{n,\eps}\colon[0,T]\times C([0,T];\R^d)\rightarrow\R$ and $\bar v_{n,\eps}\colon[0,T]\times\R^{dd_n}\rightarrow\R$, with
\[
v_{n,\eps}(t,x) \ = \ \bar v_{n,\eps}\bigg(t,\int_{[0,t]}\phi_{n,1}(s)\,d^-x(s),\ldots,\int_{[0,t]}\phi_{n,d_n}(s)\,d^-x(s)\bigg),
\]
for all $(t,x)\in[0,T]\times C([0,T];\R^d)$, such that the following holds.
\begin{enumerate}[\upshape 1)]
\item $v_{n,\eps}\in C_{\textup{pol}}^{1,2}([0,T]\times C([0,T];\R^d))$ and $\bar v_{n,\eps}\in C^{1,2}([0,T]\times\R^{dd_n})$.
\item $v_{n,\eps}$ is a classical solution of equation \eqref{HJB_eps} with $b,f,g,\bar v_\eps,y^{t,x}$ replaced respectively by $b_n,f_n,g_n,\bar v_{n,\eps},y_n^{t,x}$, where $y_n^{t,x}$ is given by \eqref{y_n^t,x}.
\item There exists a constant $\bar C_n\geq0$, independent of $\eps$, such that
\begin{equation}
- \bar C_n\,\textup{e}^{\bar C_n(T-t)}\,\big(1 + |y|\big)^{3q} \ \leq \ \partial_{y_iy_j}^2 \bar v_{n,\eps}(t,y) \ \leq \ \frac{1}{\eps^2}\bar C_n\,\textup{e}^{\bar C_n(T-t)}\,\big(1 + |y|\big)^{3q}, \label{EstimateKrylov}
\end{equation}
for all $(t,y)\in[0,T]\times\R^{d\hat d}$ and every $i,j=1,\ldots,dd_n$, with $q$ as in item (iii) of Assumption \ref{AssC}.
\item There exists a constant $\bar L\geq0$, depending only on $T$ and $K$, such that
\begin{equation}\label{EstimateLipschitz1}
\big|\partial_x^V v_{n,\eps}(t,x)\big| \ \leq \ \bar L,
\end{equation}
for every $(t,x)\in[0,T]\times C([0,T];\R^d)$.
\item There exists a constant $\bar c\geq0$, depending only on $K$ and $T$, such that
\[
|v_{n,\eps}(t,x) - v_{n,\eps}(t',x')| \ \leq \ \bar c\big(|t-t'|^{1/2} + \|x(\cdot\wedge t) - x'(\cdot\wedge t')\|_T\big),
\]
for all $(t,x),(t',x')\in[0,T]\times C([0,T];\R^d)$.
\item For every $n$, $v_{n,\eps}$ converges pointwise to $v_n$ in \eqref{Value_n} as $\eps\rightarrow0^+$.
\item $v_n$ converges pointwise to $v$ in \eqref{Value} as $n\rightarrow+\infty$.
\end{enumerate}
\end{Theorem}
\begin{proof}[\textbf{Proof.}]
Items 1)-2)-3)-4)-6) follow directly from Lemma \ref{L:CylindrApprox3} with $b,f,g$ replaced respectively by $b_n,f_n,g_n$. Moreover, item 5) follows from \eqref{Value_Lipschitz}. Finally, item 7) follows from Lemma \ref{L:CylindrApprox2}.
\end{proof}

\noindent We end this section with the next result, which plays a fundamental role in the proof of the comparison theorem (Theorem \ref{T:Comparison}), in order to show that $v\leq u_2$.

\begin{Theorem}\label{T:CylindrApprox2}
Let Assumptions \ref{AssA}, \ref{AssB}, \ref{AssC} hold. For every $s_0\in[0,T]$, consider the sequences $\{b_n\}_n$, $\{f_n\}_n$, $\{v_n(s_0,\cdot)\}_n$ obtained applying Lemma \ref{L:CylindrApprox1} to $b$, $f$, $v(s_0,\cdot)$ $($recall from Lemma \ref{L:CylindrApprox1} that $d_n$ and $\phi_{n,1},\ldots,\phi_{n,d_n}$ are the same for $b$, $f$, $v_n(s_0,\cdot)$$)$. We also assume, without loss of generality, that in Assumption \ref{AssC}\textup{-(i)}, $\bar d=d_n$ and that the functions $\varphi_1,\ldots,\varphi_{\bar d}$ coincide with $\phi_{n,1},\ldots,\phi_{n,d_n}$.\\
For every $(s_0,a_0)\in[0,T]\times A$, let $v^{s_0,a_0}\colon[0,s_0]\times C([0,T];\R^d)\rightarrow\R$ be given by
\[
v^{s_0,a_0}(t,x) \ = \ \E\bigg[\int_t^{s_0} f\big(r,X^{t,x,a_0},\alpha_r\big)\,dr + v\big(s_0,X^{t,x,a_0}\big)\bigg], \quad \forall\,(t,x)\in[0,s_0]\times C([0,T];\R^d),
\]
where $X^{t,x,a_0}$ corresponds to the process $X^{t,x,\alpha}$ with $\alpha\equiv a_0$. Similarly, for every $n\in\N$, let $v_n^{s_0,a_0}\colon[0,s_0]\times C([0,T];\R^d)\rightarrow\R$ be given by
\[
v_n^{s_0,a_0}(t,x) \ = \ \E\bigg[\int_t^{s_0} f_n\big(r,X^{t,x,a_0},\alpha_r\big)\,dr + \hat v_n\big(s_0,X^{t,x,a_0}\big)\bigg], \quad \forall\,(t,x)\in[0,s_0]\times C([0,T];\R^d),
\]
where the sequence $\{\hat v_n\}$ is defined as in Lemma \ref{L:CylindrApprox1} starting from the function $v$.\\
Then, for every $n$, there exists $\bar v_n^{s_0,a_0}\colon[0,s_0]\times\R^{dd_n}\rightarrow\R$, with
\[
v_n^{s_0,a_0}(t,x) \ = \ \bar v_n^{s_0,a_0}\bigg(t,\int_{[0,t]}\phi_{n,1}(s)\,d^-x(s),\ldots,\int_{[0,t]}\phi_{n,d_n}(s)\,d^-x(s)\bigg),
\]
for all $(t,x)\in[0,s_0]\times C([0,T];\R^d)$, such that the following holds.
\begin{enumerate}[\upshape 1)]
\item $v_n^{s_0,a_0}\in C_{\textup{pol}}^{1,2}([0,s_0]\times C([0,T];\R^d))$ and $\bar v_n^{s_0,a_0}\in C^{1,2}([0,s_0]\times\R^{dd_n})$.
\item $v_n^{s_0,a_0}$ is a classical solution of the following equation:
\begin{equation}\label{HJB_n^s_0,a_0}
\hspace{-.4cm}\begin{cases}
\vspace{2mm}
\partial_t^H v_n^{s_0,a_0}(t,x) + f_n(t,x,a_0) + \big\langle b_n(t,x,a_0),\partial^V_x v_n^{s_0,a_0}(t,x)\big\rangle \\
\vspace{2mm}+\,\dfrac{1}{2}\textup{tr}\big[(\sigma\sigma\trans)(t,x,a_0)\partial^V_{xx}v_n^{s_0,a_0}(t,x)\big] = 0, &\hspace{-2cm}(t,x)\in[0,s_0)\times C([0,T];\R^d), \\
v_n^{s_0,a_0}(s_0,x) = \hat v_n(s_0,x), &\,\hspace{-2cm}x\in C([0,T];\R^d),
\end{cases}
\end{equation}
where $y_n^{t,x}$ is given by \eqref{y_n^t,x}.
\item There exists a constant $\hat L\geq0$, depending only on $T$ and $K$, such that
\[
\big|\partial_x^V v_n^{s_0,a_0}(t,x)\big| \ \leq \ \hat L,
\]
for every $(t,x)\in[0,s_0]\times C([0,T];\R^d)$.
\item There exists a constant $\hat c\geq0$, depending only on $K$ and $T$, such that
\[
|v_n^{s_0,a_0}(t,x) - v_n^{s_0,a_0}(t',x')| \ \leq \ \hat c\big(|t-t'|^{1/2} + \|x(\cdot\wedge t) - x'(\cdot\wedge t')\|_T\big),
\]
for all $(t,x),(t',x')\in[0,s_0]\times C([0,T];\R^d)$.
\item $v_n^{s_0,a_0}$ converges pointwise to $v^{s_0,a_0}$ as $n\rightarrow+\infty$.
\end{enumerate}
\end{Theorem}
\begin{proof}[\textbf{Proof.}]
Items 1)-2)-3) follow from the same arguments as in \cite[Theorem 3.5]{cosso_russoStrict}, which indeed goes along the same lines as in the proof of items 1)-2)-4) of Lemma \ref{L:CylindrApprox3}, relying on regularity results for linear (rather than fully nonlinear as in Lemma \ref{L:CylindrApprox3}) parabolic equations as in particular \cite[Theorem 6.1, Chapter 5]{friedman75vol1}. Moreover, item 5) follows from \eqref{Value_Lipschitz}. Finally, item 4) follows from Lemma \ref{L:CylindrApprox2} with $g_n,T,A$ replaced respectively by $\hat v_n(s_0,\cdot),s_0,\{a_0\}$.
\end{proof}

\end{appendices}

\small

\bibliographystyle{plain}
\bibliography{references}

\end{document}